\documentclass[11pt,a4paper]{article}
\usepackage{epsfig}
\usepackage[numbers]{natbib}
\usepackage{amsmath, amsthm, amssymb}
\usepackage{epsfig}
\usepackage[dvipdfm]{hyperref}
\usepackage[margin=3cm, top=3.5cm, bottom=3.5cm]{geometry}

\newcommand\blfootnote[1]{%
  \begingroup
  \renewcommand\thefootnote{}\footnote{#1}%
  \addtocounter{footnote}{-1}%
  \endgroup
}

\newtheorem{theorem}{Theorem}[section]
\newtheorem{lemma}{Lemma}[section]
\newtheorem{proposition}{Proposition}[section]
\newtheorem{corollary}{Corollary}[section]
\newtheorem{definition}{Definition}[section]
\newtheorem{remark}{Remark}[section]
\newtheorem{example}{Example}[section]

\numberwithin{equation}{section}
\everymath{\displaystyle}
\allowdisplaybreaks

\title{Hardy inequalities with double singular weights }
\date{}

\author{
Nikolai Kutev\thanks{Institute of Mathematics and Informatics, Bulgarian Academe of Sciences, 1113, Sofia, Bulgaria}
 \and Tsviatko Rangelov
 \footnotemark[1]
}

\begin{document}

\maketitle
\blfootnote{Corresponding author: T. Rangelov, rangelov@math.bas.bg}

\tableofcontents
\newpage
\begin{abstract}
\noindent

The aim of this paper is  to obtain  new Hardy inequalities with double singular weights -- at an interior  point and on the boundary of the domain. These inequalities give us the possibility to derive estimates from below of the first eigenvalue of the p-Laplacian with Dirichlet boundary conditions.
\end{abstract}

\vspace{2pt}

\noindent
{\bf Keywords:} Hardy inequality; Double singular weights; Sharp estimates; p-Laplacian; First eigenvalue.

\vspace{2pt}
\noindent
{\bf 2010 Mathematics Subject Classification:} 26D10, 35P15

\section{Introduction}
\label{sect1}

The paper is devoted to the classical Hardy inequality, its generalizations and applications for  estimates from below of the first eigenvalue $\lambda_{p,n}(\Omega)$ of the $\quad$ p-Laplacian, $p>1$, in a bounded domain $\Omega\subset R^n$, $n\geq2$. Only the multidimensional case $n\geq2$ is considered because for $n=1$ there exist detailed literature and satisfactory results, see for instance \citet{Ha19, Ha20, Ha25, Ne62, Ma85, OK90, HHL02}. Unlike the one-dimensional case, the theory for $n\geq2$ is far from being completely solved.

The paper can be regarded as a work in the series of works of  \citet{Ma85, OK90, GM13, BEL15, KPS17}. We focus on the optimality of  the Hardy constant and on the sharpness of  Hardy inequality. Further on in the paper we say that the Hardy constant is optimal if for a greater one the corresponding Hardy inequality fails for all functions of the admissible class.  Sharpness of  the Hardy inequality means that an equality is achieved for some admissible function. For Hardy inequality with singular weights at an interior point of $\Omega$, or on the boundary $\partial\Omega$, we always prove the optimality of  Hardy constant. As for the sharpness of the Hardy inequality, we show that an equality is achieved only for Hardy inequality with additional `nonlinear' term. It is well-known that for inequalities with optimal constant and additional `linear'  term, an equality is not achieved. Only in the case when  Hardy constant is greater than the optimal one, the sharpness of the inequality is proved by variational technique, see for example  \citet{PT05} and the references therein. In fact, in this way the optimality of  Hardy constant is shown.

In the literature mainly inequalities  with singular weights at a point, or on the boundary, $\partial\Omega$, or on some $k$-dimensional manifold, $1\leq k\leq n-1$, have been  studied.
The subject of the investigations in the paper is Hardy type inequalities in bounded domains $\Omega\in R^n$, $n\geq2$ with double singular weights: in an interior point of $\Omega$ and on the boundary $\partial\Omega$.

Our aim is to derive new Hardy inequalities, which are with an optimal constant and with suitable additional terms that become sharp. For example,  Hardy constant is optimal for convex and star-shaped domains and the inequality is sharp as a result of a `nonlinear' additional term.

The  background of the  theory of Hardy inequalities is mathematical and functional analysis and differential equations.  Among many different applications we choose one -- the estimate of the first eigenvalue of the p-Laplacian from below, which motivates the study of Hardy inequalities with double singular weights.

There are estimates for $\lambda_{p,n}(\Omega)$ by means of the Cheeger's constant, \citet{Ch70, LW97, KF03}, by the Picone's identity \citet{BD12, BD13}, with the Sobolev inequality \citet{Ma85, LXZ11}, with estimates in parallelepiped \citet{Li95} and others. However,  Hardy inequality with  double singular weights  allows one to get better analytical estimates for $\lambda_{p,n}(\Omega)$. For completeness, we prove estimates from below for $\lambda_{p,n}(\Omega)$ as well as by the well-known Hardy inequalities with singular weights only at a point or on the boundary $\partial\Omega$. The comparison of the results definitely shows that the inequalities with double singular weights produce better analytical estimates for $\lambda_{p,n}(\Omega)$ in comparison with those obtained from other Hardy inequalities or other methods.

In the rest of this section, without claims of completeness,  we present results of Hardy inequalities which from our point of view are decisive for the  development of the subject in the recent years, as the books of \citet{Ma85, OK90, GM13, BEL15, KPS17},  and the papers of \citet{BM97, BV97, Da98, BFT03a, BFT03b, HHL02, Ti04, EH05, KK10, Le14} and others.  In section \ref{sect3} we derive a new Hardy inequality with weights  in abstract form.
Particular cases  are presented to demonstrate the applicability of
the method and to show  some generalizations of  the existing results.
In section \ref{sect4} we prove a general  Hardy  inequality with singular weights at zero and on
the boundary $\partial\Omega$ of  star-shaped domains  and an optimal Hardy constant. In section \ref{sect5} we propose Hardy  inequality with weight singular at $0\in\Omega$
in the class of functions which are not zero on the boundary
$\partial\Omega$. The Hardy constant is optimal and the inequality
is sharp due to the additional boundary  term. In section \ref{sect6} we derive improved Hardy inequality with double singular weights in bounded, star-shaped domains $\Omega\subset R^n$, $n\geq2$ where the     singularity is at an interior point and  on the boundary of the domain. The Hardy constant is optimal and the inequality is sharp due to the additional   term. In section \ref{sect7} we apply   Hardy inequalities from the previous sections and we derive estimates from below of the first eigenvalue of the p-Laplacian.

\subsection{Preliminary remarks on Hardy inequalities}
\label{sect2}
In this section classical Hardy inequalities are shown together with some definitions used in the paper.  The analysis of the existing in the literature results is far from its completeness. The aim of this section is to recall the well known  results in the frame of the study in the paper.

The  classical Hardy inequality  in $R^1_+=(0,\infty)$ states
\begin{equation}
\label{2eq1} \displaystyle\int_0^\infty |u'(x)|^px^\alpha dx\geq
\left(\frac{p-1-\alpha}{p}\right)^p\int_0^\infty
x^{-p+\alpha}|u(x)|^pdx,
\end{equation}
where $1<p<\infty$, $\alpha<p-1$ and $u(x)$ is an absolutely continuous function on
$[0,\infty)$ with $u(0)=0$, see \citet{Ha20, Ha25}  for $\alpha=0$ and \citet{HPL52}, Sect. 9.8 for $\alpha<p-1$.

The constant $\left(\frac{p-1-\alpha}{p}\right)^p$ is the best possible one, i.e., it can not be replaced with a  greater one, but the equality  in (\ref{2eq1}) is not achieved.

In the last 20 years the generalizations of  (\ref{2eq1}) for the multidimensional case are mainly oriented in two directions with respect to the structure of the singular   weight:
\begin{itemize}
\item singularities on the boundary:  when the prototype inequality is
\begin{equation}
\label{2eq2} \displaystyle\int_\Omega |\nabla u(x)|^pd^\alpha(x)
dx\geq C_\Omega\int_\Omega d^{-p+\alpha}(x)|u(x)|^pdx,
\end{equation}
 with $d(x)=\hbox{dist}(x,\partial\Omega)$ and $\alpha<p-1$, $p>1$, $n\geq2$,  $p\neq n$, $u\in
W^{1,p}_0(\Omega)$;
\item singularity at a point:  when the inequality
\begin{equation}
\label{2eq3} \displaystyle\int_\Omega |\nabla u(x)|^p dx\geq
C_{p,n}\int_\Omega \frac{|u(x)|^p}{|x|^p}dx,
\end{equation}
holds, $\Omega\subset R^n$,
$0\in\Omega$, $n\geq2$ and $p>1$, $p\neq n$, $u\in
W^{1,p}_0(\Omega)$ for $p<n$ and $u\in
W^{1,p}_0(\Omega\backslash\{0\})$  for $p>n$. The constant $C_{p,n}=\left|\frac{n-p}{p}\right|^p$ is optimal one.
\end{itemize}
We will present  some of the results on Hardy inequalities underlining their optimality and sharpness, see Definition \ref{2def1}

Next we recall several useful definitions and notions.
\begin{definition}
\label{2def1}
The constants $C_{\Omega}$ in  Hardy inequality
\begin{equation}
\label{2equ3}
\int_\Omega V(x)|\nabla u|^pdx\geq C_\Omega\int_\Omega W(x)|u|^pdx+A(u)
\end{equation}
with positive weights $V(x), W(x)$ and additional nonnegative term $A(u)$ is optimal if for every $\varepsilon>0$ there exists $u_\varepsilon$ from the admissible class of functions for which the inverse Hardy inequality holds if we replace $C_\Omega$ with $C_\Omega+\varepsilon$. The Hardy inequality (\ref{2equ3}) is sharp if there exists a function from the admissible class of functions for which (\ref{2eq2}) is an equality and  both sides of (\ref{2eq2}) are finite.
\end{definition}

Let us recall that   sharp Hardy inequalities are proved by means of variational technique, see \citet{ PT05} and the references therein.

\subsubsection*{Inequalities with general weights}
\label{2sect2}

There are  generalizations of (\ref{2eq1}) for
the $n$-dimensional case, $n\geq2$, for bounded domains and for
different weights (integral kernels), see for more details
 \citet{Da98, OK90, GM13, BEL15}.

For example, in \citet{OK90}, Theorems 14.1, 14.2, the following Hardy inequality with general weights is proposed
\begin{equation}
\label{2eq4}
\sum^n_{i=1}\int_\Omega v_i(x)\left|\frac{\partial u}{\partial x_i}\right|^pdx\geq \int_\Omega w(x)|u(x)|^pdx.
\end{equation}
Here $\Omega$ is a bounded domain in $R^n$, $n\geq2$, $u\in C^\infty_0(\Omega)$, $1<p<\infty$, functions $v_i$, $i=1,\ldots,n$ and $w$ are measurable, positive and finite for a. e. $x\in\Omega$. It is proved that under the existence of a solution $y(x)$ of the equation
\begin{equation}
\label{2eq400}
\sum^n_{i=1}\frac{\partial}{\partial x_i}\left[v_i\left|\frac{\partial y}{\partial x_i}\right|^{p-2}\frac{\partial y}{\partial x_i}\right]+w(x)|y|^{p-2}y=0 \ \ \hbox{ in } \Omega,
\end{equation}
where $y(x)\neq0$ and $\frac{\partial y}{\partial x_i}\neq0$ for a. e. $x\in\Omega$, and some regularity conditions on $y$, $v_i$ and $w$ inequality (\ref{2eq4}) hold.

A necessary and sufficient condition on $v_i$ and $w$ for the validity of (\ref{2eq4}) is proved in \citet{Ma85} in terms of capacities, see also \citet{OK90}, Theorem 16.3 and the discussion therein.

In \citet{GM13}, Theorem 4.1.1, Hardy inequality with ge\-ne\-ral positive weights $V(|x|)$ and $W(|x|)$ is proposed in the ball $B_R=\{x\in R^n, |x|<R\}$, $n\geq2$, i.e.,
\begin{equation}
\label{2eq5}
\int_{B_R}V(|x|)|\nabla u|^2dx\geq c\int_{B_R}W(|x|)u^2dx, \ \ u\in C_0^\infty(B_R).
\end{equation}
The necessary and sufficient condition for the validity of (\ref{2eq5}) given in the same book is that the couple $(V, W)$ forms  n-dimensional Bessel pair on the interval $(0, R)$, i.e., the equation
$$
y''(r)+\left(\frac{n-1}{r}+\frac{V'(r)}{V(r)}\right)y'(r)+\frac{W(r)}{V(r)}y(r)=0,
$$
has a positive solution in $(0,R)$.

Another Hardy inequality is proved in \citet{SC06}, Lemma 1.1,
\begin{equation}
\label{2eq6}
\int_{\Omega}\phi(r)|\nabla
u|^2dx\geq\int_{\Omega}\phi(r)\left|\frac{h'(r)}{h(r)}\right|^2u^2dx, \ \  r=|x|,
\end{equation}
in a bounded domain $\Omega\subset R^n$, $0\in\Omega\subset B_R$, $R>\sup_{x\in\Omega}|x|$. Here $\phi(r)\in C^1(0,R)$, $\phi(r)>0$ and $h\in C^1(0,R)$ is a positive solution of the equation
\begin{equation}
\label{2eq7} r^{n-1}\phi(r)(h^2)'(r)=c=\hbox{const},
\end{equation}
and $u\in C^\infty_0(\Omega)$ if $h^{-1}(0)=0$ while $u\in C^\infty_0(\Omega\backslash\{0\})$ if $h^{-1}(0)\neq0$.

In \citet{DH98}, see also \citet{BEL15}, Theorem 1.2.8,  the following Hardy inequality is shown
$$
\int_\Omega\frac{|\nabla V|^p}{|\Delta V|^{p-1}}|\nabla u|^pdx\geq\left(\frac{1}{p}\right)^p\int_\Omega|\Delta V||u|^pdx,
$$
for all $u\in C^\infty_0(\Omega)$ and any domain $\Omega\subset R^n$, $n\geq2$. The real valued weight $V(x)$ has derivatives up to order 2 in $L^1_{loc}(\Omega)$ and $\Delta V$ is of one sign for a.e. $x\in\Omega$.

\subsubsection*{Inequalities with  weights singular on the boundary}
\label{2sect2}
Another  research direction on Hardy inequalities
concerns the study of the geometric properties of  the domain $\Omega\subset R^n$,
on which the   Hardy's inequality (\ref{2eq2})
holds.

The inequality (\ref{2eq2}) was proved by \citet{Ne62}
for bounded domains $\Omega$ with Lipschitz boundary
$\partial\Omega$ and $u\in C^\infty_0(\Omega)$. Next generalizations
of (\ref{2eq2}) were made by \citet{Ku85} for H\"older $\partial\Omega$ and
by \citet{Wa94} for $\partial\Omega$ under generalized H\"older conditions.
Detailed description of these results can be found in \citet{OK90,
Ma85, Ha99}.

Another geometric condition on $\partial\Omega$ for the validity of (\ref{2eq2}) is suggested in \citet{EH05}. More precisely, if $R^n\backslash\Omega$ is $b$-plump for some $b$, then (\ref{2eq2}) holds for $\alpha\leq0$. If additionally $\Omega$ also satisfies the Witney cube counting condition then (\ref{2eq2}) is valid for $\alpha\in(0,\alpha_0)$, where $\alpha_0>0$ is given explicitly, see Theorem 3.1 in  \citet{EH05} for more details.

Further generalizations were made in \citet{An86, Le88,
Wa90, Ha99, Le08, Le14, KL09}. They are based on the investigation of the pointwise
Hardy inequalities with capacity methods, see the review of
\citet{KL09}. Note that in \cite{Ha99}  inequality
(\ref{2eq2}) is proved in the domain $\Omega_t=\{x\in\Omega, d(x)<t\}$ for $u\in
C_0^\infty(\Omega)$ and without zero conditions for $u$ on the set $\{x\in\Omega,
d(x)=t\}$. A more generally sufficient condition result is proved in \citet{KL09}.

Another way to describe the properties of $\Omega$ is to connect
the validity of  inequality (\ref{2eq2}) with the existence of
solutions of  certain boundary value problems for second order
elliptic equation with a singular weight. In \citet{An86} it is
proved that a necessary and sufficient condition for the validity of (\ref{2eq2})
when $p=2$ and $\alpha=0$ is the existence of a positive
super--harmonic function $v$ in $\Omega$ and a positive number
$\delta$ such that $\displaystyle\Delta
v+\frac{\delta}{d(x)^2}v\leq0$. Moreover, in
\cite{An86} it is shown that $\max\delta=C_\Omega$. For more general results for p-Laplace operator see \citet{AH98}.

For arbitrary $p>1$ and $\alpha=0$ in \citet{KK10}, Theorem 5.2, it is proved that (\ref{2eq2}) holds if and only if there exists a positive
constant $\lambda_p(\Omega)$ and a positive supersolution $\varphi\in W^{1,p}_0(\Omega)$ of the problem
$$
\hbox{div}\left(|\nabla \varphi|^{p-2}\nabla\varphi\right)+\lambda_p\frac{|\varphi|^{p-2}\varphi}{d^p(x)}\leq0 \ \ \hbox{in } \Omega, \varphi=0 \hbox{ on } \partial \Omega.
$$

The constant $C_\Omega=\left(\frac{p-1-\alpha}{p}\right)^p$ is optimal for (\ref{2eq2}) in  the case $n=1$ and for convex domains for $n\geq2$. For non-convex domains when $n>1$ the optimal constant $C_\Omega$ in (\ref{2eq2}) is unknown. There are only partial results, for example, when  $\Omega$ is non-convex and $p=2$ and $\alpha=0$, then \citet{An86} proved that $C_\Omega\geq\frac{1}{16}$ by means of the Koebe quarter theorem. Later on in \citet{LS08} better estimate for $C_\Omega$ using a stronger version of Koebe quarter theorem was obtained.

Moreover, in \citet{Av13} a sufficient condition on non-convex domain $\Omega$ is given so that  (\ref{2eq2}) with $p=2, \alpha=0$ holds with optimal  constant  $C_\Omega=\frac{1}{4}$ for $n\geq2$.

For arbitrary open domain $\Omega$ with $C^2$ smooth boundary $\partial\Omega$ with non-negative mean curvature $H(x)$ inequality (\ref{2eq2}) is proved in \citet{LLL12}, see Theorem 1.2,  for $\alpha=0$ with optimal constant $C_\Omega=\left(\frac{p-1}{p}\right)^p$. The curvature condition $H(x)\geq0$ is optimal because for arbitrary $\varepsilon>0$ and $H(x)\geq-\varepsilon$ on $\partial\Omega$, the  Hardy's inequality (\ref{2eq2}) with $C_\Omega=\left(\frac{p-1}{p}\right)^p$ fails for some $u\in W^{1,p}_0(\Omega)$.

Recently, in \citet{Zs17}, Theorem 2.1, an explicit estimate from below for $C_\Omega$ in (\ref{2eq2}) for star-shaped domains $\Omega$ was shown. This estimate coincides with $\left(\frac{p-1-\alpha}{p}\right)^p$ for convex domains $\Omega$.

For the annular domain $B_R\backslash B_r=\{0<r<|x|<R<\infty\}\subset R^n$ the following Hardy inequality is proved in \citet{AL10}, see Theorem 1 and Corollary 1,
$$\begin{array}{lll}
&&\int_{B_R\backslash B_r}|\nabla u|^2dx
\\[2pt]
\\
&&\geq\frac{1}{4}\int_{B_R\backslash B_r}\left(\frac{(n-1)(n-3)}{|x|^2}+\frac{1}{|x-r|^2}+\frac{1}{|x-R|^2}+\frac{2}{|x-r||x-R|}\right)u^2dx,
\end{array}
$$
for every $u\in H^1_0(B_R\backslash \bar{B_r})$. The weights on the right-hand side are singular on the boundary of $B_R\backslash \bar{B_r}$ and at the origin, which however does not belong to the domain $B_R\backslash \bar{B_r}$.

\subsubsection*{Inequalities with  weights singular at a point}
\label{2sect3}

Another generalization of (\ref{2eq1}) is an
inequality with a weight, singular at an interior point of $\Omega\subset R^n$, i.e., namely of type (\ref{2eq3}).
The optimal constant $\displaystyle
C_{2,n}=\left(\frac{n-2}{2}\right)^2$ is obtained in \citet{Le33} for $\Omega=R^3$, $p=2$ and in \citet{HPL52} for  $\Omega=R^n$, $n>3$, $p=2$,
 see also \citet{PV95} and \citet{OK90}.

The case $p=n$ in (\ref{2eq3}) is considered in \citet{II15}, see Theorem 1.1, where  Hardy's inequality
$$
\int_{B_1}\left|\langle\frac{x}{|x|},\nabla u\rangle\right|^ndx\geq\left(\frac{n-1}{n}\right)^n\int_{B_1}\frac{|u|^n}{|x|^n\left(\log\frac{1}{|x|}\right)^n}dx,
$$
is proved with optimal constant $\left(\frac{n-1}{n}\right)^n$ for every $u\in W^{1,n}_0(B_1)$.

For function $u\in C_0^\infty(\Omega)$, the constant $C_{p,n}$ in (\ref{2eq3}) is independent on $\Omega$. However, when  $u\in C^\infty(\Omega)$ the boundary term on $\partial\Omega$ is taken into account because $u$ does not necessary vanish on the boundary and   the geometry of $\Omega$ is important. In \citet{WZ03}, for $p=2$, the following Hardy inequality with weights and an additional boundary term was proposed.
$$
\int_{B_1}|x|^{-2\alpha}|\nabla u|^2dx\geq\left(\frac{n-2}{2}-\alpha\right)^2\int_{B_1}|x|^{-2(\alpha+1)}u^2dx-\frac{n-2\alpha-2}{2}\int_{\partial B_1}u^2dx,
$$
for $B_1\subset R^n$, $n\geq3$, $\alpha<\frac{n-2}{2}$.

In \citet{Ku85} Hardy inequality for functions which are not zero on the boundary (or part on the boundary) is considered. In this case a natural boundary term is added to the right-hand side.

\subsection{Hardy inequalities with an additional  term}
\label{2sec1-2}

When  Hardy's constant in (\ref{2eq2}) is optimal  there is no
non-trivial function of the admissible class of functions for which the Hardy
inequality becomes an equality.

 That is why in \citet{BM97}  the
question on the existence of an additional positive term $A(u)$ to the right-hand side of (\ref{2eq2}) is stated such that
the improved inequality
\begin{equation}
\label{2eq8} \int_\Omega |\nabla u(x)|^pd^\alpha(x) dx\geq
\displaystyle C_\Omega\int_\Omega
\displaystyle|u(x)|^pd^{\alpha-p}(x)dx+A(u),
\end{equation}
still holds in bounded domains for the optimal
constant $C_{\Omega}=\displaystyle\frac{1}{4}$ when $p=2$, $\alpha=0$, $n\geq2$.

For bounded convex domains \citet{BM97}, Theorem II, proved the following Hardy
inequality for all $u\in C^\infty_0(\Omega)$
\begin{equation}
\label{2eq9} \displaystyle\int_{\Omega}|\nabla u|^2dx
-\displaystyle\frac{1}{4}\int_{\Omega}\frac{u^2}{d^2(x)}
dx\geq\displaystyle b(\Omega)\int_{\Omega}u^2dx,
\end{equation}
where
$$
b(\Omega)\geq\displaystyle\frac{1}{4\hbox{diam}^2(\Omega)} \quad  \hbox{and }
\hbox{diam}(\Omega)=\max_{x,y\in\Omega}|x-y|.
$$

In \citet{HHL02},  Theorem 3.2,  for $p=2$, $n\geq2$, $C_\Omega=\frac{1}{4}$  and a convex domain $\Omega$,  inequality (\ref{2eq9}) is improved by showing that
\begin{equation}
\label{2eq10}
b(\Omega)\geq\frac{b_1}{4}|\Omega|^{-\frac{2}{n}}, \ \ b_1=\frac{n}{4}\left[\frac{2\pi^{\frac{2}{n}}}{n\Gamma(n/2)}\right]^{\frac{2}{n}},
\end{equation}
while in \citet{EL07},  Theorem 3.2, the estimate (\ref{2eq10})  is improved for bounded convex domain $\Omega\subset R^n$ to $b(\Omega)\geq\frac{3}{2}b_1|\Omega|^{-\frac{2}{n}}$.

Later on in \citet{Ti04},  Theorem 2.2, for $p>1$, $n\geq2$ and $\alpha=0$ an  optimal constant $C_\Omega=\left(\frac{p-1}{p}\right)^p$ is shown for  inequality (\ref{2eq8}), which  holds with $A(u)=b(\Omega)\int_\Omega|u|^pdx$ where
$$
b(\Omega)=\frac{(p-1)^{p+1}}{p^p}\left(\frac{\omega_n}{n|\Omega|}\right)^{p/n}
\frac{\sqrt{\pi}\Gamma\left(\frac{n+p}{2}\right)}{\Gamma\left(\frac{p+1}{2}\right)
\Gamma\left(\frac{n}{2}\right)}.
$$

In \citet{FMT06}, Theorem 1.1,  the authors  proved
that in a convex, bounded domain $\Omega\subset R^n$ inequality (\ref{2eq9}) holds with
$$
b(\Omega)\geq3(D_{int}(\Omega))^{-2}, \ \ D_{int}(\Omega)=2\sup_{x\in\Omega}d(x).
$$
Moreover, for $1<p<n$, $\alpha=0$, and $C_\Omega=\left(\frac{p-1}{p}\right)^p$, the estimate (\ref{2eq8}) holds with
$$
A(u)=b(\Omega)\int_\Omega|u|^pdx \ \  \hbox{  and  } C_1(p,n)D^{-p}_{int}\geq b(\Omega)\geq C_2(p,n)D^{-p}_{int},
$$
for some positive constants $C_1(p,n), C_2(p,n)$.

A different estimate for $b(\Omega)$  in (\ref{2eq9}) is given in \citet{AW07},  Theorem 1, where
$$
b(\Omega)=\frac{b_0^2}{D^2_{int}},
$$
and $b_0\approx 0.940\ldots$ is the Lamb constant defined as the first positive zero of the function $J_0(x)+2xJ'_0(x)$, where  $J_0$ is the Bessel function of order $0$.

Another additional term for (\ref{2eq8}) is proposed in \citet{BM97},  Theorem 5.1, for convex domains
$$
A(u)=\frac{1}{4}\int_\Omega\frac{u^2}{d^2(x)\left(1-\log(d(x)/L)\right)^2}dx, \ \ L=\hbox{diam }\Omega.
$$
For $p>1$, $p\neq n$ and in domain $\Omega$ satisfying some geometric condition, like the mean convexity of $\Omega$, the result of \citet{BM97} is generalized in \citet{BFT03b},  Theorem A,  to  Hardy inequality
\begin{equation}
\label{2eq11}\begin{array}{lll}
&&\int_\Omega|\nabla u|^pdx
\\[2pt]
\\
&&\geq\left(\frac{p-1}{p}\right)^p\int_\Omega\frac{|u|^p}{d^p(x)}dx+\frac{1}{2}\left(\frac{p-1}{p}\right)^{p-1}
\int_\Omega\frac{|u|^p}{d^p(x)}\left(\log(d(x)/D)\right)^{-2}dx,
\end{array}
\end{equation}
for $u\in W^{1,p}_0(\Omega)$ and $D\geq\sup_{x\in\Omega} d(x)$.

In \citet{FMT06}, Theorem 3.1, for a convex domain $\Omega\subset R^n$, $1<p< n$, $\alpha >-p$ and $C_\Omega=\left(\frac{p-1}{p}\right)^p$ the following Hardy inequality with additional term is proved
$$
\int_\Omega|\nabla u|^pdx\geq\left(\frac{p-1}{p}\right)^p\int_\Omega\frac{|u|^p}{d^p(x)}dx+C(p,n,\alpha)D^{-\alpha-p}_{int}\int_\Omega d^\alpha(x)|u|^pdx.
$$
The constant $C(p,n,\alpha)$ is independent of $\Omega$ and for $p=2, \alpha>-2$ the constant $C(2,n,\alpha)=C_\alpha$ is given explicitly: $C_\alpha=2^\alpha(\alpha+2)^2$ when $-2<\alpha<-1$,  while   $C_\alpha=2^\alpha(2\alpha+3)^2$ when $\alpha\geq-1$.

A different  generalization of \citet{HHL02, Ti04, BFT03b} is obtained in \citet{EH05, LLL12, Ps13} and in \citet{NT13}.

For example in \citet{LLL12},  Theorem 1.3, for an arbitrary domain $\Omega\subset R^{n+1}$ with $C^2$ smooth boundary $\partial\Omega$ with nonnegative mean curvature $H(x)\geq0$ the improved Hardy--Brezis--Marcus inequality
$$
\int_\Omega|\nabla u|^2dx\geq\frac{1}{4}\int_\Omega\frac{u^2}{d^2(x)}dx +b(n,\Omega)\int_\Omega u^2dx,
$$
is proved where $b(n,\Omega)
\geq\frac{2}{n}\left(\inf_{\partial\Omega}H(x)\right)^2$.

In \citet{EH05}, Theorem 5.1, for the domain $\Omega\subset R^n$ which is $b$-plump for some $b\in(0,1]$, the  Hardy inequality
$$
\int_\Omega|\nabla u|^pdx\geq c(p,n)b^n\left(\int_\Omega\frac{|u|^p}{d^p(x)}dx+|\Omega|^{-p/n}\int_\Omega|u|^pdx\right),
$$
is satisfied for every $u\in W^{1,p}_0(\Omega)$, $1<p<\infty$. The constant $c(p,n)$ is given expli\-citly.

For arbitrary $C^2$ smooth domains $\Omega\subset R^n$, it has been established in \citet{Ps13}, Theorem A, for the limiting case $p=1$ the inequality
\begin{equation}
\label{eqq11}
\int_\Omega\frac{|\nabla u|}{d^{s-1}(x)}dx\geq(s-1)\int_\Omega\frac{|u|}{d(x)}dx+B_1\int_\Omega\frac{|u|}{d(x)}dx,
\end{equation}
for $s\geq1$, $B_1\geq(n-1)H(x)$ and every $u\in C^\infty_0(\Omega)$, where $H(x)\geq0$ is the mean curvature of $\partial\Omega$. If $s\geq2$ then (\ref{eqq11}) is generalized with additional logarithmic term in Theorem C. For more details for $L^1$ Hardy inequalities with weights, i.e., $p=1$ in (\ref{2eq8}), see \citet{Pst11}.

Let us note the paper of \citet{FL12} where for inequality in a domain with generalized distance to the boundary a reminder term with the Sobolev-critical exponent is derived . In the particular case of a convex domain this settles a conjecture by \citet{FMT07}.

Analogously to the paper of \citet{BM97} in \citet{BV97} the
question about the existence of an additional positive term $A(u)$ such that the
improved inequality
\begin{equation}
\label{2eq14}\displaystyle\int_{\Omega}|\nabla u(x)|^2dx\geq
\left(\frac{n-2}{2}\right)^2\int_{\Omega}\frac{u^2(x)}{|x|^2}dx+A(u),
\end{equation}
with optimal constant still holds for every $u\in H^1_0(\Omega)$ is addressed. In the same paper, Theorem 4.1,
the authors find
$$
A(u)=\lambda(\Omega)\displaystyle\int_{\Omega}
u^2(x)dx, \ \ \lambda(\Omega)=z^2_0\left(\frac{\omega_n}{|\Omega|}\right)^{2/n},
$$
where $z_0\approx2.4048$ is the first zero of the Bessel's function $J_0(z)$ whereas
$\omega_n$ and $|\Omega|$ are the volume of the unit ball and resp.
$\Omega$.

The following generalization of (\ref{2eq14}) for $1<p<n$ is proved in  \citet{GGM03},  Theorem 1,
\begin{equation}
\label{2eq15}
\int_\Omega|\nabla u|^pdx\geq\left(\frac{n-p}{p}\right)^p\int_\Omega\frac{|u|^p}{|x|^p}dx+C(p,n)\left(\frac{\omega_n}{|\Omega|}\right)^{\frac{p}{n}}\int_\Omega|u|^pdx,
\end{equation}
where $\Omega$ is a bounded domain, $0\in\Omega$, the constant $\left(\frac{n-p}{p}\right)^p$ is optimal and $C(p,n)$ is given  explicitly  for $p\geq2$.

In \citet{FT02},  Theorem D, the additional term in (\ref{2eq14}) is with singular weight, i.e.
$$
\begin{array}{lll}
&&\int_\Omega|\nabla u|^2dx
\\[2pt]
\\
&&\geq\left(\frac{n-2}{2}\right)^2\int_\Omega\frac{u^2}{|x|^2}dx+\frac{1}{4}\sum_{i=1}^\infty\int_\Omega\frac{u^2}{|x|^2}X_1^2\left(\frac{|x|}{R}\right)\ldots X_i^2\left(\frac{|x|}{R}\right)dx,
\end{array}
$$
for every $u\in H^1_0(\Omega)$. Here $R\geq\sup_{x\in\Omega}|x|$, $X_1(t)=(1-\ln t)^{-1}$, $X_k(t)=X_1(X_{k-1}(t))$ for $k=2,\ldots $, $p=2$, $n\geq3$ and $\Omega\subset R^n$ is a bounded domain, $0\in\Omega$.

When $u\in W^{1,p}(\Omega)$ is not zero on $\partial\Omega$, the result in \citet{FT02} is extended  in \citet{AE05},  Theorem 1.1, to the inequalities
$$
\begin{array}{lll}
&&\int_\Omega|\nabla u|^pdx\geq\left(\frac{n-p}{p}\right)^p\int_\Omega\frac{|u|^p}{|x|^p}dx+C(p,n)\int_\Omega\sum_{j=1}^k\left(\log^{(j)}(R/|x|)\right)^{-p}\frac{|u|^p}{|x|^p}dx
\\
&+&b(\Omega,p,R)\int_{\partial\Omega}|u|^pds, \ \ \hbox{for } 1<p<n,
\end{array}
$$
and
$$
\begin{array}{lll}
&&\int_\Omega|\nabla u|^ndx\geq\left(\frac{n-1}{n}\right)^n\int_\Omega\frac{|u|^n}{\left(|x|\log R/|x|\right)^n}dx
\\[2pt]
\\
&&+C(n)\int_\Omega\sum_{j=2}^k\left(\log^{(j)}(R/|x|)\right)^{-n}\frac{|u|^n}{|x|^n}dx
+b(\Omega,n,R)\int_{\partial\Omega}|u|^nds, \ \ \hbox{for } p=n.
\end{array}
$$
Here $\Omega\subset R^n$ is a bounded domain, $0\in\Omega$, $\log_{(1)}a=\log a$, $\log_{(k)}a=\log(\log_{(k-1)} a)$ with $a>e^{(k-1)}$, $k\geq2$, $\log^{(k)}a=\prod_{j=1}^k\log_{(j)}a$ for $a>e^{(k-1)}$, $e^{(1)}=e$, $e^{(k+1)}=e^{e^k}$ and $R>e^{(k-1)}\sup_{x\in\Omega}|x|$.

Let us mention the result of \citet{WW03},  Theorem 2, where the weights in both sides of  Hardy inequality (\ref{2eq14}) are singular
$$
\int_\Omega|x|^{\alpha}|\nabla u|^2dx\geq\left(\frac{n-2+\alpha}{2}\right)^2\int_\Omega\frac{u^2}{|x|^{2-\alpha}}dx+\frac{1}{4}\int_\Omega\frac{u^2}{|x|^{2-\alpha}}\ln^{-2}\frac{R}{|x|}dx,
$$
and $R>\sup_{x\in\Omega}|x|$, $\alpha>2-n$. The constants $\left(\frac{n-2+\alpha}{2}\right)^2$ and $\frac{1}{4}$ are optimal.

Another direction  proposed in \citet{VZ00},  Theorem 2.2,  is the improved Hardy--Poincare inequality
$$
\int_\Omega|\nabla u|^2dx\geq\left(\frac{n-2}{2}\right)^2\int_\Omega\frac{u^2}{|x|^2}dx+C(q,\Omega)\left(\int_\Omega| u|^qdx\right)^{2/q},
$$
for $1\leq q<2$, $\Omega\subset R^n$ is a bounded domain, $0\in\Omega$ and $u\in H^1_0(\Omega)$.

In \citet{FT02},  Theorem A, and in \citet{AFT09},  Theorem A, the following result for the improved Hardy--Sobolev inequality in the unit ball $B_1\subset R^n$, $n\geq3$ is obtained
\begin{equation}
\label{eqq13}
\int_{B_1}|\nabla u|^2dx\geq\left(\frac{n-2}{2}\right)^2\int_{B_1}\frac{|u|^2}{|x|^2}dx
+C_n(a)\left(\int_{B_1}X_1^{\frac{2(n-1)}{n-2}}(a,|x|)|u|^{\frac{2n}{n-2}}dx\right)^{\frac{n-2}{n}},
\end{equation}

for every $u\in C_0^\infty(B_1)$, where $X_1(a,s)=(a-\log s)^{-1}$, $a>0, 0<s<1$. Here $\left(\frac{n-2}{2}\right)^2$ and $C_n(a)$ are optimal constants, where  $C_n(a)$ is given by

\begin{equation}
\label{eqq12}
C_n(a)=\left\{\begin{array}{l} (n-2)^\frac{2(n-1)}{n}S_n, \ \ a\geq\frac{1}{n-2},
\\
a^{\frac{2(n-1)}{n}}S_n, \ \ 0<a<\frac{1}{n-2},
\end{array}\right.
\end{equation}
and $S_n=\pi n(n-2)\left(\Gamma\left(\frac{n}{2}\right)/\Gamma(n)\right)^{\frac{2}{n}}$ is the best constant in the classical Sobolev inequality
$$
\int_{R^n}|\nabla u|^2dx\geq S_n\left(\int_{R^n}|u|^{\frac{2n}{n-2}}dx\right)^{\frac{n-2}{n}}.
$$
An improvement of Hardy inequality (\ref{2eq14}) and  (\ref{eqq13}) for $p>n$, which is of independent interest,  is proved  in \citet{Ps12}, Theorem C.

In \citet{BFT03b},  Theorem A,  the  improved Hardy inequality
\begin{equation}
\label{2eq17}
\int_\Omega|\nabla u|^pdx\geq\left|\frac{n-p}{p}\right|^p\int_\Omega\frac{|u|^p}{|x|^p}dx
+\frac{p-1}{2p}\left|\frac{n-p}{p}\right|^{p-2}\int_\Omega\frac{|u|^p}{|x|^p}X^2\left(\frac{|x|}{D}\right)dx,
\end{equation}
is proved in a bounded domain $\Omega\subset R^n$ with $0\in\Omega$ and $u\in W^{1,p}_0(\Omega\backslash\{0\})$.

Here $X(t)=-1/\log t$, $t\in(0,1)$ and $D\geq D_0$ where $D_0(n,p)\geq\sup_{x\in\Omega}|x|$ is some positive constant. The constants in  (\ref{2eq17}) are optimal. In fact, (\ref{2eq11}) and (\ref{2eq17}) are a consequence of a more general result in \citet{BFT03b},  Theorem  A, when the distance function $d(x)$ is the distance of $x\in\Omega$ to a piecewise smooth surface $K$ of the co-dimension $k$, $1\leq k\leq n$.

For $p>n$ improvements of Hardy inequality (\ref{2eq14}) are given in \citet{Ps12}: Theorem A for Sobolev-type inequality and Theorem B for Hardy-Morrey inequality.

Let us mention   the paper \citet{Fr09} where for the Hardy inequality with a point singularity a reminder term involving fractional power of the Laplacian is derived. By the embedding theorems, this implies inequalities in the spirit of \citet{BV97} as well remainder terms with $L^q$ norms.

Unfortunately, in all papers mentioned above, inequality  (\ref{2eq14}) with optimal constant $\left(\frac{n-2}{2}\right)^2$ or (\ref{2eq15}) with optimal constant $\left(\frac{n-p}{p}\right)^p$ are not sharp, see Definition \ref{2def1}.

In fact, in \citet{PT05} (see also the references
therein) it was shown by variational technique that  Hardy
inequality (\ref{2eq3}) with $p=2$ is not sharp for the optimal constant
$\left(\frac{n-2}{2}\right)^2$.

Finally, we will  briefly refer to the case of a point singularity of the weights on the boundary $\partial \Omega$, i.e., $0\in\partial\Omega$, see \citet{FTT09, FM12, BFT18, DFP12} and the references therein. In this case the optimal constant $\left(\frac{n-2}{2}\right)^2$ in (\ref{2eq14})  for the weight singular at an interior point is replaced with  greater constant $\frac{n^2}{4}$ when $\Omega$ satisfies some geometric conditions, for example, when $\Omega$ is a convex domain.

\section{Hardy inequalities in abstract form}
\label{sect3}
In this section we derive a new Hardy inequalities with weights  in abstract form.
Examples  are presented to demonstrate the applicability of
the method and to show  generalizations of  existing results.
The sharpness of the inequalities is proved and the results are illustrated
by several examples. The section is based on
\citet{FKR14c, FKR15}.
\subsection{General weights}
\label{31sec1}  Let $\Omega$ be a bounded domain,
$\Omega\subset R^n$, $n\geq2$ with a boundary $\partial\Omega\in C^1$. Suppose that
$f$ is a vector function defined in $\Omega$, $|f|\neq 0$   with components
$f_i\in C^1(\Omega)\cap C(\bar{\Omega})$, $i=1,\cdots,n$. Let $p>1$ and assume that in $\Omega$ there exist
measurable functions $v$, $w$, $v^{1-p}\in L^1(\Omega)$  such that

\begin{equation}
\label{3eq1} -\hbox{div}f-(p-1)v|f|^{p'}\geq w , \hbox{ for a.e. } x\in \Omega,
\end{equation}
where $\displaystyle\frac{1}{p}+\frac{1}{p'}=1$. Let $\partial\Omega$ be divided
into two parts $\partial\Omega=\Gamma_-\cup\Gamma_+$, where
\begin{equation}
\label{3eq99}
\Gamma_-=\{x\in\partial\Omega: \langle f,\eta\rangle<0\}, \ \
\Gamma_+=\{x\in\partial\Omega: \langle f,\eta\rangle\geq0\}.
\end{equation}
Here $\eta$ is the unit outward to $\Omega$ normal vector on
$\partial\Omega$ and $\langle.,.\rangle$ is the scalar product in $R^n$. We consider the
functions $u\in C^{\infty}_{\Gamma_-}(\Omega)$, where
$$
C^{\infty}_{\Gamma_-}=\{u\in C^{\infty}, u=0 \hbox{ in a
neighbourhood of } \Gamma_-\}.
$$
Let us introduce the notations
\begin{equation}
\label{3eq9}
\begin{array}{lll}L(u)&=&\displaystyle\int_{\Omega}v^{1-p}
\left|\frac{\langle f,\nabla u\rangle}{|f|}\right|^pdx, \ \
K_0(u)=\displaystyle\int_{\Gamma_+}\langle f,\eta\rangle|u|^pdS,
\\[2pt]
\\
K(u)&=&\displaystyle\int_{\Omega}v|f|^{p'}|u|^pdx, \ \
N(u)=\displaystyle\int_{\Omega}w|u|^pdx,
\end{array}
\end{equation}
where $dS$ is the (n-1)-dimensional surface measure and $u\in C^\infty_{\Gamma_-}(\Omega)$.

In this section our main result is the following theorem.
\begin{theorem}
\label{3th1} Under  condition  (\ref{3eq1})  for every
$u\in C^{\infty}_{\Gamma_-}(\Omega), u\not\equiv0$, and $v>0$, $w\geq0$, the following inequality holds
\begin{equation}
\label{3eq10} \displaystyle
L(u)\geq\left(\frac{1}{p}\right)^p\frac{(K_0(u)+(p-1)K(u)+N(u))^p}{K^{p-1}(u)}.
\end{equation}
\end{theorem}
\begin{proof}
Since
$$
\int_{\Omega}\langle f,\nabla
|u|^p\rangle dx=p\int_{\Omega}|u|^{p-2}u\langle f,\nabla u\rangle dx,
$$
applying the H\"older inequality on the right-hand side with
$$
\displaystyle v^{-1/p'}\frac{\langle f,\nabla u\rangle}{|f|} \ \  \hbox{and} \ \  \displaystyle
v^{1/p'}|f||u|^{p-2}u,
$$
as factor of the integrand we get
\begin{equation}
\label{3eq12} \left|\displaystyle\int_{\Omega}\langle f,\nabla|u|^p\rangle dx\right|\leq
p\left(\int_{\Omega}v^{1-p}\left|\frac{\langle f,\nabla
u\rangle }{|f|}\right|^pdx\right)^{1/p}\left(\int_{\Omega}v|f|^{p'}|u|^pdx\right)^{1/p'}.
\end{equation}
Rising both sides of (\ref{3eq12}) to power $p$ it follows that
\begin{equation}
\label{3eq13} \displaystyle\int_{\Omega}v^{1-p} \left|\frac{\langle f,\nabla
u\rangle}{|f|}\right|^pdx\geq\frac{\left|\displaystyle\frac{1}{p}\int_{\Omega}\langle f,\nabla|u|^p\rangle dx\right|^p}{\left(
\displaystyle\int_{\Omega}v|f|^{p'}|u|^pdx\right)^{p-1}}.
\end{equation}
Integrating by parts the numerator of the right-hand side of (\ref{3eq13}),
from (\ref{3eq1}),  we get
$$
\begin{array}{lll}
&&\left| \frac{1}{p}\int_{\Omega}\langle f,\nabla|u|^p\rangle dx\right|=\left|\frac{1}{p}
\int_{\partial\Omega}\langle f,\eta\rangle|u|^pdS-\frac{1}{p}\int_{\Omega}\hbox{div}
f|u|^pdx\right|
\\[2pt]
\\
&&=\left|\frac{1}{p}\int_{\partial\Omega}\langle f,\eta\rangle|u|^pdS-\frac{1}{p}\int_{\Omega}(\hbox{div}
f+(p-1)v|f|^{p'})|u|^pdx\right.
\\[2pt]
\\
&&\left.+\left(\displaystyle\frac{p-1}{p}\right)\int_{\Omega}v|f|^{p'}|u|^pdx\right|
\geq\frac{1}{p}\left|\int_{\partial\Omega}\langle f,\eta\rangle|u|^pdS\right.
\\[2pt]
\\
&&+\left.\int_\Omega w|u|^pdx+(p-1)\int_\Omega v|f|^{p'}|u|^pdx\right|
\\[2pt]
\\
&&\geq\frac{1}{p}\left(\int_{\Gamma_+}\langle f,\eta\rangle|u|^pdS+\int_\Omega w|u|^pdx+(p-1)\int_\Omega v|f|^{p'}|u|^pdx\right).
\end{array}
$$
The last equality follows from $u\left|_{\Gamma_-}\right.=0$.

From  (\ref{3eq13}) we obtain (\ref{3eq10}).
\end{proof}
\begin{remark}\label{3re1}\rm

The idea of the proof of Theorem \ref{3th1} comes from \citet{Bo07}, for $p=2$; \citet{FHT99}, Theorem II.1  and \citet{BFT03b}, Theorem 4.1. In our case, in contrast to these works, we consider functions not necessarily zero on the whole boundary $\partial\Omega$ and due to this there is an additional boundary term $K_0$ in (\ref{3eq10}).  In $L$ and $K$ there is also a weight $v$, which is $1$ in the above mentioned papers.
\end{remark}
The careful analysis of the proof of Theorem \ref{3th1} shows that we have a  more general result than (\ref{3eq10}) without any sign conditions of the boundary term.
\begin{corollary}
\label{cor300}
Suppose that $p>1$ and there exist in $\Omega$ measurable functions $v>0$, $w, v^{1-p}\in L^1(\Omega)$ such that  condition (\ref{3eq1}) holds. Then for every $u\in C^\infty(\bar{\Omega})$ the following inequality holds
\begin{equation}
\label{3eq100} \displaystyle
L(u)\geq\left(\frac{1}{p}\right)^p\frac{\left|K_3(u)+(p-1)K(u)+N(u)\right|^p}{K^{p-1}(u)},
\end{equation}
where $K_3(u)=\int_{\partial\Omega}\langle f,\eta\rangle|u|^pdS$.
\end{corollary}
\begin{corollary}
\label{cor3_1}
As a consequence of Theorem
\ref{3th1} we get under   condition (\ref{3eq1}) for $v>0$, $w\geq0$, $v^{1-p}\in L^1(\Omega)$ and $u\in C^\infty_{\Gamma_-}(\Omega)$ the following Hardy inequalities:
\begin{itemize}
\item[i)]
\begin{equation}
\label{3eq3}
L^{\frac{1}{p}}(u)\geq\left(\frac{1}{p'}\right)K^{\frac{1}{p}}(u)+\left(\frac{1}{p}\right)K_0(u)K^{-\frac{1}{p'}}(u),
\end{equation}
\item[ii)]
\begin{equation}
\label{300eq2}
L(u)\geq K_0(u)+N(u)\geq N(u),
\end{equation}
\item[iii)]
\begin{equation}
\label{300eq33}
L(u)\geq \left(\frac{1}{p'}\right)^pK(u)+ \left(\frac{1}{p'}\right)^{p-1}(K_0(u)+N(u))\geq \left(\frac{1}{p'}\right)^pK(u),
\end{equation}
\end{itemize}
for every $u\in C^\infty_{\Gamma_-}(\Omega)$.
\end{corollary}
\begin{proof}
i) $\quad$ Rising both sides of (\ref{3eq10}) to power $\frac{1}{p}$ and neglecting $N(u)\geq0$ we obtain (\ref{3eq3}).

ii) $\quad$ Applying the Young inequality

$$
\frac{Q^p}{H^{p-1}}\geq
ps^{p-1}Q-(p-1)s^pH,
$$
with $H>0$, $Q\geq0$ and constant $s\geq0$ to the right-hand side of
(\ref{3eq10}) for

$Q=\frac{1}{p}\left(K_0(u)+N(u)+(p-1)K(u)\right)$ and $H=K(u)$ we get
\begin{equation}
\label{3eq15} L(u)\geq s^{p-1}(K_0(u)+N(u))+(p-1)s^{p-1}(1-s)K(u) .
\end{equation}
For $s=1$ in (\ref{3eq15}) we get (\ref{300eq2}) and neglecting $K_0(u)$ since  $K_0(u)\geq0$ we obtain the last inequality in (\ref{300eq2}).

iii) $\quad$ Inequality (\ref{300eq33}) is a consequence of (\ref{3eq15}) for $s=\frac{1}{p'}=\frac{p-1}{p}$ and neglecting $K_0(u)\geq0$ and $N(u)\geq0$ we obtain the last inequality in (\ref{300eq33}).
\end{proof}
The form of  Hardy inequality  (\ref{3eq3}) is not the usual
one. It depends on  the derivative of $u$ in the direction of the unit
vector $\displaystyle\frac{f}{|f|}$, on two functions $v$, $w$
satisfying (\ref{3eq1}) and on additional term including boundary
integral.

Since $\langle f,\eta\rangle\geq0$ on $\Gamma_+$ and $\displaystyle|\nabla
u|^p\geq\left|\frac{\langle f,\nabla u\rangle}{|f|}\right|^p$,  in (\ref{3eq3})--(\ref{300eq33})
we can replace their left-hand sides correspondingly with
$$
\displaystyle \int_{\Omega}v^{1-p}|\nabla
u|^pdx \ \ \hbox{ and } \displaystyle\left(\int_{\Omega}v^{1-p}|\nabla
u|^pdx\right)^{1/p}.
$$
The careful analysis of the proof of Theorem \ref{3th1} shows that (\ref{3eq10}) is an equality if and only if
\begin{equation}
\label{3eq50}
\left|\int_{\Omega}\langle f,\nabla|u|^p\rangle dx\right|=\int_{\Omega}\langle f,\nabla|u|^p\rangle dx,
\end{equation}
H\"older inequality  becomes an equality, i.e.,
\begin{equation}
\label{3eq51}
\left|v^{-\frac{1}{p'}}\frac{\langle f,\nabla u\rangle}{|f|}\right|^p=k_1^p\left||v|^{\frac{1}{p'}}|f||u|^{p-2}u\right|^{p'},
\end{equation}
for a.e. $x\in\Omega$ and some constant $k_1>0$ and
\begin{equation}
\label{3eq52}
-\hbox{div}f-(p-1)v|f|^{p'}=w, \ \ \hbox{ in } \Omega.
\end{equation}
However, (\ref{3eq50}) and (\ref{3eq51}) are satisfied if
\begin{equation}
\label{3eq53}
u\langle f,\nabla u\rangle=|u\langle f,\nabla u\rangle|,
\end{equation}
\begin{equation}
\label{3eq54}
\langle f,\nabla u\rangle=k_1v|f|^{p'}u, \ \ k_1>0,
\end{equation}
for a.e. $x\in\Omega$.
Thus we get the following result for sharpness of Hardy inequality (\ref{3eq10}).
\begin{theorem}
\label{3th10}
Suppose $p>1, n\geq2$, $\Omega$ is a bounded domain with $C^1$ smooth boundary $\partial\Omega$, $v>0, v^{1-p}\in L^1(\Omega)$ and $w\geq0$ for a.e. $x\in\Omega$. Then Hardy inequality (\ref{3eq10}) becomes a non-trivial equality if $f, v, w, u$ satisfy (\ref{3eq52})--(\ref{3eq54}) and $u\in C^\infty_{\Gamma_-}(\Omega), u\not\equiv0$.
\end{theorem}
Let us note that the possibility to use a vector function $f$ and
two functions $v$ and $w$ in  inequalities (\ref{3eq3})--(\ref{300eq33})
serves for many new Hardy inequalities.
\subsection{Comparison with some existing results}
\label{3sect2n}
We will compare the result in Theorem \ref{3th1} with results in \citet{OK90} and \citet{BEL15}.
\begin{example}\rm
\label{3ex01}
Let  $y(x)$ be a solution of the equation (\ref{2eq400}) with properties listed in Sect. \ref{sect2}, i.e., $v_i, w$ are positive measurable functions, finite a.e. in $\Omega$, so that inequality (\ref{2eq4}) holds, see \citet{OK90}, Theorems 14.1 and 14.2. Under these conditions we will prove new Hardy inequalities by means of Theorem \ref{3th1}.

Suppose that $\Omega$ is a bounded domain in $R^n$, $n\geq2$. For  $1<p<n$ we define vector function $f=(f_1,\ldots, f_n)$ with
 $$
 f_i=v_i\left|\frac{\partial y}{\partial x_i}\right|^{p-2}\frac{\partial y}{\partial x_i}(|y|^{p-2}y)^{-1},\ \ i=1,\ldots,n,
 $$
and
 $$
 v(x)=\inf_{|\xi|=1}\left(\sum_{i=1}^nv_i\right)\left(\sum_{i=1}^nv_i^2|\xi_i|^{2(p-1)}\right)^{-\frac{p}{2(p-1)}}.
 $$
For every $u(x)\in C_0^\infty(\Omega)$ the following Hardy inequalities hold
\begin{equation}
\label{3eq171}
\begin{array}{lll}
&&\left(\int_\Omega v^{1-p}\left|\frac{\langle \nabla y, \nabla u\rangle}{|\nabla y|}\right|^pdx\right)^{\frac{1}{p}}
\\[2pt]
\\
&\geq&\frac{p-1}{p}\left[\int_\Omega\frac{v}{|y|^p}\left(\sum_{i=1}^nv_i^2\left|\frac{\partial y}{\partial x_i}\right|^{2(p-1)}\right)^{\frac{p}{2(p-1)}}|u|^pdx\right]^{\frac{1}{p}}
\\[2pt]
\\
&+&\frac{1}{p}\int_\Omega w|u|^pdx\left[\int_\Omega\frac{v}{|y|^p}\left(\sum_{i=1}^nv_i^2\left|\frac{\partial y}{\partial x_i}\right|^{2(p-1)}\right)^{\frac{p}{2(p-1)}}|u|^pdx\right]^{\frac{1-p}{p}},
\end{array}
\end{equation}
and
\begin{equation}
\label{3eq172}
\int_\Omega v^{1-p}\left|\frac{\langle \nabla y, \nabla u\rangle}{|\nabla y|}\right|^pdx\geq\int_\Omega w|u|^pdx.
\end{equation}

The proof of (\ref{3eq171})  and (\ref{3eq172}) follows from (\ref{3eq10}) and (\ref{300eq2}). Indeed, the vector function $f(x)$ satisfies inequality (\ref{3eq1}), i.e.,
$$
\begin{array}{lll}
 &-&\hbox{div}f-(p-1)v|f|^{p'}\geq w +(p-1)\sum_{i=1}^n\frac{v_i}{|y|^p}\left|\frac{\partial y}{\partial x_i}\right|^p
 \\[2pt]
 \\
 &-&(p-1)\frac{v}{|y|^p}\left(\sum_{i=1}^nv_i^2\left|\frac{\partial y}{\partial x_i}\right|^{2(p-1)}\right)^{\frac{p}{2(p-1)}}\geq w(x) \ \ \hbox{ in } \Omega,
 \end{array}
$$
because
$$
\begin{array}{lll}
&&\sum_{i=1}^nv_i\left|\frac{\partial y}{\partial x_i}\right|^p\left(\sum_{i=1}^nv_i^2\left|\frac{\partial y}{\partial x_i}\right|^{2(p-1)}\right)^{-\frac{p}{2(p-1)}}
\\[2pt]
\\
&\geq&\inf_{|\xi|=1}\sum_{i=1}^nv_i\left|\xi\right|^p\left(\sum_{i=1}^nv_i^2\left|\xi\right|^{2(p-1)}\right)^{-\frac{p}{2(p-1)}}=v(x).
\end{array}
$$
Since
$$
|f|^{p'}=\left(\sum_{i=1}^nv_i^2\left|\frac{\partial y}{\partial x_i}\right|^{2(p-1}\right)^{\frac{p}{2(p-1)}},
$$
we get
$$
\begin{array}{lll}
L(u)&=&\int_\Omega v^{1-p}\left|\frac{\langle f,\nabla u\rangle}{|f|}\right|^pdx, \ \ K_0(u)=0,
\\[2pt]
\\
K(u)&=&\int_\Omega\frac{v}{|y|^p}\left(\sum_{i=1}^nv_i^2\left|\frac{\partial y}{\partial x_i}\right|^{2(p-1}\right)^{\frac{p}{2(p-1)}}|u|^pdx, \ \  N(u)=\int_\Omega w|u|^pdx,
\end{array}
$$
because $u=0$ on $\partial\Omega$. Applying (\ref{3eq10}) and (\ref{3eq9}) we obtain (\ref{3eq171}) and (\ref{3eq172}).
\end{example}
\begin{example}\rm
\label{3ex02}
Let $z$ be a real-valued function $z\in W^{2,1}_{loc}(\Omega)$  in a bounded domain $\Omega\subset R^n$, $\partial\Omega\in C^1$, $n\geq2$ and $\Delta z$ is of one sign a.e. in $\Omega$. Then  the following inequalities hold for $p>1$, see \citet{DH98} and \citet{BEL15}, Theorem 1.2.8,
\begin{equation}
\label{30eq100}
\int_\Omega\frac{1}{|\Delta z|^{p-1}}\left|\langle\nabla z,\nabla u\rangle\right|^pdx\geq
\left(\frac{1}{p}\right)^p\int_\Omega|\Delta z||u|^pdx
\end{equation}
for every $u\in C^\infty_0(\Omega)$.

If additional $z\in C^1(\bar{\Omega})\cap W^{2,1}(\Omega)$ then
\begin{equation}
\label{3eq101}
\begin{array}{lll}
&&\left(\int_\Omega\frac{1}{|\Delta z|^{p-1}}\left|\langle\nabla z,\nabla u\rangle\right|^pdx\right)^{\frac{1}{p}}\geq
\frac{1}{p}\left(\int_\Omega|\Delta z||u|^pdx\right)^{\frac{1}{p}}
\\[2pt]
\\
&-&\frac{1}{p}\hbox{sgn }\Delta z\int_{\Gamma_+}\langle\nabla z,\eta\rangle|u|^pdS\left(\int_\Omega|\Delta z||u|^pdx\right)^{-\frac{p-1}{p}},
\end{array}
\end{equation}
for every $u\in C^\infty_{\Gamma_-}(\Omega)$. Here $\eta$ is the unit outward to $\Omega$ normal vector on $\partial\Omega$, where  $\Gamma_-, \Gamma_+$ are defined in (\ref{3eq99}).

Inequalities (\ref{30eq100}) and (\ref{3eq101}) follow from (\ref{3eq3}) for $f=-\left(\frac{1}{p-1}\right)^{p-1}(\hbox{sgn }\Delta z)\nabla z$,  $v=|\Delta z||\nabla z|^{-p'}$ and $w=0$ . Indeed, we get
$$
 -\hbox{div}f-(p-1)v|f|^{p'}=\left(\frac{1}{p-1}\right)^{p-1}|\Delta z|-\left(\frac{1}{p-1}\right)^{p-1}|\Delta z|=0 ,
$$
for a.e. in $\Omega$.

With the computations
$$
\begin{array}{lll}
L(u)&=&\displaystyle\int_{\Omega}v^{1-p}\left|\frac{\langle f,\nabla u\rangle}{|f|}\right|^pdx=\int_\Omega\frac{1}{|\Delta z|^{p-1}}\left|\langle\nabla z,\nabla u\rangle\right|^pdx,
\\[2pt]
\\
K_0(u)&=&-\int_{\Gamma_+}\langle f,\eta\rangle|u|^pdS=\left(\frac{1}{p-1}\right)^{p-1}(\hbox{sgn }\Delta z)\int_{\Gamma_+}\langle\nabla z,\eta\rangle|u|^pdS,
\\[2pt]
\\
K(u)&=&\int_{\Omega}v|f|^{p'}|u|^pdx=\left(\frac{1}{p-1}\right)^p\int_\Omega|\Delta z||u|^pdx,
\\[2pt]
\\
N(u)&=&0.
\end{array}
$$
applying (\ref{3eq3}) and (\ref{300eq33}) we obtain (\ref{30eq100}) and (\ref{3eq101}), respectively.
\end{example}
\subsection{Sharp Hardy inequalities}
\label{31sec2}

We illustrate below the possibility to choose a vector function $f$ in order to obtain sharp Hardy inequality.
\subsubsection*{Inequality with weight the first eigenfunction of the p-Laplacian}
\label{31sec2-1}
Let $\varphi$ be the first eigenfunction of the
p-Laplacian  in a bounded domain $\Omega\subset R^n$, $p>1$, $n\geq2$ with the first
eigenvalue $\lambda>0$
$$
\left|\begin{array}{lll}
&-&\Delta_p\varphi=\lambda|\varphi|^{p-2}\varphi, \ \
\hbox{ in } \Omega,
\\[2pt]
\\
&&\varphi|_{\partial\Omega}=0.\end{array}\right.
$$
Let us define the vector function
$f=\displaystyle\frac{|\nabla\varphi|^{p-2}\nabla\varphi}{|\varphi|^{p-2}\varphi}$ in every domain $\Omega_0$,  $\bar{\Omega_0}\subset\Omega$ such that
 for $f=f_1\ldots f_n$ we have $f_i\in C^1(\Omega_0)$  and
$$
-\hbox{div}f=-\displaystyle\frac{\Delta_p\varphi}{|\varphi|^{p-2}\varphi}+
(p-1)\displaystyle\frac{|\nabla\varphi|^p}{|\varphi|^p}=\lambda+(p-1)|f|^{p'}\ \ \ x\in\Omega_0,
$$
i.e., $v=1$ and $\omega=\lambda$ in (\ref{3eq1}).

If we fix $u\in C_0^\infty(\Omega)$, then $\hbox{supp }u\subset \Omega$ so that we can apply  Theorem \ref{3th1} and obtain
 the inequality
\begin{equation}
\label{3eq3-1}
L(u)\geq\left(\displaystyle\frac{1}{p}\right)^p\frac{[(p-1)K(u)+N(u)]^p}{K^{p-1}(u)},
\ \,
\end{equation}
where
$$
L(u)=\displaystyle\int_{\Omega}\left|\frac{\langle\nabla\varphi, \ \
\nabla u\rangle}{|\nabla\varphi|}\right|^pdx, \ \ K(u)=\displaystyle\int_{\Omega}\left|\frac{\nabla \varphi}{\varphi}\right|^p|u|^pdx, \ \ N(u)=\lambda\displaystyle\int_{\Omega}|u|^pdx.
$$
Note that under arguments of completeness the inequality (\ref{3eq3-1}) holds for every $u\in W_0^{1,p}(\Omega)$, moreover
simple computation gives us that  inequality (\ref{3eq3-1}) is sharp,
i.e., becomes an equality  for $u(x)=\varphi(x)$.

We consider the case $p=2$,  $n=3$ and $\Omega=B_1$. In this case
the first eigenfunction is

$\varphi=\sqrt{2}\displaystyle\frac{\sin\pi
r}{\pi r}$, $r=|x|$, and the eigenvalue is $\lambda=\pi^2$, see \citet{Vl71}, Sect. 28.1.

Now we have

$f=\frac{\nabla\varphi}{\varphi}=\left(\pi \cot\pi
r-\displaystyle\frac{1}{r}\right)\frac{x}{r}$,
$|f|^2=\left(\pi \cot\pi r-\displaystyle\frac{1}{r}\right)^2$ and $-\hbox{div}f=|f|^2+\pi^2$.

Applying Theorem \ref{3th1} to

$L(u)=\int_{B_1}|\nabla u|^2dx$, $K(u)=\int_{B_1}|f|^2u^2dx$ and $N(u)=\pi^2\int_{B_1}u^2dx$

we get
\begin{equation}
\label{3eqn3}
L(u)\geq\frac{1}{4}\left(K(u)+2N(u)+N^2(u)K^{-1}(u)\right) \hbox{ for } u\in C_0^\infty(B_1).
\end{equation}
Using the series expansion for the function $\cot(z)$, see \citet{Re91},
we obtain
\begin{equation}
\label{3eqn5}
\begin{array}{lll}
&&\left(\pi\cot\pi
r-\frac{1}{r}\right)^2=\left(\displaystyle\sum_{k=1}^\infty\displaystyle\frac{2r}{r^2-k^2}\right)^2
\\[2pt]
\\
&&=\frac{4r^2}{(r^2-1)^2}\left[1+\displaystyle\sum_{k=2}^\infty\displaystyle\frac{r^2-1}{r^2-k^2}\right]^2
=\frac{r^2}{(r-1)^2}\left[\frac{2}{r+1}
+2\sum_{k=2}^\infty\displaystyle\frac{r-1}{r^2-k^2}\right]^2,
\end{array}
\end{equation}
for the kernel of $K(u)$ and $r\in(0,1)$

Using (\ref{3eqn5}) for  (\ref{3eqn3}),  the following Hardy inequality holds
\begin{equation}
\label{2eq12}
\begin{array}{lll}
\int_{B_1}| \nabla
u|^2dx&\geq&\frac{1}{4}\left[\int_{B_1}\frac{u^2}{(1-|x|)^2}\left(\frac{2}{|x|+1}
+2\sum_{k=2}^\infty\frac{|x|-1}{|x|^2-k^2}\right)^2dx\right.
\\[2pt]
\\
&+&\left.2\pi^2\int_{B_1}u^2dx+\frac{\pi^4
\left(\int_{B_1}u^2dx\right)^2}{\int_{B_1}
\frac{u^2}{(1-|x|)^2}\left(\frac{2}{|x|+1}
+2\sum_{k=2}^\infty\frac{|x|-1}{|x|^2-k^2}\right)^2dx}\right].
\end{array}
\end{equation}
Since the last term in (\ref{2eq12}) is positive,
$\frac{2}{|x|+1}\geq1$ and $1-|x|=d(x)=\hbox{dist}(x,\partial B_1)$,
we can rewrite (\ref{2eq12}) as
\begin{equation}
\label{2eq13} \int_{B_1}| \nabla
u|^2dx\geq\frac{1}{4}\int_{B_1}\frac{u^2}{d(x)^2}dx+A(u)\ \ u\in W^{1,p}_0(B_1),
\end{equation}
where
$$
\begin{array}{lll}
A(u)&=&\frac{1}{4}\int_{B_1}\left[\left(\frac{1-|x|}{1+|x|}
+2\sum_{k=2}^\infty\frac{|x|-1}{|x|^2-k^2}\right)^2-1\right]u^2dx
\\[2pt]
\\
&+&\frac{1}{2}\pi^2\int_{B_1}u^2dx+\frac{1}{4}\frac{\pi^4
\left(\int_{B_1}u^2dx\right)^2}{\int_{B_1}
\frac{u^2}{(1-|x|)^2}\left(\frac{2}{|x|+1}
+2\sum_{k=2}^\infty\frac{|x|-1}{|x|^2-k^2}\right)^2dx}>0.
\end{array}
$$
Inequality (\ref{2eq13}) has an optimal constant $\frac{1}{4}$ and moreover it is sharp, i.e., for function $u(x)=\sqrt{2}\frac{\sin\pi|x|}{\pi|x|}$ it becomes an equality.
The above example shows that sharp inequality (\ref{2eq9}) with optimal constant $\frac{1}{4}$, see Definition \ref{2def1}, is possible but for more complicated additional term $A(u)$. Thus, in this special case we give a positive answer to the question of \citet{BM97}.

\subsubsection*{Hardy inequalities in an annulus and in a ball}
\label{31sec2-2}

Let us define for $p>1$, $p'= \displaystyle\frac{p}{p-1}$, $n\geq 2$, $m= \displaystyle\frac{p-n}{p-1}$ and $0\leq r<R$
the sets of functions:
\begin{equation}
\label{70eq12}
M(r,R)=\left\{\begin{array}{l} u:
\displaystyle\int_{B_R\backslash B_r}\left|\frac{\langle x,\nabla
u\rangle}{|x|}\right|^pdx<\infty, \ \ 0\leq r<R,
\\[2pt]
\\
 \hbox{ and }
\left|R^m-\hat{R}^m\right|^{1-p}\displaystyle\int_{\partial
B_{\hat{R}}}|u|^pdS\rightarrow0, \ \ \hat{R}\rightarrow R-0, \
\ m\neq0,
\\[2pt]
\\
\left|\displaystyle\ln{\frac{R}{\hat{R}}}\right|^{1-p}\displaystyle\int_{\partial
B_{\hat{R}}}|u|^pdS\rightarrow0, \ \ \hat{R}\rightarrow R-0, \
\ m=0.
\end{array}\right.
\end{equation}
\begin{proposition}
\label{3prop1n}
For $u\in M(r,R)$ and $0<r<R$  the following Hardy inequalities hold:
\begin{itemize}
\item $\quad$ for $m\neq0$, i.e., $p\neq n$
\begin{equation}
\label{7eq12}
\begin{array}{lll}
&&\left(\int_{B_R\backslash B_r}\left|\frac{<x,\nabla
u>}{|x|}\right|^pdx\right)^{\frac{1}{p}}
\\[2pt]
\\
&&\geq\left|\frac{n-p}{p}\right|\left(\int_{B_R\backslash
B_r}\frac{|u|^p}{|x|^{(n-1)p'}\left|R^m-|x|^m
\right|^p}dx\right)^{\frac{1}{p}}
\\[2pt]
\\
&&+\frac{1}{p}r^{1-n}\left|R^m-r^m
\right|^{1-p}\int_{\partial B_r}|u|^pdS\left(\int_{B_R\backslash
B_r}\frac{|u|^p}{|x|^{(n-1)p'}\left|R^m-|x|^m
\right|^p}dx\right)^{-\frac{1}{p'}}.
\end{array}
\end{equation}
For function
$u_k(x)=\left(\frac{R^m-|x|^m}{m}\right)^{k}, \ \
k>\frac{1}{p'}$ inequality  (\ref{7eq12}) becomes an equality;
\item $\quad$ for $m=0$, i. e., $p=n$
\begin{equation}
\label{7eq13}
\begin{array}{lll}
&&\left(\displaystyle\int_{B_R\backslash B_r}\left|\frac{<x,\nabla
u>}{|x|}\right|^ndx\right)^{\displaystyle\frac{1}{n}}
\geq\displaystyle\frac{n-1}{n}\left(\int_{B_R\backslash
B_r}\frac{|u|^n}{|x|^n\left|\ln\displaystyle\frac{R}{|x|}\right|^n}dx\right)^{\displaystyle\frac{1}{n}}
\\[2pt]
\\
&+&\displaystyle\frac{1}{n}\left(r\ln\displaystyle\frac{R}{r}\right)^{1-n}\int_{\partial
B_r}|u|^ndS\left(\int_{B_R\backslash
B_r}\frac{|u|^n}{|x|^n\left|\ln\displaystyle\frac{R}{|x|}\right|^n}dx\right)^{-\frac{1}{n'}}.
\end{array}
\end{equation}
For function
$u_s(x)=\left(\ln\frac{R}{|x|}\right)^{s}, \ \
s>\frac{1}{n'}$ inequality  (\ref{7eq13}) becomes an equality.
\end{itemize}
\end{proposition}
\begin{proof}
Let the function $\psi(x)$ be a solution of the problem:
$$
\left\{\begin{array}{l} \Delta_p\psi=0, \ \ \hbox{ in } B_R\backslash B_r,
\\[2pt]
\\
 \psi|_{\partial B_R}=0,\ \ \psi\left|_{\partial B_r}=1,\right.\end{array}
\right.,
$$
then
\begin{equation}
\label{39eq88}
\psi(x)=\left\{\begin{array}{l}\frac{R^m-|x|^m}{R^m-r^m},
\ \ m\neq 0,
\\[2pt]
\\
\frac{\ln{\frac{R}{|x|}}}{\ln{\frac{R}{r}}},
\ \ m= 0.\end{array}\right. \ \
\end{equation}
Indeed
\begin{itemize}
\item $\quad$ for $m\neq0$ we have
$$
\begin{array}{lll}
&&\nabla\psi=-m|x|^{m-1}\frac{x}{|x|}\frac{1}{R^m-r^m},
\\[2pt]
\\
&&|\nabla\psi|^{p-2}\nabla\psi=-|m|^{p-2}m|x|^{(m-1)(p-1)}\frac{x}{|x|}\left(\frac{1}{R^m-r^m}\right)^{p-1},
\\[2pt]
\\
&&\Delta_p\psi=\hbox{div}\left(|\nabla\psi|^{p-2}\nabla\psi\right)
\\[2pt]
\\
&&=-|m|^{p-2}m|x|^{(m-1)(p-1)-1}\left(\frac{1}{R^m-r^m}\right)^{p-1}[(m-1)(p-1)+n-1]=0,
\end{array}
$$
because $(m-1)(p-1)+n-1=\left(\frac{p-n}{p-1}-1\right)(p-1)+n-1=0$.
\item for $m=0$ we have
$$
\begin{array}{lll}
\nabla\psi&=&-\frac{1}{|x|}\frac{1}{\ln\frac{R}{r}}\frac{x}{|x|}, \ \ |\nabla\psi|^{n-2}\nabla\psi=-\frac{1}{|x|^{n-1}}\frac{1}{\left(\ln\frac{R}{r}\right)^{n-1}}\frac{x}{|x|},
\\[2pt]
\\
\Delta_n\psi&=&\hbox{div}\left(|\nabla\psi|^{n-2}\nabla\psi\right)=-|x|^{-n}\left(\ln\frac{R}{r}\right)^{n-1}[-(n-1)+n-1]=0.
\end{array}
$$
\end{itemize}
Using the function $\psi$ in (\ref{39eq88}) we define  the vector function $f(x)$ in $B_R\backslash
B_r$  as $f=\frac{|\nabla\psi|^{p-2}}{|\psi|^{p-2}}\frac{\nabla\psi}{\psi}$ and let us check that
$$
f(x)=\left\{\begin{array}{l}-|x|^{-n}x\left(\displaystyle\frac{R^m-|x|^m}{m}\right)^{1-p},\
\ m\neq 0,
\\[2pt]
\\
-|x|^{-n}x\left(\ln\displaystyle\frac{R}{|x|}\right)^{1-n}, \ \
m=0.\end{array}\right.
$$
Indeed
\begin{itemize}
\item $\quad$ for $m\neq0$ we have
$$
\begin{array}{lll}
&&f=-|m|^{p-2}m|x|^{(m-1)(p-1)-1}x\left(R^m-|x|^m\right)^{-1}\left|R^m-|x|^m\right|^{2-p}
\\[2pt]
\\
&&=-|x|^{-n}x\left(\frac{R^m-|x|^m}{m}\right)^{1-p},
\end{array}
$$
because $(m-1)(p-1)-1=\left(\frac{p-n}{p-1}-1\right)(p-1)-1=-n$;
\item $\quad$ for $m=0$ we have
$$
f=-|x|^{-(n-1)}\left(\ln\frac{R}{|x|}\right)^{1-n}\frac{x}{|x|}=-|x|^{-n}x\left(\ln\frac{R}{|x|}\right)^{1-n}.
$$
\end{itemize}
Note that the outward normal $\eta$ to $B_{\hat{R}}\backslash B_r$, $r<\hat{R}<R$ is defined as
$$
\eta|_{\partial B_{\hat{R}}}=\displaystyle\frac{x}{|x|}|_{\partial B_{\hat{R}}}, \ \
\eta|_{\partial B_r}=-\displaystyle\frac{x}{|x|}|_{\partial B_r}.
$$
Moreover, we get for $u\in M(r,R)$ and
\begin{itemize}
\item $\quad$ for $m\neq0$
$$
|f|^{p'}=|x|^{(1-n)p'}\left(\frac{R^m-|x|^m}{m}\right)^{-p}, \ \ \int_{\partial(B_{\hat{R}}\backslash B_r)}\langle f,\eta\rangle|u|^pdS\geq0,
$$
and
\begin{equation}
\label{3eq2297}\begin{array}{lll}
&& -\hbox{div}f= -\hbox{div}\left(\left|\frac{\nabla\psi}{\psi}\right|^{p-2}\frac{\nabla\psi}{\psi}\right)
\\[2pt]
\\
&&=-\frac{\Delta_p\psi}{\psi^{p-1}}+|x|^{-n}\left\langle x, \nabla\left(\frac{R^m-|x|^m}{m}\right)^{1-p}\right\rangle
\\[2pt]
\\
&&=(p-1)|x|^{-n}\left(\frac{R^m-|x|^m}{m}\right)^{-p}|x|^m
\\[2pt]
\\
&&=(p-1)|x|^{m-n}\left(\frac{R^m-|x|^m}{m}\right)^{-p}=(p-1)|f|^{p'};
\end{array}
\end{equation}
\item $\quad$ for $m=0$
$$
|f|^{n'}=|x|^{-n}\left(\ln\frac{R}{|x|}\right)^{-n}, \ \ \int_{\partial(B_R\backslash B_r)}\langle f,\eta\rangle|u|^ndS\geq0,
$$
and
\begin{equation}
\label{3eq299}\begin{array}{lll}
&& -\hbox{div}f=-\hbox{div}\left(\left|\frac{\nabla\psi}{\psi}\right|^{n-2}\frac{\nabla\psi}{\psi}\right)
\\[2pt]
\\
&&=-\frac{\Delta_n\psi}{\psi^{n-1}}+|x|^{-n}\left\langle x, \nabla\left(\ln\frac{R}{|x|}\right)^{1-n}\right\rangle
\\[2pt]
\\
&&=-(1-n)|x|^{-n}\left|\ln\frac{R}{|x|}\right|^{-n}\left\langle x, \frac{x}{|x|^2}\right\rangle
\\[2pt]
\\
 &&=(n-1)|x|^{-n}\left(\ln\frac{R}{|x|}\right)^{-n}=(n-1)|f|^{n'}.
\end{array}
\end{equation}
\end{itemize}
Since the vector function $f(x)$ satisfies (\ref{3eq1}) with $v=1, w\equiv0$ then using (\ref{3eq2297}), (\ref{3eq299}) and   applying Theorem \ref{3th1} in $B_{\hat{R}}\backslash \bar{B_r }$ we  obtain inequalities (\ref{7eq12}), (\ref{7eq13}) after the limit $\hat{R}\rightarrow R$.

We will prove the sharpness of (\ref{7eq12}) only for $m>0$ because in the case $m<0$ the proof is similar and we omit it.

First, let us evaluate for $k>\frac{1}{p'}$ the integral
$$
I_m=\int_{B_R\backslash B_r}\frac{dx}{|x|^{(n-1)p'}(R^m-|x|^m)^{p(1-k)}}.
$$
With a change of variables $y=\frac{x}{R}$ and $\rho=|y|$ we get
\begin{equation}
\label{3eq399}\begin{array}{lll}
I_m&=&R^{m(1-p+kp)}\int_{B_R\backslash B_r}\frac{dy}{|y|^{(n-1)p'}(1-|y|^m)^{p(1-k)}}
\\[2pt]
\\
&=&R^{m(1-p+kp)}\omega_n\int_{r/R}^1\frac{\rho^{m-1}d\rho}{(1-\rho^m)^{p(1-k)}}
\\[2pt]
\\
&=&\omega_nm^{-1} \left(R^m-r^m\right)^{1-p+kp}(1-p+kp)^{-1},
\end{array}
\end{equation}
where we use  $(n-1)p'=n-m$ and $1-p+kp>0$ because $k>\frac{1}{p'}=\frac{p-1}{p}$.

With  $u_k(x)=\left(\frac{R^m-|x|^m}{m}\right)^{k}$ for  the left-hand side of (\ref{7eq12}) we get
$$
\begin{array}{lll}
(lhs)&=&\left(\int_{B_R\backslash B_r}\left|\frac{\langle x,\nabla
u_k\rangle}{|x|}\right|^pdx\right)^{\frac{1}{p}}
\\[2pt]
\\
&=&k\left(\int_{B_R\backslash B_r}|x|^{(m-1)p}\left(\frac{R^m-|x|^m}{m}\right)^{(k-1)p}dx\right)^{\frac{1}{p}}
\\[2pt]
\\
&=&km^{1-k}\left(\int_{B_R\backslash B_r}\frac{dx}{|x|^{(n-1)p'}\left(R^m-|x|^m\right)^{p(1-k)}}\right)^{\frac{1}{p}}=km^{1-k}I_m^{\frac{1}{p}}
\\[2pt]
\\
&=&k\omega_n^{\frac{1}{p}}m^{-\frac{1-p+kp}{p}}\left(R^m-r^m\right)^{\frac{1-p+kp}{p}}(1-p+kp)^{-\frac{1}{p}},
\end{array}
$$
where we use  $(m-1)p=\left(\frac{p-n}{p-1}-1\right)p=-(n-1)p'$.

For the terms in the right-hand side of (\ref{7eq12}) using the expression (\ref{3eq399}) for $I_m$ we get
\begin{equation}
\label{3eq292}\begin{array}{lll}
(rhs)_1&=&\frac{p-n}{p}\left(\int_{B_R\backslash B_r}\frac{|u_k|^pdx}{|x|^{(n-1)p'}\left(R^m-|x|^m\right)^{p}}\right)^{\frac{1}{p}}
\\[2pt]
\\
&=&\frac{p-n}{p}m^{-k}\left(\int_{B_R\backslash B_r}\frac{dx}{|x|^{(n-1)p'}\left(R^m-|x|^m\right)^{p(1-k)}}\right)^{\frac{1}{p}}
\\[2pt]
\\
&=&\frac{p-n}{p}m^{-1}m^{-k+1}I_m^{\frac{1}{p}}=\frac{p-1}{p}m^{-k+1}I_m^{\frac{1}{p}}
\\[2pt]
\\
&=&\frac{p-1}{p}\omega_n^{\frac{1}{p}}m^{-\frac{1-p+kp}{p}}\left(R^m-r^m\right)^{\frac{1-p+kp}{p}}(1-p+kp)^{-\frac{1}{p}}.
\end{array}
\end{equation}

\begin{equation}
\label{3eq293}\begin{array}{lll}
(rhs)_2&=&\frac{1}{p}r^{1-n}\left(R^m-r^m\right)^{1-p}\int_{\partial B_r}|u_k|^pdS
\\[2pt]
\\
&\times&\left(\int_{B_R\backslash B_r}\frac{|u_k|^pdx}{|x|^{(n-1)p'}\left(R^m-|x|^m\right)^{p}}\right)^{-\frac{1}{p'}}
\\[2pt]
\\
&=&\frac{1}{p}\omega_n r^{n-1}r^{1-n}\left(R^m-r^m\right)^{1-p}m^{-kp}\left(R^m-r^m\right)^{pk}\left(m^{-pk}I_m\right)^{-\frac{1}{p'}}
\\[2pt]
\\
&&=\frac{1}{p}\omega_n^\frac{1}{p}m^{-\frac{1-p+kp}{p}}\left(R^m-r^m\right)^{\frac{1-p+kp}{p}}(1-p+kp)^{\frac{1}{p'}}.
\end{array}
\end{equation}
Adding (\ref{3eq292}) and (\ref{3eq293})   we obtain
$$\begin{array}{lll}
(rhs)_1+(rhs)_2&=&\frac{1}{p}\omega_n^{\frac{1}{p}}m^{-\frac{1-p+kp}{p}}\left(R^m-r^m\right)^{\frac{1-p+kp}{p}}
\\[2pt]
\\
&\times&\left[\frac{p-1}{(1-p+kp)^{\frac{1}{p}}}+(1-p+kp)^{\frac{1}{p'}}\right]
\\[2pt]
\\
&=&k\omega_n^{\frac{1}{p}}m^{-\frac{1-p+kp}{p}}\left(R^m-r^m\right)^{\frac{1-p+kp}{p}}(1-p+kp)^{-\frac{1}{p}}
\\[2pt]
\\
&=&\left(\int_{B_R\backslash B_r}\left|\frac{\langle x,\nabla
u\rangle}{|x|}\right|^pdx\right)^{\frac{1}{p}}=(lhs),
\end{array}
$$
which proves the sharpness of (\ref{7eq12}) for $m>0$.

Analogously, for the sharpness of (\ref{7eq13}), first we evaluate the  integral
$$
\begin{array}{lll}
I_0&=&\int_{B_R\backslash B_r}|y|^{-n}\left(\ln\frac{1}{|y|}\right)^{n(s-1)}dy
=\omega_n\int_{r/R}^1\left(\ln\frac{1}{\rho}\right)^{n(s-1)}\rho^{-1}d\rho
\\[2pt]
\\
&=&\omega_n\left(\ln\frac{R}{r}\right)^{1-n+ns}(1-n+ns)^{-1},
\end{array}
$$
where we use  $1-n+ns>0$ because $s>\frac{1}{n'}$.

With $u_s(x)=\left(\ln\frac{R}{|x|}\right)^s$ for the left-hand side of (\ref{7eq13}) we get
$$
\begin{array}{lll}
(lhs)&=&\left(\int_{B_R\backslash B_r}\left|\frac{\langle x,\nabla
u_s\rangle}{|x|}\right|^ndx\right)^{\frac{1}{n}}=s\left(\int_{B_R\backslash B_r}|x|^{-n}\left(\ln\frac{R}{|x|}\right)^{(s-1)n}dx\right)^{\frac{1}{n}}
\\[2pt]
\\
&=&sI_0^{\frac{1}{n}}=s\omega_n^{\frac{1}{n}}\left(\ln\frac{R}{r}\right)^{\frac{1-n+sn}{n}}(1-n+sn)^{-\frac{1}{n}}.
\end{array}
$$
For the terms in the right-hand side we get
\begin{equation}
\label{3eq296}
\begin{array}{lll}
(rhs)_1&=&\frac{n-1}{n}\left(\int_{B_R\backslash B_r}\frac{|u_s|^ndx}{|x|^n\left(\ln\frac{R}{|x|}\right)^n}\right)^{\frac{1}{n}}
\\[2pt]
\\
&=&\frac{n-1}{n}\left(\int_{B_R\backslash B_r}|x|^{-n}\left(\ln\frac{R}{|x|}\right)^{n(s-1)}dx\right)^{\frac{1}{n}}
\\[2pt]
\\
&=&\frac{n-1}{n}I_0^{\frac{1}{n}}=\frac{n-1}{n}\omega_n^{\frac{1}{n}}(1-n+sn)^{-{\frac{1}{n}}}\left(\ln\frac{R}{r}\right)^{\frac{1-n+sn}{n}},
\end{array}
\end{equation}
and
\begin{equation}
\label{3eq297}
\begin{array}{lll}
(rhs)_2&=&\frac{1}{n}\left(r\ln\frac{R}{r}\right)^{1-n}\int_{\partial B_r}|u_s|^ndS\left(\int_{B_R\backslash B_r}\frac{|u_s|^ndx}{|x|^n\left(\ln\frac{R}{|x|}\right)^n}\right)^{-\frac{1}{n'}}
\\[2pt]
\\
&=&\frac{1}{n}\omega_n r^{1-n}\left(\ln\frac{R}{r}\right)^{1-n+sn}\left(\int_{B_R\backslash B_r}\frac{dx}{|x|^n\left(\ln\frac{R}{|x|}\right)^{n(s-1)}}\right)^{-\frac{1}{n'}}
\\[2pt]
\\
&=&\frac{1}{n}\omega_n\left(\ln\frac{R}{r}\right)^{1-n+sn}I_0^{\frac{1}{n'}}
\\[2pt]
\\
&=&\frac{1}{n}\omega_n^{\frac{1}{n}}\left(\ln\frac{R}{r}\right)^{1-n+sn}\left(\ln\frac{R}{r}\right)^{-(1-n+sn)\frac{1}{n'}}(1-n+sn)^{\frac{1}{n'}}
\\[2pt]
\\
&=&\frac{1}{n}\omega_n^{\frac{1}{n}}\left(\ln\frac{R}{r}\right)^{\frac{1}{n}}(1-n+sn)^{\frac{n-1}{n}}.
\end{array}
\end{equation}
Adding (\ref{3eq296}) and (\ref{3eq297}) we obtain
$$
\begin{array}{lll}
(rhs)_1+(rhs)_2&=&\omega_n^{\frac{1}{n}}\left(\ln\frac{R}{r}\right)^{\frac{1-n+sn}{n}}\left[\frac{1}{(1-n+sn)^{\frac{1}{n}}}\frac{n-1}{n}+\frac{(1-n+sn)^{\frac{n-1}{n}}}{n}\right]
\\[2pt]
\\
&=&s\omega_n^{\frac{1}{n}}\left(\ln\frac{R}{r}\right)^{\frac{1-n+sn}{n}}(1-n+sn)^{-\frac{1}{n}}
\\[2pt]
\\
&=&\left(\int_{B_R\backslash B_r}\left|\frac{\langle x,\nabla
u_s\rangle}{|x|}\right|^n\right)^{\frac{1}{n}}=(lhs),
\end{array}
$$
which proves the sharpness of (\ref{7eq13}).
\end{proof}
Using Proposition \ref{3prop1n} we will obtain Hardy inequalities in a ball which are sharp for $p>n$ and optimal for $p>1$

\begin{proposition}
\label{3prop2n}
For functions $u\in M(0,R)$  defined in (\ref{70eq12})  the following inequalities hold:

i) $\quad$ for $m>0$, i.e., $p>n$
\begin{equation}
\label{7eq16}
\begin{array}{lll}
&&\left(\displaystyle\int_{B_R}\left|\frac{<x,\nabla
u>}{|x|}\right|^pdx\right)^{\displaystyle\frac{1}{p}}
\\[2pt]
\\
&&\geq\displaystyle\frac{p-n}{p}\left(\int_{B_R}\frac{|u|^p}{|x|^{(n-1)p'}
\left|R^m-|x|^m\right|^p}dx\right)^{\displaystyle\frac{1}{p}}
\\[2pt]
\\
&&+\displaystyle\frac{1}{p}R^{n-p}\hbox{limsup}_{r\rightarrow0}\left[r^{1-n}\int_{\partial
B_r}|u|^pdS\right]
\\[2pt]
\\
&&\times\left(\int_{B_R}\frac{|u|^p}{|x|^{(n-1)p'}
\left|R^m-|x|^m\right|^p}dx\right)^{-\displaystyle\frac{1}{p'}}.
\end{array}
\end{equation}
For the functions
$u_k(x)=\left(\displaystyle\frac{R^m-|x|^m}{m}\right)^k, \ \
k>\displaystyle\frac{1}{p'}$, inequality  (\ref{7eq16}) becomes an equality and the constant $\frac{p-n}{p}$ in (\ref{7eq16}) is optimal.

 ii) $\quad$ for $m<0$, i.e., $p<n$
\begin{equation}
\label{7eq161}
\begin{array}{lll}
&&\left(\displaystyle\int_{B_R}\left|\frac{<x,\nabla
u>}{|x|}\right|^pdx\right)^{\displaystyle\frac{1}{p}}
\\[2pt]
\\
&&\geq\left|\displaystyle\frac{p-n}{p}\right|\left(\int_{B_R}\frac{|u|^p}{|x|^{(n-1)p'}
\left|R^m-|x|^m\right|^p}dx\right)^{\displaystyle\frac{1}{p}}
\\[2pt]
\\
&&+\displaystyle\frac{1}{p}R^{n-p}\hbox{limsup}_{r\rightarrow0}\left[r^{1-p}\int_{\partial
B_r}|u|^pdS\right]
\\[2pt]
\\
&&\times\left(\int_{B_R}\frac{|u|^p}{|x|^{(n-1)p'}
\left|R^m-|x|^m\right|^p}dx\right)^{-\displaystyle\frac{1}{p'}}.
\end{array}
\end{equation}
The constant $\left|\displaystyle\frac{p-n}{p}\right|$ in  (\ref{7eq161}) is  optimal.

iii) $\quad$ for $m=0$, i.e., $p=n$
\begin{equation}
\label{7eq17}
\begin{array}{lll}
&&\left(\displaystyle\int_{B_R}\left|\frac{<x,\nabla
u>}{|x|}\right|^ndx\right)^{\displaystyle\frac{1}{n}}
\geq\displaystyle\frac{n-1}{n}\left(\int_{B_R}\frac{|u|^n}{|x|^n
\left|\ln\displaystyle\frac{R}{|x|}\right|^n}dx\right)^{\displaystyle\frac{1}{n}}
\\[2pt]
\\
&+&\displaystyle\frac{1}{n}\hbox{limsup}_{r\rightarrow0}
\left[\left(r\ln\displaystyle\frac{R}{r}\right)^{1-n}\int_{\partial
B_r}|u|^ndS\right]\left(\int_{B_R}\frac{|u|^n}{|x|^n\left|\ln\displaystyle\frac{R}{|x|}\right|^n}dx\right)^{-\frac{1}{n'}}.
\end{array}
\end{equation}
The constant $\displaystyle\frac{n-1}{n}$ in  (\ref{7eq17}) is optimal.
\end{proposition}
\begin{proof}
Let us apply  (\ref{7eq12}) and (\ref{7eq13}) in Proposition \ref{3prop1n}. Then  after the limit  $r\rightarrow0$ we obtain (\ref{7eq16}), (\ref{7eq161}) and (\ref{7eq17}) for functions $u\in M(0,R)$. More precisely,

i) $\quad$ Inequality  (\ref{7eq16}) becomes equality for $u_k(x)=\left(\displaystyle\frac{R^m-|x|^m}{m}\right)^k, \ \
k>\displaystyle\frac{1}{p'}$ and hence the constant $\frac{p-n}{n}$ is optimal. Indeed, the sharpness of (\ref{7eq16}) is a consequence of the sharpness of (\ref{7eq12}) and the limit $r\rightarrow0$ since $u_k\in M(0,R)$ for $k>\frac{1}{p'}$.

ii) $\quad$ Note that for $m<0$ in the proof of (\ref{7eq161}) we use the identities
$$\begin{array}{lll}
&&\hbox{limsup}_{r\rightarrow0}\frac{1}{p}r^{1-n}|R^m-r^m|^{1-p}\int_{\partial B_r}|u|^pdS
\\[2pt]
\\
&&=\hbox{limsup}_{r\rightarrow0}\frac{1}{p}r^{1-n+m(1-p)}|R^mr^{-m}-1|^{1-p}\int_{\partial B_r}|u|^pdS
\\[2pt]
\\
&&=\hbox{limsup}_{r\rightarrow0}\frac{1}{p}r^{1-p}|R^mr^{-m}-1|^{1-p}\int_{\partial B_r}|u|^pdS=\hbox{limsup}_{r\rightarrow0}\frac{1}{p}r^{1-p}\int_{\partial B_r}|u|^pdS.
\end{array}
$$

We will  prove that the constant $\left|\displaystyle\frac{p-n}{p}\right|$ in (\ref{7eq161}) is optimal using the function $u_{\varepsilon}(x)=|x|^{-\frac{|m|}{p'}(1-\varepsilon)}\left(R^{|m|}-|x|^{|m|}\right)^{\frac{1}{p'}(1+\varepsilon)}$ for $0<\varepsilon<1$. Note that $u_{\varepsilon}(x)\in M(0,R)$ because for the power of $|x|$ we have
$$
\frac{-|m|}{p'}(1-\varepsilon)=\varepsilon\frac{|m|}{p'}+\frac{p-n}{p}=\varepsilon\frac{|m|}{p'}+\frac{n(p-1)}{p}+1-n>1-n,
$$
hence $u_{\varepsilon}(x)$ is integrable at $0$.

Ignoring the boundary term in (\ref{7eq161}) and rising both sides to $p$-th power for the left-hand side we obtain
$$
\begin{array}{lll}
(lhs)&=&\int_{B_R}\left|\frac{\langle x,\nabla u_\varepsilon\rangle}{|x|}\right|^pdx
\\[2pt]
\\
&&=\int_{B_R}|x|^{-|m|(p-1)(1-\varepsilon)-p}\left(R^{|m|}-|x|^{|m|}\right)^{(p-1)(1+\varepsilon)-p}
\\[2pt]
\\
&&\times\left|\frac{|m|}{p'}[(1-\varepsilon)R^{|m|}+2\varepsilon|x|^{|m|}]\right|^pdx
\\[2pt]
\\
&&\leq\left|\frac{p-n}{p}\right|^p(1+\varepsilon)^pR^{|m|p}
\\[2pt]
\\
&&\int_{B_R}|x|^{(p-n)(1-\varepsilon)-p}\left(R^{|m|}-|x|^{|m|}\right)^{(p-1)(1+\varepsilon)-p}dx.
\end{array}
$$
For the right-hand side we get
$$
\begin{array}{lll}
&&(rhs)
\\
&&=\left|\frac{p-n}{p}\right|^p\int_{B_R}\frac{|u_\varepsilon|^p}{|x|^{(n-1)p'}\left|R^m-|x|^m\right|^p}dx
\\[2pt]
\\
&&=\left|\frac{p-n}{p}\right|^p\int_{B_R}\frac{|x|^{-\frac{|m|}{p'}(1-\varepsilon)p}\left(R^{|m|}-|x|^{|m|}\right)^{\frac{p}{p'}(1+\varepsilon)-p}R^{|m|p}|x|^{|m|p}}{|x|^{(n-1)p'}}dx
\\[2pt]
\\
&&=\left|\frac{p-n}{p}\right|^p R^{|m|p}\int_{B_R}|x|^{-|m|(p-1)(1-\varepsilon)-p}\left(R^{|m|}-|x|^{|m|}\right)^{(p-1)(1+\varepsilon)-p}dx,
\end{array}
$$
and hence $1<\frac{(lhs)}{(rhs)}<(1+\varepsilon)^p$  because $\left(\frac{|m|}{p'}\right)^p=\left(\left|\frac{p-n}{p-1}\right|\frac{p-1}{p}\right)^p=\left|\frac{p-n}{p}\right|^p$. For $\varepsilon\rightarrow 0$ it follows that the constant $\left|\displaystyle\frac{p-n}{p}\right|^p$ is optimal.

(iii) $\quad$ We will  prove that the constant $\frac{n-1}{n}$ in (\ref{7eq17}) is optimal  using the function
 $$
u_s(x)=\left\{\begin{array}{l} \left(\ln\frac{R}{|x|}\right)^s, \ \ \hbox{ for } r_0<|x|<R,
\\[2pt]
\\
 \left(\ln\frac{R}{r_0}\right)^s, \ \ \hbox{ for } 0\leq|x|\leq r_0,
 \end{array}\right.
$$
with $s=\frac{n-1}{n}(1+\varepsilon)$, $0<\varepsilon$.

 Ignoring the boundary term in (\ref{7eq17}) and rising both sides to $n$-th power we obtain the inequality
\begin{equation}
\label{7eq177}
\int_{B_R}\left|\frac{\langle x,\nabla u_s\rangle}{|x|}\right|^ndx\geq\left(\frac{n-1}{n}\right)^n\int_{B_R}\frac{|u_s|^n}{|x|^n\left|\ln\frac{R}{|x|}\right|^n}dx.
\end{equation}

For the left-hand side and for the right-hand side of (\ref{7eq177}) as in Proposition \ref{3prop1n} we get
$$
\begin{array}{lll}
&&(lhs)=\int_{B_R}\left|\frac{\langle x,\nabla u_s\rangle}{|x|}\right|^ndx
=s^n\int_{B_R\backslash B_{r_0}}|x|^{-n}\left(\ln\frac{R}{|x|}\right)^{n(s-1)}dx
\\[2pt]
\\
&&=s^n\omega_n\int_{r_0/R}^1\rho^{-1}\left(\ln\frac{1}{\rho}\right)^{n(s-1)}=s^n\omega_n\left(\ln\frac{R}{r_0}\right)^{1-n+sn}(1-n+sn)^{-1}.
\end{array}
$$
$$
\begin{array}{lll}
&&(rhs)=\left(\frac{n-1}{n}\right)^n\int_{B_R}\frac{|u_s|^n}{|x|^n\left|\ln\frac{R}{|x|}\right|^n}dx
\\[2pt]
\\
&&=\left(\frac{n-1}{n}\right)^n\int_{B_R\backslash B_{r_0}}|x|^{-n}\left(\ln\frac{R}{|x|}\right)^{n(s-1)}dx
\\[2pt]
\\
&&+\left(\frac{n-1}{n}\right)^n\left(\ln\frac{R}{r_0}\right)^{sn}\int_{ B_{r_0}}|x|^{-n}\left(\ln\frac{R}{|x|}\right)^{n}dx
\\[2pt]
\\
&&=\left(\frac{n-1}{n}\right)^n\omega_n\left(\ln\frac{R}{r_0}\right)^{1-n+sn}(1-n+sn)^{-1}
\\[2pt]
\\
&&+\left(\frac{n-1}{n}\right)^n\left(\ln\frac{R}{r_0}\right)^{sn}\omega_n\int_{0}^{r_0/R}\rho^{-1}\left(\ln\frac{1}{\rho}\right)^{-n}d\rho
\\[2pt]
\\
&&=\left(\frac{n-1}{n}\right)^n\omega_n\left(\ln\frac{R}{r_0}\right)^{1-n+sn}\left[\frac{1}{1-n+sn}+\frac{1}{n-1}\right]
\\[2pt]
\\
&&=\left(\frac{n-1}{n}\right)^n\omega_n\left(\ln\frac{R}{r_0}\right)^{1-n+sn}\frac{sn}{(1-n+sn)(n-1)}.
\end{array}
$$
Since $s=(1+\varepsilon)\frac{n-1}{n}>\frac{1}{n'}$, then $1\leq\frac{(lhs)}{(rhs)}=(1+\varepsilon)^{n-1}$ and the sharpness of (\ref{7eq177}) is proved.
\end{proof}

\subsection{Hardy inequalities with additional logarithmic term}
\label{31sec3-3}

The aim of this section is to prove Hardy inequality in a ball $B_R$  centered at zero with radius $0<R<\infty$,  $B_R\subset R^n$, $n\geq 2$ with double singular weights on the boundary $\partial B_R$ and at the origin and with an additional logarithmic term. We generalize the results in \citet{BFT03a} where the weights are  singular only on the boundary of the domain.

Let $p>1$, $p'=\frac{p}{p-1}$, $n\geq2$, $m=\frac{p-n}{p-1}$.
In order to formulate the new  Hardy inequality, let us prove Lemma \ref{9lem2}  following the result of  \citet{BFT03a}, Lemma 3.1.
\begin{lemma}
\label{9lem2}
For every $p\geq2$, there exists $a=a(p)<0$ and $\tau_0>0$ such that for every $\tau>\tau_0$  the function
$$
Z(s)=\left(\frac{1}{p'}\right)^{p-1}\left(1-\frac{1}{1+\ln\tau-s}+\frac{a}{(1+\ln\tau-s)^2}\right),
$$
satisfies
\begin{equation}
\label{9eq4}
Z(s)\in C^1(-\infty,0), \ \ Z>0, \ \ Z'<0, \ \ Z(-\infty)=\left(\frac{1}{p'}\right)^{p-1},
\end{equation}
and is  a solution of the inequality
\begin{equation}
\label{9eq52}
-Z'+(p-1)Z-(p-1)Z^{p'}\geq H(s),
\end{equation}
where
$$
H(s)=\left(\frac{1}{p'}\right)^p\left(1+\frac{p}{2(p-1)}\frac{1}{(1+\ln\tau-s)^2}\right).
$$
\end{lemma}
\begin{proof}
Let us denote for simplicity $y(s)=\frac{1}{1+\ln\tau-s}$, so that
$$
Z(s)=\left(\frac{1}{p'}\right)^{p-1}(1-y+ay^2), \ \ Z'(s)=\left(\frac{1}{p'}\right)^{p-1}(-y^2+2ay^3)
$$
Expanding $Z^{p'}(y)$ for a small $y$ near $y=0$ in a Taylor polynomial up to  third order we obtain
$$
\begin{array}{lll}
Z^{p'}&=&\left(\frac{1}{p'}\right)^p\left\{1-\frac{p}{p-1}y+\frac{p}{p-1}\left(2a+\frac{1}{p-1}\right)
\frac{y^2}{2}\right.
\\[2pt]
\\
&+&\left.\frac{p}{p-1}\left[-\frac{6a}{p-1}-\frac{p-2}{(p-1)^2}\right]\frac{y^3}{6}+o(y^3)\right\}.
\end{array}
$$
Then if $a<-\frac{p-2}{6(p-1)}$  we get
$$
\begin{array}{lll}
&&-Z'+(p-1)Z-(p-1)Z^{p'}
\\[2pt]
\\
&&=\left(\frac{1}{p'}\right)^p\left[1+\frac{p}{2(p-1)}y^2+\frac{p}{p-1}\left(-a+\frac{p-2}{6(p-1)}\right)y^3+o(y^3)\right]
\\[2pt]
\\
&&\geq\left(\frac{1}{p'}\right)^p\left(1+\frac{p}{2(p-1)}y^2\right).
\end{array}
$$

With this choice of $a$ inequalities (\ref{9eq52}) and $Z'(s)<0$ hold. In order to satisfy the rest of the conditions in (\ref{9eq4}) we  choose $\tau$ such that $Z(s)>0$, i.e.,  $1-y-|a|y^2>0$. This means that
\begin{equation}
\label{9eq54}
0<y<y_0=\frac{1-\sqrt{1+4|a|}}{-2|a|}.
\end{equation}
Let $\tau_0=e^{\frac{1}{y_0}-1}$ then for every $\tau>\tau_0$ we get $Z(s)>0$.
\end{proof}
First, we will obtain an inequality in an annulus. Let us define   the vector function $f$,
\begin{equation}
\label{39eq54}
f=\left|\frac{\nabla\psi}{\psi}\right|^{p-2}\frac{\nabla\psi}{\psi}Z(\ln\psi) \ \ \hbox{ in } B_R\backslash B_r,
\end{equation}
where $\psi(x)$ is defined in (\ref{39eq88}) and  $Z$ is given in  Lemma \ref{9lem2}.
\begin{proposition}
\label{9prop1}
The vector function $f=\{f_1,\ldots,f_n\}$ in (\ref{39eq54}) satisfies
$f_j\in C^1(B_R\backslash B_r)$ and
\begin{equation}
\label{9eq8-1}
-\hbox{div}f-(p-1)|f|^{p'}\geq w, \ \ \hbox{ in } B_R\backslash B_r,
\end{equation}
where $w=\left|\frac{\nabla\psi}{\psi}\right|^pH(\ln\psi)$  and $H(s)$ is defined in Lemma \ref{9lem2}.

Moreover, for every $u\in
W_0^{1,p}(B_R)$, the following inequality holds
\begin{equation}
\label{9eq2}
L(u)\geq N(u),
\end{equation}
where
%
$$
L(u)=\int_{B_R\backslash B_r}
\left|\frac{\langle f,\nabla u\rangle}{|f|}\right|^pdx,
$$
%
and
\begin{itemize}
\item $\quad$ for $m\neq0$, i.e., $p\neq n$
\begin{equation}
\label{9eq12-1}\begin{array}{lll}
N(u)&=&\int_{B_R\backslash B_r}
w|u|^pdx
\\[2pt]
\\
&=&\left|\frac{p-n}{p}\right|^p\int_{B_R\backslash B_r}
\left[1+\frac{p}{2(p-1)}\frac{1}{\ln^2\frac{\psi}{e\tau}}\right]
\frac{|u|^p}{|x|^{(n-1)p'}|R^m-|x|^m|^p}dx
\end{array}
\end{equation}
\item $\quad$ for $m=0$, i.e., $p=n$
\begin{equation}
\label{9eq13-1}\begin{array}{lll}
N(u)&=&\int_{B_R\backslash B_r}
w|u|^ndx
\\[2pt]
\\
&=&
\left(\frac{n-1}{n}\right)^n\int_{B_R\backslash B_r}\left[1+\frac{n}{2(n-1)}\frac{1}{\ln^2\frac{\psi}{e\tau}}\right]
\frac{|u|^n}{|x|^n\left|\ln\frac{R}{|x|}\right|^n}dx,
\end{array}
\end{equation}
\end{itemize}
where $\tau>\tau_0=e^{\frac{1}{y_0}-1}$ and $y_0$ is defined in (\ref{9eq54}).
\end{proposition}
\begin{proof}
Let us check that the function $f$ satisfies (\ref{9eq8-1}).
Indeed,
$$
\begin{array}{lll}
-\hbox{div}f&=&-\left(\frac{\Delta_p\psi}{|\psi|^{p-1}}-(p-1)
\left|\frac{\nabla\psi}{\psi}\right|^p\right)Z(\ln\psi)-\left|\frac{\nabla\psi}{\psi}\right|^p
Z'(\ln\psi)
\\[2pt]
\\
&=&\left|\frac{\nabla\psi}{\psi}\right|^p\left[-Z'+(p-1)Z-(p-1)Z^{p'}\right]
+(p-1)\left|\frac{\nabla\psi}{\psi}\right|^pZ^{p'}
\\[2pt]
\\
&\geq&(p-1)|f|^{p'}+\left|\frac{\nabla\psi}{\psi}\right|^pH(\ln\psi).
\end{array}
$$
Since   $f$ as a function of $x$ has the form
$$
f(x)=\left\{\begin{array}{l}-x|x|^{-n}\left(\displaystyle\frac{R^m-|x|^m}{m}\right)^{1-p}Z(\ln\psi),\
\ m\neq 0,
\\[2pt]
\\
-x|x|^{-n}\left(\ln\displaystyle\frac{R}{|x|}\right)^{1-n}Z(\ln\psi), \ \
m=0,\end{array}\right.
$$
then $\langle f,\eta\rangle\left|_{\partial B_r}\right.>0$. Ignoring the boundary term over $\partial B_r$ we can apply  (\ref{300eq2}) in the domain $B_{\hat{R}}\backslash B_r$, $r<\hat{R}<R$ for $v=1$ and $w=\left|\frac{\nabla\psi}{\psi}\right|^pH(\ln\psi)$ to obtain (\ref{9eq2}) after the limit $\hat{R}\rightarrow R$.
\end{proof}

The inequality with additional logarithmic weight in a ball can be obtained only for $p>n$, i.e., $m>0$. In the case $m\leq0$ the function $\psi(x)$ defined in (\ref{39eq88}) satisfies $$
\psi(x)=\frac{r^{|m|}(R^{|m|}-|x|^{|m|})}{|x|^{|m|}(R^{|m|}-r^{|m|})}\rightarrow 0, \ \ \hbox{ and } \ln\frac{R}{|x|}/\ln\frac{R}{r}\rightarrow 0  \ \ \hbox{ for } r\rightarrow 0,
$$
and the inequalities with $ N(u)$ defined in   (\ref{9eq12-1}),   (\ref{9eq13-1}) are the same as the inequalities (\ref{7eq161}) and (\ref{7eq17}) without boundary and logarithmic terms.
When $m>0$, i.e., $p>n$,  the function  $\psi(x)=\frac{R^m-|x|^m}{R^m-r^m}\rightarrow \frac{R^m-|x|^m}{R^m}$ for $r\rightarrow 0$ so we  obtain a new Hardy inequality with additional logarithmic term  using the expression  (\ref{9eq12-1}) for $N(u)$.

\begin{proposition}
\label{39prop2}
For $m>0$, i.e., $p>n$,  the following inequality holds for every $u\in W^{1,p}_0(B_R)$
\begin{equation}
\label{9eq14}
\begin{array}{lll}
&&\int_{B_R}\left|\nabla u\right|^pdx\geq\int_{B_R}\left|\frac{\langle x, \nabla u\rangle}{|x|}\right|^pdx
\\[2pt]
\\
&&\geq\left(\frac{p-n}{p}\right)^p\int_{B_R}
\left[1+\frac{p}{2(p-1)}\frac{1}{\ln^2\frac{R^m-|x|^m}{e\tau_0 R^m}}\right]
\frac{|u|^p}{|x|^{(n-1)p'}|R^m-|x|^m|^p}dx,
\end{array}
\end{equation}
where $\tau_0=e^{\frac{1}{y_0}-1}$, $y_0$ is defined in (\ref{9eq54}) with $a=-\frac{p-2}{6(p-1)}$.
\end{proposition}
\begin{proof}
Since the function $\psi(x)=\frac{R^m-|x|^m}{R^m-r^m}\rightarrow \frac{R^m-|x|^m}{R^m}$ for $r\rightarrow 0$ and after the limit $r\rightarrow 0$ in $N(u)$ from  the expression  (\ref{9eq12-1}) we obtain (\ref{9eq14}). By continuity we  get $\tau=\tau_0$ in (\ref{9eq12-1})
\end{proof}
\subsection{One-parametric family of Hardy inequalities}
\label{310sec4}

For $n\geq2$, $p>1$ we consider the Poisson problem in $B_R$
\begin{equation}
\label{10eq4} \left\{\begin{array}{l} -\hbox{div}(|\nabla
\phi|^{p-2}\nabla \phi)=w(|x|) \ \ \hbox{ in } B_R,
\\[2pt]
\\
\phi=0 \ \ \hbox{ on } \partial B_R,
\end{array}\right.
\end{equation}
where
\begin{equation}
\label{10eq5}
\int_0^R s^{n-1}w(s)ds<\infty.
\end{equation}
We  choose a function $w(|x|)$  such that the function $\phi$ has a simple form.

For this purpose let us apply the result in \citet{BBEM12}, where it is shown that
the solution of (\ref{10eq4}),  (\ref{10eq5}) is given by
\begin{equation}
\label{10eq6}
\phi(|x|)=\int_{|x|}^R\theta^{\frac{1-n}{p-1}}\left(\int_0^\theta s^{n-1}w(s)ds\right)^{\frac{1}{p-1}}d\theta.
\end{equation}

Indeed, from the invariance of (\ref{10eq4})  under rotation, problem (\ref{10eq4}) is equivalent to the boundary value problem for ordinary differential equation
\begin{equation}
\label{10eq7} \left\{\begin{array}{l}-(r^{n-1}|\phi'|^{p-2}\phi')'=r^{n-1}w(r), \ \ 0<r<R,
\\[2pt]
\\
\phi(R)=0, \phi'(0)=0.
\end{array}\right.
\end{equation}
Integrating twice the equation in (\ref{10eq7}) and applying boundary conditions we obtain (\ref{10eq6}).

Let us chose the function $w(|x|)=|x|^{-\delta}$, $\delta\in(0,n)$, then (\ref{10eq5}) holds and from (\ref{10eq6}) we have
$$
 \phi(|x|)=\left\{\begin{array}{l} \frac{p-1}{p-\delta}(n-\delta)^{-\frac{1}{p-1}}\left(R^{\frac{p-\delta}{p-1}}-|x|^{\frac{p-\delta}{p-1}}\right) \ \ \hbox{ for } \delta\neq p,
\\[2pt]
\\
(n-p)^{-\frac{1}{p-1}}\ln\frac{R}{|x|} \ \ \hbox{ for } \delta=p.
\end{array}\right.
$$

For $\varepsilon\in(0,R)$ we define the vector function
$f=\frac{|\nabla\phi|^{p-2}\nabla\phi}{|\phi|^{p-2}\phi}$.
Then $f\in C^1(\overline{B_R\backslash B_{\varepsilon}})$, $\varepsilon<{\hat{R}}<R$ and $f$ satisfies the equation
$$
-\hbox{div}f=-\displaystyle\frac{\Delta_p\phi}{|\phi|^{p-2}\phi}+
(p-1)\displaystyle\frac{|\nabla\phi|^p}{|\phi|^p}=\frac{w(|x|)}{|\phi|^{p-1}}+(p-1)|f|^{\frac{p}{p-1}} \ \ \hbox{ in }B_{R}\backslash B_\varepsilon.
$$
 According  to Corollary  \ref{cor300},   the following Hardy inequality holds for $u\in C_0^\infty(B_R)$, with $\hbox{supp }u\subset B_{\hat{R}}$
\begin{equation}
\label{10eq8}
L_\varepsilon(u)\geq\left(\frac{1}{p}\right)^p\frac{\left|(p-1)K_\varepsilon(u)+K_{3\varepsilon}(u)+N_\varepsilon(u)\right|^p}{K_\varepsilon^{p-1}(u)},
\ \ u\in C^{\infty}_0(B_R),
\end{equation}
where
$$
K_{3\varepsilon}(u)=\int_{S_\varepsilon}\langle f, \eta\rangle|u|^pdS+\int_{\partial B_{\hat{R}}}\langle f, \eta\rangle|u|^pdS=\int_{S_\varepsilon}\langle f, \eta\rangle|u|^pdS.
$$
Here $\eta$ is the unit outward normal vector to $\partial(B_{\hat{R}}\backslash B_\varepsilon)$. The expressions for $L_\varepsilon(u)$,  $K_\varepsilon(u)$ and $N_\varepsilon(u)$  in (\ref{10eq8}) are correspondingly:
\begin{itemize}
\item $\quad$ in the case $\delta\neq p$
\begin{equation}
\label{10eq9}
\begin{array}{lll}
&&\int_{B_R\backslash B_\varepsilon}|\nabla u|^pdx\geq L_\varepsilon(u)=\int_{B_{\hat{R}}\backslash B_\varepsilon}\left|\frac{\langle\nabla\phi,
\nabla u\rangle}{|\nabla\phi|}\right|^pdx=\int_{B_R\backslash B_\varepsilon}\left|\frac{\langle x,
\nabla u\rangle}{|x|}\right|^pdx,
\\[2pt]
\\
&&K_\varepsilon(u)=\int_{B_{\hat{R}}\backslash B_\varepsilon}\left|\frac{\nabla \phi}{\phi}\right|^p|u|^pdx
=\left|\frac{p-\delta}{p-1}\right|^p\int_{B_R\backslash B_\varepsilon}\frac{|u|^p}{|x|^{\frac{(\delta-1)p}{p-1}}\left|R^{\frac{p-\delta}{p-1}}-|x|^{\frac{p-\delta}{p-1}}\right|^p}dx,
\\[2pt]
\\
&&N_\varepsilon(u)=\int_{B_{\hat{R}}\backslash B_\varepsilon)}\frac{w(|x|)}{|\phi|^{p-1}}|u|^pdx
\\[2pt]
\\
&&=(n-\delta)\left|\frac{p-\delta}{p-1}\right|^{p-1}\int_{B_R\backslash B_\varepsilon}\frac{|u|^p}{|x|^{\delta}
\left|R^{\frac{p-\delta}{p-1}}-|x|^{\frac{p-\delta}{p-1}}\right|^{p-1}}dx,
\end{array}
\end{equation}
\item $\quad$ in the case  $\delta= p$
\begin{equation}
\label{10eq14-1}
 K_\varepsilon(u)=\int_{B_R\backslash B_\varepsilon}\frac{|u|^p}{|x|^{p}\left|\ln\frac{R}{|x|}\right|^p}dx,
\ \
N_\varepsilon(u)=|n-p|\int_{B_R\backslash B_\varepsilon}\frac{|u|^p}{|x|^{p}\left|\ln\frac{R}{|x|}\right|^{p-1}}dx.
\end{equation}
\end{itemize}

Since
$$
\nabla\phi=-(n-\delta)^{-\frac{1}{p-1}}|x|^{\frac{1-\delta}{p-1}}\frac{x}{|x|}, \ \ \eta\left|_{S_\varepsilon}\right.=-\frac{x}{\varepsilon},
$$
and
$$
\langle\nabla\phi,\eta\rangle=(n-\delta)^{-\frac{1}{p-1}}|\varepsilon|^{\frac{1-\delta}{p-1}}\geq0, \ \ \hbox{for } x\in S_\varepsilon,
$$
we get
$$
\int_{S_\varepsilon}\langle f,\eta\rangle|u|^pdS=(n-\delta)^{-\frac{1}{p-1}}\varepsilon^{\frac{1-\delta}{p-1}}\int_{S_\varepsilon}|u|^pdS\geq0.
$$
Hence, neglecting $K_{3\varepsilon}$ (\ref{10eq8}) becomes
$$
L_\varepsilon(u)\geq\left(\frac{1}{p}\right)^p\frac{[(p-1)K_\varepsilon(u)+N_\varepsilon(u)]^p}{K_\varepsilon^{p-1}(u)},
\ \ u\in C^{\infty}_0(B_R).
$$
After the limit $\varepsilon\rightarrow0$ inequality
\begin{equation}
\label{100eq88}
L(u)\geq\left(\frac{1}{p}\right)^p\frac{[(p-1)K(u)+N(u)]^p}{K^{p-1}(u)},
\ \ u\in C^{\infty}_0(B_R).
\end{equation}
 holds in $B_R$ with $L(u)$, $K(u)$ and $N(u)$ defined in (\ref{10eq9})  and (\ref{10eq14-1}) for $\varepsilon=0$, i.e.,
\begin{itemize}
\item $\quad$ in the case $\delta\neq p$
\begin{equation}
\label{100eq9}
\begin{array}{l}
 L(u)=\int_{B_R}\left|\frac{\langle x,
\nabla u\rangle}{|x|}\right|^pdx,
\\[2pt]
\\
K(u)=\left|\frac{p-\delta}{p-1}\right|^p\int_{B_R}\frac{|u|^p}{|x|^{\frac{(\delta-1)p}{p-1}}\left|R^{\frac{p-\delta}{p-1}}-|x|^{\frac{p-\delta}{p-1}}\right|^p}dx,
\\[2pt]
\\
 N(u)=(n-\delta)\left|\frac{p-\delta}{p-1}\right|^{p-1}\int_{B_R}\frac{|u|^p}{|x|^{\delta}
\left|R^{\frac{p-\delta}{p-1}}-|x|^{\frac{p-\delta}{p-1}}\right|^{p-1}}dx;
\end{array}
\end{equation}
\item $\quad$ in the case  $\delta= p$
\begin{equation}
\label{100eq14-1}
 K(u)=\int_{B_R}\frac{|u|^p}{|x|^{p}\left|\ln\frac{R}{|x|}\right|^p}dx,
\ \
N(u)=|n-p|\int_{B_R}\frac{|u|^p}{|x|^{p}\left|\ln\frac{R}{|x|}\right|^{p-1}}dx.
\end{equation}
\end{itemize}

From  (\ref{3eq15}) and (\ref{100eq88}) we get
\begin{equation}
\label{10eq11} L(u)\geq(p-1)s^{p-1}(1-s)K(u)+ s^{p-1}N(u).
\end{equation}
In particular, for $s=\frac{p-1}{p}$ in (\ref{10eq11})   a `linear' form of Hardy inequality holds.
\begin{lemma}
\label{7lemma00}\rm
The inequalities  $K(u)<\infty$ and $N(u)<\infty$ for  $K(u)$ and $N(u)$ defined in (\ref{100eq9}) and (\ref{100eq14-1}) hold for every $u\in C^\infty_0(B_R)$.
\end{lemma}
\begin{proof}
For $u\in C^\infty_0(B_R)$,
$$
d_1=\hbox{dist }\left(\hbox{ supp }u, \partial B_R\right)>0,  \ \ d_2=R^{\left|\frac{p-\delta}{p-1}\right|}-(R-d_1)^{\left|\frac{p-\delta}{p-1}\right|}>0,
$$
we obtain these statements  from the following estimates
\begin{itemize}
\item $\quad$ for $0<\delta<p$, $\delta<n$
$$
\begin{array}{lll}
&&\int_{B_R}\frac{|u|^p}{|x|^{\frac{(\delta-1)p}{p-1}}\left|R^{\frac{p-\delta}{p-1}}-|x|^{\frac{p-\delta}{p-1}}\right|^p}dx\leq\left(\frac{\sup|u|}{d_2}\right)^p
\int_{S_1}\int_0^{R-d_1}\frac{\rho^{n-1}}{\rho^\frac{(\delta-1)p}{p-1}}d\rho d\theta
\\[2pt]
\\
&&=\omega_n\left(\frac{\sup|u|}{d_2}\right)^p\left(n-\delta+\frac{p-\delta}{p-1}\right)^{-1}(R-d_1)^{n-\delta+\frac{p-\delta}{p-1}}<\infty
\end{array}
$$
where $\omega_n=\hbox{ meas }S_1$,
$$
\begin{array}{lll}
&&\int_{B_R}\frac{|u|^p}{|x|^\delta\left|R^{\frac{p-\delta}{p-1}}-|x|^{\frac{p-\delta}{p-1}}\right|^{p-1}}dx\leq\frac{(\sup|u|)^p}{d_2^{p-1}}
\int_{S_1}\int_0^{R-d_1})\frac{\rho^{n-1}}{\rho^\delta}d\rho d\theta
\\[2pt]
\\
&&=\omega_n\left(\frac{\sup|u|}{d_2}\right)^p\left(n-\delta\right)^{-1}(R-d_1)^{n-\delta}<\infty;
\end{array}
$$
\item $\quad$ for $1<p<\delta<n$
$$
\begin{array}{lll}
&&\int_{B_R}\frac{|u|^p}{|x|^{\frac{(\delta-1)p}{p-1}}\left|R^{\frac{p-\delta}{p-1}}-|x|^{\frac{p-\delta}{p-1}}\right|^p}dx=R^{\frac{(\delta-p)p}{p-1}}
\int_{B_R}\frac{|u|^p}{|x|^p\left|R^{\frac{\delta-p}{p-1}}-|x|^{\frac{\delta-p}{p-1}}\right|^p}dx
\\[2pt]
\\
&&\leq\frac{\omega_n}{n-p}\left(\frac{\sup|u|}{d_2}\right)^p(R-d_1)^{n-p}R^{\frac{(\delta-p)p}{p-1}}<\infty,
\end{array}
$$
$$
\begin{array}{lll}
&&\int_{B_R}\frac{|u|^p}{|x|^\delta\left|R^{\frac{p-\delta}{p-1}}-|x|^{\frac{p-\delta}{p-1}}\right|^{p-1}}dx=R^{\frac{(\delta-p)p}{p-1}}
\int_{B_R}\frac{|u|^p}{|x|^p\left|R^{\frac{\delta-p}{p-1}}-|x|^{\frac{\delta-p}{p-1}}\right|^{p-1}}dx
\\[2pt]
\\
&&\leq\frac{\omega_n}{n-p}\frac{(\sup|u|)^p}{d_2^{p-1}}(R-d_1)^{n-p}R^{\frac{(\delta-p)p}{p-1}}<\infty;
\end{array}
$$
\item $\quad$ for $p=\delta<n$
$$\begin{array}{lll}
&&\int_{B_R}\frac{|u|^p}{|x|^p\left|\ln\frac{R}{|x|}\right|^p}dx\leq\left(\sup|u|\right)^p\frac{\omega_n}{\left(\ln\frac{R}{R-d_1}\right)^p}\int_0^{R-d_1}\rho^{n-p-1}d\rho
\\[2pt]
\\
&&=\left(\sup|u|\right)^p\frac{\omega_n}{\left(\ln\frac{R}{R-d_1}\right)^p}\frac{R^{n-p}}{n-p}<\infty,
\end{array}
$$
$$
\begin{array}{lll}
&&\int_{B_R}\frac{|u|^p}{|x|^p\left|\ln\frac{R}{|x|}\right|^{p-1}}dx\leq\left(\sup|u|\right)^p\frac{\omega_n}{\left(\ln\frac{R}{R-d_1}\right)^{p-1}}\int_0^{R-d_1}\rho^{n-p-1}d\rho
\\[2pt]
\\
&&=\left(\sup|u|\right)^p\frac{\omega_n}{\left(\ln\frac{R}{R-d_1}\right)^{p-1}}\frac{(R-d_1)^{n-p}}{n-p}<\infty.
\end{array}
$$
\end{itemize}
\end{proof}

\begin{proposition}
\label{10lem1}
For $\delta\in(0,n)$ and all functions $u\in W_0^{1,p}(B_R)$ the following Hardy inequalities hold:

(i) $\quad$ for $\delta\neq p$, $p<n$
\begin{equation}
\label{10eq15}
\begin{array}{lll}
\int_{B_R}|\nabla u|^pdx&\geq&\left|\frac{p-\delta}{p}\right|^p\int_{B_R}\frac{|u|^p}{|x|^{\frac{(\delta-1)p}{p-1}}\left|R^{\frac{p-\delta}{p-1}}-|x|^{\frac{p-\delta}{p-1}}\right|^p}dx
\\[2pt]
\\
&+&(n-\delta)\left|\frac{p-\delta}{p}\right|^{p-1}\int_{B_R}\frac{|u|^p}{|x|^{\delta}
\left|R^{\frac{p-\delta}{p-1}}-|x|^{\frac{p-\delta}{p-1}}\right|^{p-1}}dx.
\end{array}
\end{equation}
By arguments of continuity   inequality (\ref{10eq15}) is also true in the case of $\delta=n$ and   $\delta=0$.

(ii) $\quad$ for $\delta= p$
\begin{equation}
\label{10eq15-1}
\begin{array}{lll}
\int_{B_R}|\nabla u|^pdx&\geq&\left(\frac{p-1}{p}\right)^p\int_{B_R}\frac{|u|^p}{|x|^{p}\left|\ln\frac{R}{|x|}\right|^p}dx
\\[2pt]
\\
&+&\left(\frac{p-1}{p}\right)^{p-1}|n-p|\int_{B_R}\frac{|u|^p}{|x|^{p}\left|\ln\frac{R}{|x|}\right|^{p-1}}dx.
\end{array}
\end{equation}
\end{proposition}
\begin{proof}
With the expressions (\ref{100eq9}), (\ref{100eq14-1}), applying (\ref{100eq88}) we obtain (i) and (ii).
\end{proof}

It is important to mention that in the new Hardy inequalities (\ref{10eq15}) and (\ref{10eq15-1}) for $\delta\in(1,n)$  the  constants $\left|\frac{p-\delta}{p}\right|^p$ and $\left(\frac{p-1}{p}\right)^p$, correspondingly, at their leading terms in the right-hand side are optimal. This follows from Proposition \ref{10lem2} below.
\begin{proposition}
\label{10lem2}
If $p>1$, $n\geq2$, $1<\delta<n$ then for $0<|x|<R$ the following inequalities hold:
\begin{itemize}
\item[(i)] $\quad$ for $\delta\neq p$
\begin{equation}
\label{10eq15-2}
\left|\frac{p-\delta}{p}\right|^p|x|^{\frac{(1-\delta)p}{p-1}}\left|R^{\frac{p-\delta}{p-1}}-|x|^{\frac{p-\delta}{p-1}}\right|^{-p}\geq\left(\frac{p-1}{p}\right)^p(R-|x|)^{-p};
\end{equation}
\item[(ii)] $\quad$ for $\delta=p$
\begin{equation}
\label{10eq15-3}
\left(\frac{p-1}{p}\right)^p|x|^{-p}\left|\ln\frac{R}{|x|}\right|^{-p}\geq\left(\frac{p-1}{p}\right)^p(R-|x|)^{-p}.
\end{equation}
\end{itemize}
\end{proposition}
\begin{proof}
(i)$\quad$ For $p>\delta$ inequality (\ref{10eq15-2}) is equivalent to the estimate
\begin{equation}
\label{10eq15-4}
g(r)=(p-\delta)(R-r)-(p-1)r^{\frac{\delta-1}{p-1}}\left(R^{\frac{p-\delta}{p-1}}-r^{\frac{p-\delta}{p-1}}\right)\geq0, \ \ \hbox{ for } 0\leq r\leq R.
\end{equation}
Since $g(R)=0$ and
$$
g'(r)=-(\delta-1)r^{\frac{\delta-p}{p-1}}\left(R^{\frac{p-\delta}{p-1}}-r^{\frac{p-\delta}{p-1}}\right)\leq0
$$
inequality (\ref{10eq15-4}) follows from the monotonicity of $g(r)$.

For $1<p<\delta<n$ estimate (\ref{10eq15-2}) is equivalent to
$$
\left(\frac{\delta-p}{p}\right)^pr^{\frac{(1-\delta)p}{p-1}}r^{\frac{(\delta-p)p}{p-1}}R^{\frac{(\delta-p)p}{p-1}}
\left(R^{\frac{\delta-p}{p-1}}-r^{\frac{\delta-p}{p-1}}\right)^{-p}
\geq\left(\frac{p-1}{p}\right)^p(R-r)^{-p},
$$
or
$$
h(r)=(\delta-p)R^{\frac{\delta-p}{p-1}}(R-r)-(p-1)r\left(R^{\frac{\delta-p}{p-1}}-r^{\frac{\delta-p}{p-1}}\right)\geq0, \ \ \hbox{ for } 0\leq r\leq R.
$$
Since $h(R)=0$ and
$$
h'(r)=-(\delta-1)\left(R^{\frac{\delta-p}{p-1}}-r^{\frac{\delta-p}{p-1}}\right)\leq0,
$$
inequality (\ref{10eq15-2}) follows from the monotonicity of $h(r)$.

(ii)$\quad$ Inequality (\ref{10eq15-3}) is equivalent to
\begin{equation}
\label{10eq15-7}
z(r)=R-r-r\ln\frac{R}{r}\geq0, \ \ \hbox{ for } 0< r< R.
\end{equation}
Since $z(R)=0$ and $z'(r)=-\ln\frac{R}{r}<0$, for $r\in(0,R)$, the  inequality (\ref{10eq15-7}) is a consequence of the monotonicity of $z(r)$.
\end{proof}
However, for  Hardy inequality
$$
\int_{B_R}|\nabla u|^pdx\geq\left(\frac{p-1}{p}\right)^p\int_{B_R}\frac{|u|^p}{(R-|x|)^p}dx,
$$
the constant $\left(\frac{p-1}{p}\right)^p$ is optimal, see (\ref{2eq2}) with $\alpha=0$ and $\Omega=B_R$. The same optimality is also true for (\ref{10eq15}) and (\ref{10eq15-1}).

\section{General Hardy inequalities with optimal constant}
\label{sect4}
In this section we prove general Hardy  inequalities with singular at zero and on
the boundary $\partial\Omega$ weights is proved.
The   Hardy constant is optimal when  $\Omega$ is a ball. The section is based on
\citet{FKR13}.

 Let $\Omega\subset R^n$, $n\geq 2$ be a bounded domain. Suppose that
\begin{equation}
\label{4eq1} \begin{array}{l}\Omega\subset\Omega^{\ast} \hbox{ and there exists a positive function }
\lambda\in
C^{0,1}(\Omega^{\ast}), \lambda(x)>0,
\\[2pt]
\\
 \hbox{ such that }
 |x|<\lambda(x)  \hbox{
and } \langle x,\nabla\lambda\rangle\leq0 \hbox{ for a.e. } x\in\Omega^{\ast},
\end{array}
\end{equation}
where $\langle , \rangle$ denotes the scalar product in $R^n$. If $\Omega$ is convex or
star-shaped domain with respect to some interior ball centered at
zero, then $\Omega$ satisfies (\ref{4eq1}),  see Section 1.1.8 in
\citet{Ma85}, so one can take $\Omega^{\ast}=\Omega$. When $\Omega$
is an arbitrary domain, then its star-shaped envelope  with respect
to some fixed interior ball (i.e., the intersection of all
star-shaped domains with respect to the fixed ball containing
$\Omega$) can be taken as $\Omega^{\ast}$. Further on, we suppose that $0\in\Omega$.

For  $\alpha, \beta \in R$, $p>1$, $\beta p>1$, $\alpha p\leq
n+\beta p-1$ and $k=\frac{n-\alpha p}{\beta p-1}$,
we define the function
\begin{equation}
\label{44eq00}
g(s)=\left\{\begin{array}{l}\frac{1-s^k}{k} \hbox{ for } k\neq0
\\[2pt]
\\
\ln\frac{1}{s} \hbox{ for } k=0\end{array}\right.
\end{equation}
and
the constant $\gamma=\frac{\beta p-1}{p}$. Note that
$\gamma>0$  and $k\geq -1$.  For
$s(x)=\displaystyle\frac{|x|}{\lambda(x)}$ let us consider the
non-negative weights
\begin{equation}
\label{4eq2} v(x)=|x|^{1-\alpha}|g(s(x))|^{1-\beta}, \ \
w(x)=|x|^{-\alpha}|g(s(x))|^{-\beta}.
\end{equation}
The function $v$
is singular at the origin  when $k\geq0$  for $\alpha>1$ and when $k\in[-1,0)$ for $\alpha>k(\beta-1)+1$. The function $w(x)$  is singular at $0$ when $k\geq0$ for $\alpha>0$ and when $k\in[-1,0)$ for $\alpha>-\beta k$.  Moreover, $v$ and $w$ are singular on the whole boundary if
$\partial\Omega^{\ast}=\partial\Omega=\{x: |x|=\lambda(x)\}$.  Otherwise, if $\partial\Omega^{\ast}\neq\partial\Omega$,  the functions $v$ and $w$ are singular only on a part of the boundary.

Let the space $W^{1,p}_{0,v}(\Omega)$ be
the completion of $C^{\infty}_0(\Omega)$ functions with respect to
the norm $\left(\displaystyle\int_{\Omega}v^p|\nabla
u|^pdx\right)^{1/p}<\infty$, see Lemma \ref{4le1}.

The aim of this section is to prove the following new Hardy inequality with
double singular weights.
\begin{theorem}
\label{4th1} For every $u\in W^{1,p}_{0,v}(\Omega)$ the following
inequality holds
\begin{equation}
\label{4eq3} \displaystyle\int_{\Omega}v^p|\nabla
u|^pdx\geq\gamma^p\displaystyle\int_{\Omega}w^p|u|^pdx,
\end{equation}
where $\gamma=\displaystyle\frac{\beta p-1}{p}$.
\end{theorem}
At the beginning, let us  analyze condition  (\ref{4eq1}) and
simplify it in polar coordinates. If $S_1=\{x\in R^n:|x|=1\}$ is the unit sphere and $\rho=|x|$,
$\theta=\frac{x}{|x|}\in S_1$,  the property of
$\lambda(x)$, $\langle x,\nabla\lambda\rangle\leq0$  in (\ref{4eq1}) becomes $\rho\lambda_{\rho}\leq0$, where
$\lambda=\lambda(\rho,\theta)$.

Indeed, $\langle x,\nabla\lambda\rangle=\rho\lambda_{\rho}$ because
$$
\frac{\partial\lambda}{\partial
x_j}=\frac{x_j}{\rho}\lambda_{\rho}+\frac{\partial\lambda}{\partial\theta_k}
\left(\frac{\delta_{jk}}{\rho}-\frac{x_jx_k}{\rho^3}\right),
$$
and
$$
\sum_{l,j=1}^n\frac{\partial\lambda}{\partial\theta_l}x_j\left(\frac{\partial\delta_{jl}}{\rho}-\frac{x_jx_l}{\rho^3}\right)=0.
$$
The proof of the Theorem  is based on Theorem \ref{3th1} and Lemma \ref{4le1} which
clarifies the properties of the weights $w$ and $v$ and the definition
of the weighted Sobolev space $W^{1,p}_{0,v}(\Omega)$.

\begin{lemma}
\label{4le1} Functions $w(x)$ and $v(x)$ belong to
$L^p_{loc}(\Omega)$.
\end{lemma}
\begin{proof}
Since the functions $v$ and $w$ are singular  at $0$ and on $\partial\Omega$, it is enough to prove Lemma \ref{4le1} only in a small ball containing $0$. For this purpose
let us fix $\varepsilon\in(0,1)$ and we will  prove that $w(x)$
is an $L^p$ function in a small ball $B_{\sigma}$ with radius $\sigma\in(0,1)$
centered at zero, $B_{\sigma}\subset\Omega$, so that
\begin{equation}
\label{44eq0}
\frac{\rho}{\lambda}<1-\varepsilon
\end{equation}
for $\rho<\sigma$.

In polar coordinates, for $k\neq0$ we have
$$
\int_{B_{\sigma}}w^pdx=\int_{S_1}\int_0^{\sigma}H(\rho)d\rho
d\theta, \quad \hbox{where} \quad H(\rho)=\rho^{n-1-\alpha
p}\left(\frac{1-(\frac{\rho}{\lambda})^k}{k}\right)^{-\beta
p}.
$$
We define the constant
$C_{k,\varepsilon}=\left(\displaystyle\frac{1-(1-\varepsilon)^{k}}{k}\right)^{-\beta
p}\omega_n$, where $\omega_n$ is the measure of the unit sphere $S_1$ in $R^n$, and, let us consider all  possibilities for the sign of
$k$:
\begin{itemize}
\item[(a)]$\quad$  $k>0$, i.e., $n>\alpha p$. Then from
$$
\frac{\partial}{\partial\rho}\left(\frac{1-\left(\frac{\rho}{\lambda}\right)^k}{k}\right)^{-\beta p}>0
$$
we get
$$
\int_{S_1}\int_0^{\sigma}H(\rho)d\rho d\theta\leq\int_{S_1}\int_0^{\sigma}\frac{C_{k,\varepsilon}}{\omega_n}\rho^{n-1-\alpha p}d\rho d\theta=
C_{k,\varepsilon}\frac{\sigma^{n-\alpha p}}{n-\alpha p}<\infty.
$$
\item[(b)]$\quad$  $k<0$, i.e., $n<\alpha p$.  In this case
$$
\begin{array}{lll}
&& H(\rho)=\left(\frac{\rho}{\lambda}\right)^{-k\beta p}\rho^{n-1-\alpha p}\left(\frac{1-(\frac{\rho}{\lambda})^{|k|}}{|k|}\right)^{-\beta
p}
\\[2pt]
\\
&&=\lambda^{k\beta p}\rho^{\frac{\alpha
p-n}{\beta
p-1}-1}\left(\frac{1-(\frac{\rho}{\lambda})^{|k|}}{|k|}\right)^{-\beta
p},
\end{array}
$$
and then
$$
\int_{S_1}\int_0^{\sigma}H(\rho)d\rho d\omega\leq
C_{|k|,\varepsilon} \frac{\beta p-1}{\alpha p-n}\sigma^{\frac{\alpha
p-n}{\beta p-1}}\sup_{\theta\in S_1}\lambda^{k\beta
p}(\sigma,\theta)<\infty,
$$
because $\lambda$ is a decreasing function on $\rho$ and $k\beta p<0$.
\item[(c)]$\quad$ $k=0$, i.e., $n=\alpha p$. Then  the simple computations taking into account  the monotonicity of $\lambda(\rho, \theta)$  and (\ref{44eq0}) give the following chain of
inequalities
$$
\frac{\lambda(\rho, \theta)}{\rho}\geq\frac{\lambda(\sigma,
\theta)}{\rho}>\frac{1}{1-\varepsilon}>1,
$$
for every $\rho<\sigma$, and
$$
0<\left(\ln\frac{\lambda(\rho, \theta)}{\rho}\right)^{-\beta
p}\leq\left(\ln\frac{\lambda(\sigma, \theta)}{\rho}\right)^{-\beta
p}.
$$
Hence from (\ref{44eq00}) and the choice of $H(\rho)$ we get
$$
\begin{array}{lll}
 &&\int_{S_1}\int_0^{\sigma}H(\rho)d\rho
d\theta=\int_{S_1}\int_0^{\sigma}\left(\ln\frac{\lambda}{\rho}\right)^{-\beta
p}\frac{d\rho}{\rho}d\theta
\\[2pt]
\\
&&\leq
\int_{S_1}\int_0^{\sigma}\left(\ln\frac{\lambda(\sigma,
\theta)}{\rho}\right)^{-\beta p}\frac{d\rho}{\rho}d\theta
\\[2pt]
\\
&&= \frac{1}{\beta
p-1}\int_{S_1}\left(\ln\frac{\lambda(\sigma,
\theta)}{\sigma}\right)^{1-\beta p}d\theta<\infty,
\end{array}
$$
since $\beta p>1$.
\end{itemize}
The proof of $v(x)\in L^p_{loc}(\Omega)$ is analogous for $k\geq0$
because the singularity of $v(x)$ at the origin is weaker than the
singularity of $w(x)$.

As for $k<0$, i.e., $n<\alpha p$, we obtain
$$
\int_{B_{\sigma}}v^pdx=\int_{S_1}\int_0^\sigma H_1(\rho)d\rho d\theta,
$$
and
$$
\begin{array}{lll}
&&H_1(\rho)=\rho^{n-1+(1-\alpha) p}\left(\frac{1-(\frac{\rho}{\lambda})^{k}}{k}\right)^{(1-\beta)
p}
\\[2pt]
\\
&&=\rho^{n-1+(1-\alpha) p+kp(1-\beta)}\lambda^{kp(\beta-1)}\left(\frac{1-(\frac{\rho}{\lambda})^{|k|}}{|k|}\right)^{(1-\beta)
p}.
\end{array}
$$
Hence
$$
\int_{B_{\sigma}}v^pdx\leq\frac{\sigma^{n+(1-\alpha)p+k(1-\beta)p}}
{n+(1-\alpha)p+k(1-\beta)p}\left(\frac{1-(\frac{\rho}{\lambda})^{|k|}}{|k|}\right)^{(1-\beta)
p}\lambda^{k(\beta-1)p}\omega_n<\infty,
$$
because $\lambda<\infty$   and
$$
\left(\frac{1-(\frac{\rho}{\lambda})^{|k|}}{|k|}\right)^{(1-\beta)
p}\leq\left(\frac{1-(\frac{\sigma}{\lambda})^{|k|}}{|k|}\right)^{(1-\beta)
p} \hbox{ for } \beta>1,
$$
$$
\left(\frac{1-(\frac{\rho}{\lambda})^{|k|}}{|k|}\right)^{(1-\beta)
p}\leq1 \hbox{ for } \beta\in\left[\frac{1}{p},1\right],
$$
$$
n+(1-\alpha)p+k(1-\beta)p>\frac{p}{\beta p-1}(\beta p-1+n-\alpha
p)\geq0,
$$
due to the condition
$\frac{n}{p}<\alpha\leq\frac{n+\beta p-1}{p}$, as well as $(1-\beta)p<p-1$.
\end{proof}

\begin{proof}[Proof of Theorem \ref{4th1}]

Without loss of generality we can suppose that $u\in
C^{\infty}_0(\Omega)\cap W^{1,p}_{0,v}(\Omega)$.

First, let us consider the case
$\alpha\neq\displaystyle\frac{n}{p}$, i.e., $k\neq0$. By (\ref{4eq1})
it holds that
\begin{equation}
\label{4eq14} \begin{array}{l} \langle x,\nabla g(s(x))\rangle
=-s(x)^{k-1}\left(\frac{|x|}{\lambda(x)}-\frac{|x|}{\lambda(x)}
\frac{\langle x,\nabla\lambda(x)\rangle}{\lambda(x)}\right)
\\[2pt]
\\
\leq- s(x)^k=kg(s(x))-1.
\end{array}
\end{equation}
For the vector function $f$
$$
f(x)=\displaystyle\frac{x}{|x|^{1+\alpha(p-1)}}|g(s(x))|^{-\beta(p-1)-1}g(s(x)),
$$
we
have
\begin{equation}
\label{4eq141}\begin{array}{lll}
&&f(x)v(x)=\displaystyle\frac{x}{|x|^{\alpha
p}}|g(s(x))|^{-\beta p}g(s(x)) \ \ \hbox{ and }
\\[2pt]
\\
&&|f(x)|^{p'}=|x|^{-\alpha p}|g(s(x))|^{-\beta p}\equiv
w(x)^p.
\end{array}
\end{equation}
Since $\beta p>1$ by (\ref{4eq14}) we get
\begin{equation}
\label{4eq15}\begin{array}{lll}
&&\hbox{div}(f(x)v(x))
\\[2pt]
\\
&&=\frac{n-\alpha
p}{|x|^{\alpha p}}|g(s(x))|^{-\beta p}g(s(x))+\frac{1-\beta
p}{|x|^{\alpha p}}|g(s(x))|^{-\beta p}\langle x,\nabla g(s(x))\rangle
\\[2pt]
\\
&&\geq\frac{1}{|x|^{\alpha p}|g(s(x))|^{\beta
p}}[(n-\alpha p) g(s(x))+(1-\beta p)(kg(s(x))-1)]
\\[2pt]
\\
&&=(\beta p-1)w^p.
\end{array}
\end{equation}
The last equality in (\ref{4eq15}) holds because $(n-\alpha p)+
k(1-\beta p)=0$.

Thus for the functions
$$
f_0(x)=-f(x)v(x), \quad v_0(x)=v^{-p'}(x), \quad w_0(x)=p(\beta-1)w^p(x),
$$
from (\ref{4eq141}) and (\ref{4eq15}) we obtain
\begin{equation}
\label{4eq142}
\begin{array}{lll}
&&-\hbox{div}f_0-(p-1)v_0|f_0|^{p'}=\hbox{div}(fv)-(p-1)|v|^{-p'}|f|^{p'}|v|^{p'}
\\[2pt]
\\
&\geq&(\beta p-1)w^p-(p-1)w^p=p(\beta-1)w^p=p(\beta-1)w_0.
\end{array}
\end{equation}
Moreover, the following identities hold in $\Omega_\varepsilon=\Omega_1\backslash B_\varepsilon$, $B_\varepsilon=\{|x|\leq\varepsilon\}$,  $\hbox{ supp } u\subset\Omega_1\subset\Omega$, $B_\varepsilon\subset\Omega_1$
\begin{equation}
\label{4eq143}\begin{array}{lll}
L(u)&=&\int_{\Omega_\varepsilon}v_0^{1-p}
\left|\frac{\langle f_0,\nabla u\rangle}{|f_0|}\right|^pdx\leq\int_{\Omega_\varepsilon}v^p
|\nabla u|^pdx;
\\[2pt]
\\
K(u)&=&\int_{\Omega_\varepsilon}v_0|f_0|^{p'}|u|^pdx=\int_{\Omega_\varepsilon}|f|^{p'}|u|^pdx=\int_{\Omega_\varepsilon}w^p|u|^pdx;
\\[2pt]
\\
N(u)&=&\int_{\Omega_\varepsilon}w_0|u|^pdx=p(\beta-1)\int_{\Omega_\varepsilon}w^p|u|^pdx;
\\[2pt]
\\
K_0(u)&=&\int_{\Gamma_+}\langle f_0,\eta\rangle|u|^pdS,
\end{array}
\end{equation}
where $\eta$ is outward normal vector to $\partial\Omega_{\varepsilon}$. Note that
$$
\int_{\partial\Omega_1}\langle f_0,\eta\rangle|u|^pdS=0, \hbox{ and }\int_{\partial B_\varepsilon}\langle f_0,\eta\rangle|u|^pdS\geq0,
$$
because $\langle f_0,\eta\rangle=|x|^{-\alpha(p-1)+1}|g(s(x))|^{-\beta(p-1)-1}g(s(x))\eta(x)>0$ on $\partial B_\varepsilon$.

From (\ref{4eq142}), (\ref{4eq143}) and Corollary \ref{cor300}, we obtain
$$
\begin{array}{lll}
&&\int_{\Omega_\varepsilon}v^p
|\nabla u|^pdx\geq L(u)
\geq\left(\frac{1}{p}\right)^p\frac{|K_0(u)+(p-1)K(u)+N(u)|^p}{K^{p-1}(u)}
\\[2pt]
\\
&&\geq\left(\frac{1}{p}\right)^p\frac{|(p-1)K(u)+N(u)|^p}{K^{p-1}(u)}=\left(\frac{1}{p}\right)^p\frac{\left|(p-1+p(\beta-1))
\int_{\Omega_\varepsilon}w^p|u|^pdx\right|^p}{\left(\int_{\Omega_\varepsilon}w^p|u|^pdx\right)^{p-1}}
\\[2pt]
\\
&&\geq\left(\frac{p\beta-1}{p}\right)^p\int_{\Omega_\varepsilon}w^p|u|^pdx.
\end{array}
$$
After the limit $\varepsilon\rightarrow0$ we get (\ref{4eq3}) in $\Omega$.

Second, for the case $\alpha=\frac{n}{p}$, i.e., $k=0$, it is enough
to replace inequality  (\ref{4eq14}) with inequality
$$
\begin{array}{lll}
&&\langle x,\nabla g(s(x))\rangle=-\frac{\lambda(x)}{|x|}\left\langle x,\nabla\left(\frac{|x|}{\lambda(x)}\right)\right\rangle
\\[2pt]
\\
&&=\frac{\lambda(x)}{|x|}\left(\frac{|x|}{\lambda(x)}\frac{\langle x,\nabla\lambda(x)\rangle}{\lambda(x)}-\frac{|x|}{\lambda(x)}\right)\leq-1,
\end{array}
$$
from (\ref{4eq1}).
The rest of the proof is similar to
the case $\alpha\neq\frac{n}{p}$.
\end{proof}

\subsection{Optimality of the Hardy constant $\gamma^p$}
\label{4sec2}

The optimality of   $\gamma^p$  for  $\Omega^{\ast}=\Omega=\{x: |x|<\lambda\}$, $\lambda=const$, is
guaranteed by the following theorem. However, the question whether the
inequality (\ref{4eq3}) is sharp in  $W^{1,p}_{0,v}(\Omega)$ is still open.
\begin{theorem}
\label{4th2} For every $\varepsilon, 0<\varepsilon<1$ there exists
$u_{\varepsilon}\in W^{1,p}_{0,v}(\Omega)$ such that for
$L(u_{\varepsilon})=\displaystyle\int_{\Omega} v^p|\nabla
u_{\varepsilon}|^pdx$ and
$R(u_{\varepsilon})=\displaystyle\int_{\Omega}
w^p|u_{\varepsilon}|^pdx$, we have
\begin{equation}
\label{4eq4}\gamma^p\leq\displaystyle\frac{L(u_{\varepsilon})}{R(u_{\varepsilon})}\leq(1+\varepsilon)^p\gamma^p.
\end{equation}
\end{theorem}
In the proof of the optimality of the constant $\gamma^p$ in inequality (\ref{4eq3}) we
will choose the function $u_{\varepsilon}$ so that the ratio of the
left-hand side to the right-hand side of (\ref{4eq3}) is in the interval
$[\gamma^p, (1+\varepsilon)^p\gamma^p]$.   In addition, one have to
show that both sides of (\ref{4eq3}) are finite.

\begin{proof}[Proof of Theorem \ref{4th2}]
For simplicity we suppose that
 $\lambda\equiv1$ so that $\Omega= B_1$. In polar coordinates $(\rho,\theta)$  with the functions  $v$ and $w$
 defined in (\ref{4eq2}) and for functions $u$ depending only on $\rho$, the inequality (\ref{4eq3}) becomes
\begin{equation}
\label{4eq16}
\begin{array}{lll}
\displaystyle
L(u)&=&\omega_n\int_0^1\rho^{n-1-p(\alpha-1)}\left|\frac{1-\rho^k}{k}\right|^{p-\beta
p}|u_{\rho}|^pd\rho
\\[2pt]
\\
&\geq&\gamma^p\omega_n\int_0^1\rho^{n-1-p\alpha}\left|\frac{1-\rho^k}{k}\right|^{-\beta
p}|u|^pd\rho =\gamma^pR(u),
\end{array}
\end{equation}
where $\omega$ is the measure of the unite sphere $S_1$ in $R^n$.

 Let us fix
$0<\varepsilon<1$. In order to prove
(\ref{4eq4}), we choose $u_{\varepsilon}$  for different cases of $k$
 as follows:
\begin{itemize}
\item[(a) ]$\quad$ Let $k>0$, i.e., $n>\alpha p$. For
$u_{\varepsilon}=(1-\rho^{k})^{\gamma(1+\varepsilon)}\rho^{-\gamma k
(1-\varepsilon)}$, we have
$$
\begin{array}{lll}
&&(u_{\varepsilon})_{\rho}=-\gamma
k(1+\varepsilon)(1-\rho^{k})^{\gamma(1+\varepsilon)-1}\rho^{-\gamma
k(1-\varepsilon)+k-1}
\\[2pt]
\\
&&-\gamma
k(1-\varepsilon)(1-\rho^{k})^{\gamma(1+\varepsilon)}\rho^{-\gamma
k(1-\varepsilon)-1}
\\[2pt]
\\
&&=-\gamma
k(1-\rho^{k})^{\gamma(1+\varepsilon)-1}\rho^{-\gamma
k(1-\varepsilon)-1}[(1+\varepsilon)\rho^{k}+(1-\varepsilon)(1-\rho^{k})].
\end{array}
$$
Now comparing the functions in the left-hand and  right-hand sides in
(\ref{4eq16}) with $u=u_{\varepsilon}$ we obtain
$$
\begin{array}{lll}
&&\frac{v(\rho)|(u_{\varepsilon})_{\rho}(\rho)|}{w(\rho)u_{\varepsilon}(\rho)}=k\gamma\left[\frac{1-\rho^k}{k}\rho\right](1-\rho^k)^{-1}\rho^{-1}(2\varepsilon\rho^k+1-\varepsilon)
\\[2pt]
\\
&&\leq (1+\varepsilon)\gamma
k\left[\frac{1-\rho^{k}}{k}\rho\right]\left[\frac{1}{\rho(1-\rho^{k})}\right]\leq(1+\varepsilon)\gamma,
\end{array}
$$
and hence
\begin{equation}
\label{4eq17}
\frac{L(u_{\varepsilon})}{R(u_{\varepsilon})}\leq(1+\varepsilon)^p\gamma^p.
\end{equation}
Note that
$$
R(u_{\varepsilon})=\omega_nk^{\beta p}\displaystyle\int_0^1\rho^{n-\alpha
p-1-\gamma kp(1-\varepsilon)}(1-\rho^{k})^{\gamma
p(1+\varepsilon)-\beta p}d\rho<\infty,
$$
since $ n-\alpha p-1-\gamma kp(1-\varepsilon)=\varepsilon(n-\alpha
p)-1>-1$

and

$\gamma p(1+\varepsilon)-\beta p=\varepsilon(\beta
 p-1)-1>-1$.

Inequalities (\ref{4eq16}) and (\ref{4eq17}) give the estimate
(\ref{4eq4}) for $k>0$.
\newline
\item[(b)]$\quad$ Let $k<0$, i.e., $n<\alpha p$. In this case we define
$u_{\varepsilon}=(1-\rho^{-k})^{\gamma(1+\varepsilon)}$. Similar
calculations as in (a) give us
$$
(u_{\varepsilon})_{\rho}=k\gamma
(1+\varepsilon)\rho^{-k-1}(1-\rho^{-k})^{\gamma(1+\varepsilon)-1},
$$
and
$$
\frac{v(\rho)|(u_{\varepsilon})_{\rho}|}{w(\rho)u_{\varepsilon}(\rho)}=(1+\varepsilon)\gamma
k\left[\frac{|1-\rho^{k}|}{k}\rho\right][\rho^{-k-1}(1-\rho^{-k})^{-1}]=(1+\varepsilon)\gamma.
$$
Hence
$$
\frac{L(u_{\varepsilon})}{R(u_{\varepsilon})}=(1+\varepsilon)^p\gamma^p,
$$
and with (\ref{4eq3}) we obtain inequality (\ref{4eq4}).

It remains to check that $R(u_{\varepsilon})<\infty$. Indeed,
$$
\begin{array}{lll}
 R(u_{\varepsilon})&=&\omega_n\int_0^1\rho^{n-\alpha
p-1}(1-\rho^{-k})^{\gamma
p(1+\varepsilon)}\left(\frac{\rho^k-1}{|k|}\right)^{-\beta
p}d\rho
\\[2pt]
\\
&=& \omega_n|k|^{\beta p}\int_0^1\rho^{n-\alpha
p-\beta pk-1}(1-\rho^{|k|})^{\gamma p(1+\varepsilon)-\beta
p}d\rho<\infty,
\end{array}
$$
because
$$
n-\alpha p-\beta pk-1=\frac{\alpha p-n}{\beta p-1}-1>-1,
$$

and

$$
\gamma p(1+\varepsilon)-\beta p=\varepsilon(\beta
p-1)-1>-1.
$$
\newline
\item[(c)] $\quad$ Let $k=0$, i.e., $n=\alpha p$. We  define for a  fixed $\mu$ the function
$0<\mu<1$
$$
u_{\varepsilon}=\left\{\begin{array}{l}\left(\ln\frac{1}{\rho}\right)^{\gamma(1+\varepsilon)}
\hbox{ for } \mu<\rho<1,
\\[2pt]
\\
\left(\ln\frac{1}{\mu}\right)^{\gamma(1+\varepsilon)} \hbox{ for }
0\leq\rho\leq\mu.
\end{array}\right.
$$
Then
$$
(u_{\varepsilon})_{\rho}=\left\{\begin{array}{l}-\gamma(1+\varepsilon)
\frac{1}{\rho}\left(\ln\frac{1}{\rho}\right)^{\gamma(1+\varepsilon)-1}, \hbox{ for } \mu<\rho<1,
\\[2pt]
\\
 0, \quad  \hbox{ for } 0\leq\rho\leq\mu,
\end{array}\right.
$$
and for
$$
v(\rho)=\rho^{1-\alpha}\left(\ln\frac{1}{\rho}\right)^{1-\beta}, \quad w(\rho)=\rho^{-\alpha}\left(\ln\frac{1}{\rho}\right)^{-\beta},
$$
it follows that
$$
\frac{v(\rho)|(u_{\varepsilon})_{\rho}|}{w(\rho)u_{\varepsilon}}=\left\{\begin{array}{l}\gamma(1+\varepsilon)
\left[\rho\ln\displaystyle\frac{1}{\rho}\right]\left[\displaystyle\frac{1}{\rho}\left(\ln\frac{1}{\rho}\right)^{-1}\right]=\gamma(1+\varepsilon) \hbox{ for } \mu<\rho<1, \\
0, \quad  \hbox{ for } 0\leq\rho\leq\mu.\end{array}\right.
$$
Hence we get
$$
\frac{L(u_{\varepsilon})}{R(u_{\varepsilon})}\leq(1+\varepsilon)^p\gamma^p,
$$
and combining with (\ref{4eq3}) we obtain inequality (\ref{4eq4}).

Let us check that  $R(u_{\varepsilon})$ is finite. Simple
computations give us
$$
\begin{array}{lll} R(u_{\varepsilon})&=&\omega_n\left(\ln
\frac{1}{\mu}\right)^{p\gamma(1+\varepsilon)}\int_0^{\mu}\left(\ln\frac{1}{\rho}\right)^{-\beta
p}
\frac{d\rho}{\rho}+\omega_n\displaystyle\int_{\mu}^1\left(\ln\frac{1}{\rho}\right)^{p\gamma(1+\varepsilon)-\beta
p} \frac{d\rho}{\rho}
\\[2pt]
\\
&=&
\omega_n\left(\ln
\frac{1}{\mu}\right)^{p\gamma(1+\varepsilon)}\frac{\left(\ln\frac{1}{\rho}\right)^{1-\beta
p}\left|^{\mu}_0\right.}{\beta
p-1}-\omega_n\frac{\left(\ln\frac{1}{\rho}\right)^{p\gamma(1+\varepsilon)-\beta
p+1}\left|^1_{\mu}\right.}{p\gamma(1+\varepsilon)-\beta
p+1}
\\[2pt]
\\
&=&\omega_n\left(\ln\frac{1}{\mu}\right)^{p\gamma(1+\varepsilon)-\beta
p+1}\left[\frac{1}{\beta p-1}+\frac{1}{p\gamma(1+\varepsilon)-\beta p+1}\right]
\\[2pt]
\\
&=&\omega_n\left(\frac{1}{\beta
p-1}\right)\left(\displaystyle\frac{1+\varepsilon}{\varepsilon}\right)\left(\ln\displaystyle\frac{1}{\mu}\right)^{p\gamma(1+\varepsilon)-\beta
p+1}<\infty,
\end{array}
$$
because
$$
p\gamma(1+\varepsilon)-\beta p+1=\varepsilon(\beta p-1)>0 \hbox{ and
} 1-\beta p<0.
$$
\end{itemize}
\end{proof}

\subsection{Examples and comments}
\label{4sec3}

The examples below illustrate that Theorem \ref{4th1} provides a new correction term in  Hardy  inequalities for weights with one type of singularity either at $0$ or on $\partial \Omega$. Let us recall that  the classical Hardy inequality for $p=2$, $n=3$ does not attain equality on functions from $H^1_0(B_1)$, $B_1$ is the unit ball, it allows the so-called correction term $A(u)=\left(\frac{1}{4}\right)\int_{B_1}Q(x)u^2dx$, see (\ref{2eq14}), i.e.,
\begin{equation}
\label{4eq7} \int_{B_1}|\nabla
u|^2dx-\displaystyle\frac{1}{4}\int_{B_1}\frac{|
u|^2}{|x|^2}dx\geq\displaystyle\frac{1}{4}\int_{B_1}Q(x)|u|^2dx,
\hbox{ for } u\in H^1_0(B_1).
\end{equation}
In \citet{ACR02, AVV10, BV97, FT02},
inequality (\ref{4eq7}) is proved for different radial symmetric
weights $Q(x)=q(r)$, $r=|x|$:
\begin{itemize}
\item $q(r)=const$ in  \citet{BV97};
\item $q(r)=\frac{1}{r^2}\sum_{j=1}^l\prod_{i=1}^j\ln^{(j)}\frac{\rho}{r}$,
$\rho=(\max |x|)e^{e^{e^{.^{.^{e -l-times}}}}}$, where
$\ln^{(m)}(.)=\ln(\ln^{(m-1)}(.))$ in \citet{ACR02};
\item $q(r)=\displaystyle\frac{1}{r^2}\sum_{i=1}^{\infty}X^2_1(\frac{r}{D})\cdots
X^2_i(\frac{r}{D})$, $X_1(t)=(1-\ln(t))^{-1}$,
$X_m=X_1(X_{m-1}(t))$, $D\geq1$ in \citet{FT02};
\item $q(r)=\displaystyle\frac{1}{(r\ln
r^2)}$ in \citet{AVV10}.
\end{itemize}
General characteristics of the possible  $Q(x)$ were
given  in \citet{GM08}, Theorem 1. It was shown that  (\ref{4eq7}) is valid for
$Q(x)$
 with decreasing   $q(r)$  if and only if  the ordinary
differential equation
\begin{equation}
\label{4eq8} y''(r)+\displaystyle\frac{y'(r)}{r}+q(r)y(r)=0,
\end{equation}
has a positive solution on $(0,1)$. Note that the results in
\citet{BV97, FT02, ACR02}, mentioned above, follow from \citet{GM08} with a
special choice of $Q(x)$.

\begin{example}
\label{4ex1}\rm
Let $\Omega=B_1$, $\lambda(x)=1$, $\alpha=\beta=1$, $n=3$,
$p=2$. Then $\Omega^{\ast}=B_1$, $k=1$,
$\gamma=\frac{1}{2}$ $v=1$, $w=|x|^{-1}(1-|x|)^{-1}$ and  (\ref{4eq3}) takes the form of
\begin{equation}
\label{4eq6} \displaystyle\int_{B_1}|\nabla
u|^2dx\geq\displaystyle\frac{1}{4}\int_{B_1}\frac{|u|^2}{|x|^2(1-|x|)^2}dx,
\hbox{ for } u\in H^1_0(B_1).
\end{equation}

The weight of the right-hand side of (\ref{4eq6}) has
singularities at zero and on $\partial B_1$ in contrast to the
weight in  the papers of \citet{ACR02, BV97, FT02}.

Factorize
$$
\displaystyle\frac{1}{|x|^2(1-|x|)^2}=\frac{1}{|x|^2}+\frac{2-|x|}{|x|(1-|x|)^2}
$$
then  inequality (\ref{4eq6}) takes the form of (\ref{4eq7}) with a kernel $Q(x)=\displaystyle\frac{2-|x|}{|x|(1-|x|)^2}$, i.e.,
\begin{equation}
\label{4eq9} \displaystyle\int_{B_1}|\nabla
u|^2dx-\frac{1}{4}\int_{B_1}\frac{|
u|^2}{|x|^2}dx\geq\displaystyle\frac{1}{4}\int_{B_1}\frac{(2-|x|)|u|^2}{|x|(1-|x|)^2}dx,
\hbox{ for } u\in H^1_0(B_1).
\end{equation}
Here  $Q(x)=q(r)$, where $q(r)$ is radially symmetric for  $r=|x|\in(0,1)$, convex
and $q(r)\rightarrow\infty$ for $r\rightarrow 0$ or $r\rightarrow
1$. The inequality (\ref{4eq9}) is not included in
\citet{GM08} since the function $q(r)$ in (\ref{4eq8}) is not decreasing
on $(0,1)$. Moreover, the general solution of equation (\ref{4eq8})
for $q(r)=\displaystyle\frac{2-r}{r(1-r)^2}$ is a linear combination
of Hypergeometric functions and has no positive solution in the
whole interval $(0,1)$, see \citet{GR80}.
\end{example}
Let us compare  Hardy inequality (\ref{4eq6}) in Example \ref{4ex1}
with two known results. In  \citet{SC06}  under
a sufficient condition (\ref{2eq7}),  Hardy inequality (\ref{2eq6}) is proved, see Sect. \ref{sect2}.

Actually the condition (\ref{2eq7}) is sufficient but not necessary
 for the validity of (\ref{2eq6}). Indeed, in the special case of
$\Omega=B_1$ and $n=3$ for $\phi(r)=1 $,
$h(r)=\left(\frac{r}{1-r}\right)^{1/2} $,   inequality
(\ref{2eq6}) follows from Example \ref{4ex1}. However,
$$
r\phi(r)(h^2)'(r)=\frac{r}{(1-r)^2}\neq const,
$$
and (\ref{2eq7}) fails. Moreover, a simple computation shows that for
$\phi(r)\equiv1$, the weight $\left|\frac{h'(r)}{h(r)}\right|^2$
of the right-hand side of (\ref{2eq6}) cannot be singular both at $0$
and at $1$ if  condition (\ref{2eq7}) is satisfied.

In \citet{BM97} the following
Hardy  inequality
\begin{equation}
\label{4eq12} \displaystyle\int_{B_1}|\nabla
u|^2dx-\frac{1}{4}\int_{B_1}\frac{|
u|^2}{(1-|x|)^2}dx\geq\frac{1}{4}\int_{B_1}Q(x)|u|^2dx, \hbox{ for }
u\in H^1_0(B_1),
\end{equation}
was proved for $Q(x)=const$, see (\ref{2eq9}).
If we factorize
$$
\displaystyle\frac{1}{|x|^2(1-|x|)^2}=\frac{1}{(1-|x|)^2}+\frac{1+|x|}{|x|^2(1-|x|)},
$$
inequality  (\ref{4eq6}) transforms into (\ref{4eq12}) with
$Q(x)=\frac{(1+|x|)}{|x|^2(1-|x|)}$.

In all inequalities obtained in
\citet{ ACR02, BV97, BM97, FT02}, the constants are optimal. From Theorem \ref{4th2}, it follows that for  Hardy inequality (\ref{4eq3}) the constant $\gamma^p$ is
optimal as well.

Finally, in \citet{FT02}, the authors show the validity of (\ref{4eq3}) under the restriction $2=p<n$. In the present section there are no
restrictions on $p$ except $p>1$.

Let us comment the  geometry of the domain
$\Omega^{\ast}$ in (\ref{4eq1}). It is well known that there are no
conditions on $\Omega$ for Hardy inequality with singularity at
$0\in\Omega$. However, when the singularity is on $\partial\Omega$ then the
restrictions about the convexity of the domain or its generalization
are always considered. We will give three simple examples for
$\lambda(x)$ and $\Omega^{\ast}$ when condition (\ref{4eq1}) holds.
\begin{example}
\label{4ex2} \rm
Let $\Omega= B_{\lambda_0}=\{x, |x|<\lambda_0\}$,
$\lambda_0=const>0$ and $\Omega^{\ast}=\Omega$.  In this case the
weights $v(x)$, $w(x)$ are singular at $0$ and on the whole boundary
$\partial\Omega=\partial B_{\lambda_0}$. If for simplicity
$\alpha=\beta=1$, $\gamma=\frac{p-1}{p}$ and $k=\frac{n-p}{p-1}\neq0$, i.e.,
$1<p\neq n$, then Hardy inequality (\ref{4eq3}) becomes
$$
 \int_{B_{\lambda_0}}|\nabla
u|^pdx\geq\left(\frac{|n-p|\lambda_0^{k}}{p}\right)^p\int_{B_{\lambda_0}}\frac{|u|^p}{|x|^p(\lambda^k_0-|x|^k)^p}dx.
$$
\end{example}
\begin{example}
\label{4ex3}\rm
Let $\Omega$ be a star-shaped domain with respect to
an interior ball centered at zero, so that $\Omega^{\ast}=\Omega$.
In this case we can choose $\lambda=\lambda(\theta)$ where $\theta$
is the angular variable of $x$ and $\partial\Omega=\{x,
|x|=\lambda(\theta)\}$. According to Lemma in section 1.1.8 of
\citet{Ma85}, $\lambda(\theta)\in C^{0,1}(\Omega)$ and condition
(\ref{4eq1}) and respectively  Hardy inequality (\ref{4eq3}) holds.
Note that in this case  the weights $v(x)$ and $w(x)$ are  singular
at zero and on the whole boundary $\partial\Omega$.
\end{example}

\begin{example}
\label{4ex4} \rm Let $n>p$, $\alpha=1$, $\beta=1$, $\gamma=\frac{p-1}{p}$ so
that $k=\frac{n-p}{p-1}>0$ and let (\ref{4eq1}) hold for
$\Omega$. For a domain $D\subset R^n$ we define
$$
\begin{array}{lll}
&&H(D)=\displaystyle\int_D|\nabla u|^pdx, \ \
H_1(D)=C_1\displaystyle\int_D\frac{|u|^p}{|x|^p}dx,
\\[2pt]
\\
&&H_2(D)=C_2\displaystyle\int_D\frac{|u|^p}{(\lambda(x)-|x|)^p}dx.
\end{array}
$$
We will define two domains $\Omega_1, \Omega_2$ such that
$\bar{\Omega}_1\subset\Omega$,
$\Omega_2=\Omega\backslash\bar{\Omega}_1$, $0\in\Omega_1$ and we will
show that $H(\Omega)\geq H_1(\Omega_1)+H_2(\Omega_2)$, i.e.,
\begin{equation}
\label{4eq19}\displaystyle\int_{\Omega}|\nabla u|^pdx\geq
C_1\displaystyle\int_{\Omega_1}\frac{|u|^p}{|x|^p}dx+C_2\displaystyle\int_{\Omega_2}\frac{|u|^p}{(\lambda(x)-|x|)^p}dx,
\end{equation}
with $C_1=\left(\frac{n-p}{p}\right)^p$,
$C_2=\left(\frac{1}{p'}\right)^p$ for $u\in
W^{1,p}_0(\Omega)$.

In order to prove (\ref{4eq19}) let us mention that
$\left(\frac{|x|}{\lambda(x)}\right)^{k}<1$ due to
(\ref{4eq1}) and $\quad$ $k>0$. Now we consider the function
$h(s)=\frac{1-s^{k}}{k}-s^{-1}+1$, where
$s=\frac{|x|}{\lambda(x)}$. Since $h'(s)>0$, this
function is increasing on $(0,1]$ and its maximum is attained for
$s=1$. Hence $h(s)\leq0$ and the following inequality is true
\begin{equation}
\label{44eq20}
\frac{|x|}{k}\left(1-\left(\frac{|x|}{\lambda(x)}\right)^k\right)\leq\lambda(x)-|x|.
\end{equation}
From $k>0$ the trivial inequality
$$
\frac{|x|}{k}\left(1-\left(\frac{|x|}{\lambda(x)}\right)^k\right)\leq\frac{|x|}{k},
$$
holds and combining with (\ref{44eq20}) we get
\begin{equation}
\label{4eq20}
|x|\frac{1}{k}\left(1-\left(\frac{|x|}{\lambda(x)}\right)^{k}\right)\leq
min\left(\frac{|x|}{k}, \lambda(x)-|x|\right), \ \  x\in\Omega.
\end{equation}
For
$$
v(x)=1, \ \  w(x)=|x|^{-1}\left(1-\left(\frac{|x|}{\lambda(x)}\right)^{k}\right)^{-1},
$$
applying (\ref{4eq20}) in
$\Omega_1=\{x\in\Omega: \min\left(\frac{|x|}{k}, \lambda(x)-|x|\right)=\frac{|x|}{k}\}$, $0\in\Omega_1$ and $\Omega_2=\Omega\backslash \bar{\Omega}_1$ correspondingly, we get from   (\ref{4eq3}) the chain of inequalities
$$
\begin{array}{lll}
H(\Omega)&=&\int_{\Omega}|\nabla
u|^pdx\geq\left(\frac{p-1}{p}\right)^p
\int_{\Omega}\frac{|u|^p}{|x|^p\left(1-\left(\frac{|x|}{\lambda(x)}\right)^{k}\right)^p}dx
\\[2pt]
\\
&\geq&\left(\frac{p-1}{p}\right)^p(k)^p
\int_{\Omega_1}\frac{|u|^p}{|x|^p}dx+\left(\frac{1}{p'}\right)^p\displaystyle\int_{\Omega_2}\frac{|u|^p}{(\lambda(x)-|x|)^p}dx
\\[2pt]
\\
&=&H_1(\Omega_1)+H_2(\Omega_2).
\end{array}
$$
Note that in (\ref{4eq19}) the constants $C_1$ and $C_2$ are optimal
for the corresponding classical cases with single singular weights.
However, (\ref{4eq19}) cannot be obtained by summing the classical
Hardy inequalities $ H(\Omega)\geq H_1(\Omega_1)$ and $H(\Omega)\geq
H_2(\Omega_2)$ because they are valid only for $u\in
W^{1,p}_0(\Omega_1)\cap W^{1,p}_0(\Omega_2)$.
\end{example}
Let us recall that $0\in\Omega_1$ in Example \ref{4ex4}  and this is
essential for the optimality of the constant $C_1$ in (\ref{4eq19}).
Remark that in  the case when $0\in\partial\Omega_1$ there exists a
constant $C'_1>C_1$, see \citet{Na08, PT05}.

\section{Sharp Hardy inequalities with weights singular at an interior point}
\label{sect5}

In this section we prove Hardy  inequality with weight singular at $0\in\Omega\subset R^n$, $n\geq2$
in the class of functions which are not zero on the boundary
$\partial\Omega$.  Hardy's constant is optimal and the inequality
is sharp due to the additional boundary  term. The section is based on
\citet{FKR14a}.

In order to formulate our main results we recall  the definition
of the trace operator, see \citet{Ad75},  \citet{Ev98}, Ch. 5.5  and \citet{Ma85}, Ch. 1.4.5.
\begin{definition}
\label{5def0}
For a bounded $C^1$ smooth domain $\Omega\subset R^n$, $n\geq2$
and $p>1$ the trace operator $T: W^{1,p}(\Omega)\rightarrow
L^p(\partial\Omega)$ is a bounded linear operator,
$Tu=u|_{\partial\Omega}$ for $u\in W^{1,p}(\Omega)\cap
C(\bar{\Omega})$ and $\|Tu\|_{L^p(\partial\Omega)}\leq
C(p,\Omega)\|u\|_{W^{1,p}(\Omega)}$ for $u\in W^{1,p}(\Omega)$.
\end{definition}
Let us  consider the following inequality
\begin{equation}
\label{5eq5}\displaystyle\int_{\Omega}|x|^l\left|\frac{\langle
x,\nabla u(x)\rangle}{|x|}\right|^pdx \geq
\left|\displaystyle\frac{p-l-n}{p}\right|^p\displaystyle\int_{\Omega}|x|^{l-p}|u(x)|^pdx,
\ \ u\in W^{1,p}_{l,0}(\Omega),
\end{equation}
for the constant $l\neq p-n$, where $p>1$, $n\geq2$ and $C^1$ is smooth, bounded domain
$\Omega\subset R^n$, $0\in\Omega$.   Here
$W^{1,p}_{l,0}(\Omega)$ is the completion of $C_0^{\infty}(\Omega)$
functions with respect to the norm
\begin{equation}
\label{5eq6}\left(\displaystyle\int_{\Omega}|x|^l|\nabla u(x)|^pdx\right
)^{1/p}<\infty,
\end{equation}
satisfying the condition
\begin{equation}
\label{5eq7}\lim_{\varepsilon\rightarrow0}\varepsilon^{l-p+1}\displaystyle
\int_{S_{\varepsilon}}|Tu|^pdS=0, \ \ S_{\varepsilon}=\{x\in\Omega;
|x|=\varepsilon\}.
\end{equation}
The constant $\left|\displaystyle\frac{p-l-n}{p}\right|^p$ in
(\ref{5eq5}) is optimal but  inequality (\ref{5eq5}) is not sharp in
$W^{1,p}_{l,0}(\Omega)$, see \citet{HPL52} for $n=1$. That is why we introduce a more general
class of functions $\hat{W}^{1,p}_l(\Omega)$, without any
restrictions of $u$ on $\partial\Omega$, and define an additional
term depending on the trace of $u$ on $\partial\Omega$.

We denote $\partial\Omega_-=\{x\in\partial\Omega:
\hbox{sgn}(p-l-n)\langle x,\eta\rangle<0\}$, where $\eta$ is the unit outward normal vector to $\partial\Omega$, and consider the norm
\begin{equation}
\label{5eq6a}\left(\displaystyle\int_{\Omega}|x|^l|\nabla
u(x)|^pdx\right)^{1/p}+\left(\frac{p-1}{p}\right)^{p-1}
\int_{\partial\Omega_-}|\langle x, \eta\rangle||x|^{l-p}|
u(x)|^pdS<\infty,
\end{equation}
see \citet{Ma85}, Ch. 1.1.15 in the case $|\partial\Omega_-|\neq\emptyset$
and Ch. 1.1.6 in the case $|\partial\Omega_-|=\emptyset$.

Let us mention that $|\partial\Omega_-|=0$ if and only if $p>l+n$
and $\Omega$ is a star-shaped domain with respect to the origin,
according to Definition \ref{5def1} below

We define $\hat{W}^{1,p}_l(\Omega)$ as  the  completion of
$C^{\infty}(\Omega)\cap C(\bar{\Omega})$ functions in the norm
(\ref{5eq6a}) which satisfy (\ref{5eq7}). Note that for $p-l-n<0$
condition (\ref{5eq7}) is fulfilled for every $u\in
C^{\infty}(\Omega)\cap C(\bar{\Omega})$, while for $p-l-n>0$,
condition (\ref{5eq7}) requires $u(0)=0$. Actu\-ally,
$\hat{W}^{1,p}_l(\Omega)$ for $p-l-n>0$ is the completion of
$C^{\infty}(\Omega)\cap C(\bar{\Omega})$ functions in the norm
(\ref{5eq6a}) which are equal to zero near the origin (see Remark \ref{5rem2}
below).

\begin{theorem}
\label{5th1} Suppose $\Omega\subset R^n$ is a bounded domain with $C^1$
smooth boundary and $0\in\Omega$. Then for every constant $l\neq p-n$, $p>1$,
$n\geq2$ and for every $u\in\hat{W}^{1,p}_l(\Omega)$ the
following inequality holds
\begin{equation}\label{5eq8}
\begin{array}{lll}
&&\left(\displaystyle\int_{\Omega}|x|^l\left|\displaystyle\frac{\langle
x,\nabla u(x)\rangle}{|x|}\right|^pdx\right)^{1/p} \geq
\left|\displaystyle\frac{p-l-n}{p}\right|\left(\displaystyle\int_{\Omega}|x|^{l-p}|u(x)|^pdx\right)^{1/p}
\\[2pt]
\\
&+&\displaystyle\frac{1}{p}\hbox{sgn}(p-l-n)\displaystyle\int_{\partial\Omega}|x|^{l-p}\langle
x,\eta\rangle|Tu|^pdS\left(\displaystyle\int_{\Omega}|x|^{l-p}|u(x)|^pdx\right)^{-1/p'},
\end{array}
\end{equation}
where $\displaystyle\frac{1}{p}+\displaystyle\frac{1}{p'}=1$ and
$\eta$ is the unit outward normal vector to $\partial\Omega$.
\end{theorem}
\begin{remark}\rm
\label{5rem2}
The standard definition of the space of functions in
(\ref{5eq8}) for
$p-l-n>0$ is the completion of
$C^{\infty}(\Omega)\cap C(\bar{\Omega})$ functions, with
respect to the norm (\ref{5eq6a}), which are zero near the origin.
However, applying Hardy inequality (\ref{5eq8}) to the ball
$B_{\varepsilon}=\{x\in\Omega: |x|<\varepsilon\}$, after the limit
$\varepsilon\rightarrow0$, we get  from $\int_{B_\varepsilon}|x|^l|\nabla u|^pdx\rightarrow 0$ that
\begin{equation}
\label{5eq80}
\int_{B_{\varepsilon}}|x|^{l-p}|u|^pdx\rightarrow0.
\end{equation}
Hence from (\ref{5eq8}) and (\ref{5eq80}) it follows that
$$
\displaystyle\int_{S_{\varepsilon}}|x|^{l-p}\langle
x,\eta\rangle|u|^pdS=
\varepsilon^{l-p+1}\int_{S_{\varepsilon}}|u|^pdS\rightarrow0,
$$
i.e., (\ref{5eq7}) is satisfied. In this way we get the same space
$\hat{W}^{1,p}_l(\Omega)$.
\end{remark}
\begin{proof}[Proof of Theorem \ref{5th1}]
In order to prove Theorem \ref{5th1}, let us introduce
the notations
\begin{equation}
\label{5eq11} f(x)=\hbox{sgn}(p-l-n)|x|^{l-p}x, \ \
v(x)=\frac{|p-l-n|}{p-1}|x|^{l/(1-p)}
\end{equation}
and
\begin{equation}
\label{5eq12} \begin{array}{lll}
L(u)&=&\int_{\Omega}v^{1-p}\left|\frac{\langle f,\nabla
u\rangle}{|f|}\right|^pdx
=\left(\frac{|p-l-n|}{p-1}\right)^{1-p}
\int_{\Omega}|x|^l\left|\displaystyle\frac{\langle
x,\nabla u\rangle}{|x|}\right|^pdx;
\\[2pt]
\\
K(u)&=&\int_{\Omega}v|f|^{p'}|u|^pdx=
\frac{|p-l-n|}{p-1}\int_{\Omega}|x|^{l-p}|u|^pdx;
\\[2pt]
\\
K_3(u)&=&\int_{\partial\Omega}\langle
f,\eta\rangle|Tu|^pdS
=\hbox{sgn}(p-l-n)\displaystyle\int_{\partial\Omega}|x|^{l-p}\langle
x, \eta\rangle|Tu|^pdS.
\end{array}
\end{equation}
Simple computations show that $f$ and $v$ satisfy the equality
\begin{equation}
\label{5eq13} -\hbox{div}f-(p-1)v|f|^{p'}=0, \ \ \hbox{in} \ \
\Omega\backslash\{0\},
\end{equation}
because $-(p-1)v|f|^{p'}=-|p-l-n||x|^{l-p}$

and
$$
\begin{array}{lll}
-\hbox{div}f&=&-\hbox{sgn}(p-l-n)\left[n|x|^{l-p}+(l-p)|x|^{l-p}\right]
\\[2pt]
\\
&=&|p-l-n||x|^{l-p}.
\end{array}
$$
Without loss of generality we will prove (\ref{5eq8}) for every $u\in
C^{\infty}(\Omega)\cap C^1(\bar{\Omega})$ satisfying (\ref{5eq7}), (\ref{5eq6a}).

For every small positive constant $\varepsilon$ we apply Corollary \ref{cor300}  in $\Omega_{\varepsilon}=\Omega\backslash
B_{\varepsilon}$, $B_\varepsilon=\{|x|\leq\varepsilon\}$, for $w(x)\equiv0$

After the limit $\varepsilon\rightarrow0$ in (\ref{3eq100}), since $N(u)=0$,  in notations (\ref{5eq12}) we obtain
$$
L(u)\geq\left(\frac{1}{p}\right)^p\frac{\left|K_3(u)+(p-1)K(u)\right|^p}{K^{p-1}(u)}.
$$
Now using the choice of $f$ in (\ref{5eq11}) we obtain
(\ref{5eq8}) for $u\in C^{\infty}(\Omega)\cap C(\bar{\Omega})$.

By standard approximation argument (see \citet{Ma85}, Ch. 1.1.15
and Ch. 1.1.6), we get  (\ref{5eq8}) for every
$u\in\hat{W}^{1,p}_l(\Omega)$.
\end{proof}
\begin{remark}
\label{6rem00}\rm

By means of the simple inequality $|1+z|^p\geq1+pz$ for every $z$ and $p>1$, we get
$$
\begin{array}{lll}
L(u)&\geq&\left(\frac{p-1}{p}\right)^pK(u)\left|1+\frac{1}{p-1}K_3(u)
K^{-1}\right|^p
\\[2pt]
\\
&\geq&\left(\frac{p-1}{p}\right)^pK(u)+\left(\frac{p-1}{p}\right)^{p-1}
(K_{3,+}(u)-K_{3,-}(u)),
\end{array}
$$
where
$$
K_{3,+}=\int_{\partial\Omega\backslash\partial\Omega_-}|x|^{l-p}|\langle
x,\eta\rangle||Tu|^pdS, \ \
K_{3,-}=\displaystyle\int_{\partial\Omega_-}|x|^{l-p}|\langle
x,\eta\rangle||Tu|^pdS.
$$
So, for $u\in\hat{W}^{1,p}_l(\Omega)$ we obtain the `linear' form of Hardy inequality (\ref{5eq8})
$$
 L(u)+\left(\frac{p-1}{p}\right)^{p-1}
K_{3,-}(u)\geq
\left(\frac{p-1}{p}\right)^pK(u)+\left(\frac{p-1}{p}\right)^{p-1}
K_{3,+}(u).
$$
\end{remark}
\begin{theorem}
\label{5th2} Under the assumptions of Theorem \ref{5th1} inequality
(\ref{5eq8}) is an equality for
\begin{equation}
\label{5eq10}
u_k=|x|^k\Phi\left(\displaystyle\frac{x}{|x|}\right)
\end{equation}
for every smooth function $\Phi$ and every  constant
$k>\displaystyle\frac{p-l-n}{p}$, such that $u_k\in
\hat{W}^{1,p}_l(\Omega)$, i.e., (\ref{5eq8}) is sharp in
$\hat{W}^{1,p}_l(\Omega)$.
\end{theorem}
\begin{proof}

From Theorem \ref{3th10}, it follows that (\ref{5eq13}) is an equality if  (\ref{3eq52})--(\ref{3eq54}) hold, i.e., if (\ref{5eq13}) and
\begin{equation}
\label{5eq50}\begin{array}{l} u\langle f,\nabla u\rangle=|u\langle
f,\nabla u\rangle|,
\\[2pt]
\\
\langle f,\nabla u\rangle=k^p_1v|f|^{p'}u,
\end{array}
\end{equation}
are fulfilled for a.e. $x\in\Omega$ and for some constant $k_1\geq0$. From the choice of $f$ and $v$ in
(\ref{5eq11}), equation  (\ref{5eq50}) is equivalent to
\begin{equation}
\label{5eq17} sgn(p-l-n)u\langle x,\nabla u\rangle=|u\langle x,\nabla
u\rangle|, \ \ \langle x,\nabla u\rangle=ku,
\end{equation}
where $k=k_1\displaystyle\frac{p-l-n}{p-1}$.
We are looking for solution of the second equation in (\ref{5eq17})
in the form $u=|x|^kz(x)\in\hat{W}^{1,p}_l(\Omega)$. Simple computations give us that $z(x)$
is a solution of the homogeneous, first order partial differential
equation
\begin{equation}
\label{5eq17a} \sum_{i=1}^nx_iz_{x_i}=0 \ \ \hbox{in }
\Omega\backslash\{0\}.
\end{equation}
The system of characteristic equations of (\ref{5eq17a}) becomes
\begin{equation}
\label{5eq17b} \dot{x}_1(t)=x_1,\ldots, \dot{x}_n(t)=x_n,
\end{equation}
and the functions
$\varphi_1(x)=\displaystyle\frac{x_1}{|x|},\ldots,\varphi_n(x)=\displaystyle\frac{x_n}{|x|}$
are constants along the trajectories of (\ref{5eq17b}), i.e.,
$\varphi_1,\ldots,\varphi_n$ are the first integrals of (\ref{5eq17a}).
Note that only $n-1$ of them are independent first integrals, so
that the general solution of (\ref{5eq17a}) is given by
$$
z(x)=\Phi\left(\displaystyle\frac{x}{|x|}\right).
$$
Hence function $u_k(x)=|x|^k\Phi\left(\displaystyle\frac{x}{|x|}\right)$  is a general
solution of (\ref{5eq17}) for arbitrary smooth function $\Phi$ and constant $k$, such that
 $|x|^k\Phi\in \hat{W}^{1,p}_l(\Omega)$, i.e., $u_k$ satisfies (\ref{5eq6}) and (\ref{5eq7}).

Let us check up when condition (\ref{5eq7}) holds. With
the change of the variables $x=\varepsilon y$ we get
$$
\begin{array}{lll}
&&\varepsilon^{l-p+1}\int_{
S_{\varepsilon}}|u_k|^{p}dS
=\varepsilon^{l-p+1}\int_{
S_{\varepsilon}}|x|^{kp}\left|\Phi\left(\frac{x}{|x|}\right)\right|^pdS
\\[2pt]
\\
&=&\varepsilon^{l-p+1+kp+n-1}\displaystyle
\int_{|y|=1}\left|\Phi(y)\right|^pdS.
\end{array}
$$
Therefore, (\ref{5eq7}) is satisfied if and only if  $k>\displaystyle\frac{p-l-n}{p}$ and $
\int_{|y|=1}\left|\Phi(y)\right|^pdS<\infty$.

We will prove that  both sides of (\ref{5eq8}) are finite for all
functions $u_k$ defined in (\ref{5eq10}). For this purpose it is enough
to check that $K(u_k)<\infty$. In fact, for some fixed small constant,
$a\in(0,1)$ and  for the ball $B_a=\{|x|<a\}$, we get the chain of inequalities
$$
\begin{array}{lll}
\int_{\Omega}|x|^{l-p+kp}\left|\Phi\left(
\frac{x}{|x|}\right)\right|^pdx
&=&\int_{B_a}|x|^{l-p+kp}|\Phi|^pdx+\int_{\Omega\backslash
B_a}|x|^{l-p+kp}|\Phi|^pdx
\\[2pt]
\\
&\leq&C_1\int_{B_a}|x|^{\lambda}dx+C_2<\infty,
\end{array}
$$
for $\lambda=l-p+kp>-n$ and some constants
$0<C_1, C_2<\infty$. The above inequality follows from
$k>\displaystyle\frac{p-l-n}{p}$.
\end{proof}

In the following remarks we will compare our result in Theorem \ref{5th1}, i.e., inequality (\ref{5eq8}) with the corresponding
results about Hardy inequalities with additional boundary term in
\citet{WZ03, BFT04, BGP10}.
\begin{remark}\rm
\label{5re1}
 Consider the special case of the constants : $p=2$, $l=-2a$,
$a>0$, $n-2a-2>0$ and $\Omega$ is the unit ball $B_1\subset R^n$.
 Since $\left|\frac{\langle x, \nabla u\rangle}{|x|}\right|^p<|\nabla u|^p$, from (\ref{5eq8}),
rising both sides of this inequality to second power
and since $\eta=x$ on $\partial B_1$ we get
\begin{equation}\label{5r2}
\begin{array}{lll}
&&\displaystyle\int_{B_1}|x|^{-2a}|\nabla u|^2dx\geq
\left(\displaystyle\frac{n-2-2a}{2}\right)^2\left(\displaystyle\int_{B_1}|x|^{-2a-2}|u|^2dx\right)
\\
[2pt]
\\&-&\displaystyle\frac{n-2a-2}{2}\displaystyle\int_{\partial
B_1}|Tu|^2dS
\\
[2pt]
\\&+&\displaystyle\frac{1}{4}\left(\displaystyle\int_{\partial B_1}|Tu|^2dS\right)^2
\left(\displaystyle\int_{B_1}|x|^{-2a-2}|u|^2dx\right)^{-1}, \ \
u\in \hat{W}^{1,2}_{-2a}(B_1).
\end{array}
\end{equation}
In \citet{WZ03}, the following Hardy inequality with
additional boundary term was proved
\begin{equation}
\label{5eq9}
\begin{array}{lll}
\displaystyle\int_{B_1}|x|^{-2a}|\nabla u(x)|^2dx&>&
\left(\displaystyle\frac{n-2-2a}{2}\right)^2\displaystyle\int_{B_1}|x|^{-2a-2}|u(x)|^2dx
\\
[2pt]
\\&-&\displaystyle\frac{n-2-2a}{2}\displaystyle\int_{\partial
B_1}|Tu|^2dS,  \ \ u\in \hat{W}^{1,2}_{-2a}(B_1).
\end{array}
\end{equation}
The constant $\left(\displaystyle\frac{n-2-2a}{2}\right)^2$ in both
inequalities (\ref{5r2}) and (\ref{5eq9}) is optimal. The difference
between (\ref{5eq9}) and (\ref{5r2})  is the  additional positive term.
Moreover, (\ref{5eq9}) is not sharp, i.e., it is a strict inequality, while  the
additional term in (\ref{5r2}) guarantees its sharpness, i.e., there
exists a class of functions $\hat{W}^{1,2}_{-2a}(B_1)$
for which (\ref{5r2}) is an equality.
\end{remark}
\begin{remark}\rm
\label{5re2}
In \citet{BFT04} new Hardy inequalities in bounded
domains $\Omega$ for functions $H^1(\Omega)$, see eq. (2.4) in
\citet{BFT04} are obtained
\begin{equation}\label{5r3}
\int_{\Omega}|\nabla
u|^2dx+c\displaystyle\int_{\partial\Omega}|x|^{-2}\langle x,
\eta\rangle|Tu|^2dS\geq
c(n-2-c)\displaystyle\int_{\Omega}|x|^{-2}|u|^2dx,
\end{equation}
for $u\in H^1(\Omega)$, where $0<c\leq\displaystyle\frac{n-2}{2}$.

 The inequality (\ref{5eq8}) for the case $p=2, l=0, n>p$ with a similar transformation as in
 Remark \ref{5re1} reads
\begin{equation}
\label{5r5}
\begin{array}{lll}
&&\int_{\Omega}|\nabla u(x)|^2dx\geq
\left(\frac{n-2}{2}\right)^2\displaystyle\int_{\Omega}|x|^{-2}|u(x)|^2dx
\\[2pt]
\\
&-&\frac{n-2}{2}\displaystyle\int_{\partial
\Omega}|x|^{-2}\langle x, \eta\rangle|Tu|^2dS
\\[2pt]
\\
&+&\frac{1}{4}\left(\displaystyle\int_{\partial \Omega}|x|^{-2}\langle
x, \eta\rangle|Tu|^2dS\right)^2
\left(\int_{\Omega}|x|^{-2}|u(x)|^2dx\right)^{-1}, \ \
u\in H^1(\Omega).
\end{array}
\end{equation}
The comparison of (\ref{5r5}) and (\ref{5r3})  for  the optimal constant $c=\displaystyle\frac{n-2}{2}$ shows that in (\ref{5r5}) there exists an additional positive term on the right-hand side. Moreover, with this additional term inequality (\ref{5r5}) is sharp,
i.e., it is an equality for the class of functions $u=u_{k}\in H^1(\Omega)$ defined in
(\ref{5eq10}).
\end{remark}

\begin{remark}\rm
\label{5re3}
In \citet {BGP10}, see (13),  the following Hardy inequality is studied
\begin{equation}\label{5r6}
\int_{\Omega}|\nabla
u|^2dx+c\displaystyle\int_{\partial\Omega}|Tu|^2dS\geq
h(c)\displaystyle\int_{\Omega}|x|^{-2}|u|^2dx, \ \  u\in
H^1(\Omega),
\end{equation}
where $c\in \left[0,C_n\right]$,
$C_n\geq\frac{n-2}{2}$,
($C_n=\frac{n-2}{2}$ for $\Omega=B_1$) and
$h(c)\in\left[0,\left(\frac{n-2}{2}\right)^2\right]$,
are defined as

\begin{equation}\label{5r7}
h(c)=\inf_{u\in
H^1(\Omega)\backslash\{0\}}\displaystyle\frac{\displaystyle\int_{\Omega}|\nabla
u|^2dx+c\displaystyle\int_{\partial\Omega}|Tu|^2dS}{\displaystyle\int_{\Omega}|x|^{-2}|u|^2dx}.
\end{equation}
By means of positive solutions of the eigenvalue problem   under Steklov boundary conditions
$$
\left|\begin{array}{lll}
&-&\Delta u=h(c)\displaystyle\frac{|u|^2}{|x|^2}, \ \ \hbox{ in }
\Omega
\\[2pt]
\\
&& u_{\eta}+cu=0, \ \ \hbox{on } \partial\Omega,
\end{array}\right.
$$
see (15) in \citet{BGP10}, it is shown (see
Theorem 8 in  \citet{BGP10}) that for the  value of
$c=\displaystyle\frac{n-2}{2}$ the infimum in  (\ref{5r7}), i.e.,
$$
h\left(\displaystyle\frac{n-2}{2}\right)=\left(\displaystyle\frac{n-2}{2}\right)^2
$$
is not achieved. This means that for the optimal constant $\left(\frac{n-2}{2}\right)^2$ Hardy inequality (\ref{5r6}) is not sharp.
\end{remark}
The inequality (\ref{5eq8}) for the case $p=2, l=0, n>p$ and $\Omega=B_1(0)$ becomes (\ref{5r2})
in Remark \ref{5re1} for $a=0$ and in comparison with (\ref{5r6}) has an additional positive term in the
right-hand side. Moreover, the inequality (\ref{5r2}) is an equality for the functions
defined in (\ref{5eq10}).

As a consequence of Theorem \ref{5th1} we get an
extension of the classical Hardy inequality (\ref{5eq5}) for functions $u$ in the largest class
$\hat{W}^{1,p}_l(\Omega)$, i.e., when $u$ is not necessary zero on the whole boundary $\partial\Omega$.

\subsection{Star--shaped domains}
\label{5sec3-1}

We consider the case of domains $\Omega\subset R^n$, $n\geq2$, $0\in\Omega$ which are star-shaped with respect to the origin.
Let us recall the definitions of a star-shaped and strictly star-shaped $C^1$ smooth domains.
\begin{definition}
\label{5def1}
The domain $\Omega$, $\partial\Omega\in C^1$ is:
\begin{itemize}
\item[i)]$\quad$ star-shaped domains with respect to the origin if
\begin{equation}
\label{5eq18a}
\langle x,\eta(x)\rangle\geq0, \ \ \hbox{ for every } x\in\partial\Omega,
\end{equation}
where $\eta(x)$ is the unit outward normal vector to $\partial\Omega$ at the point $x\in\partial\Omega$.
\item[ii)]$\quad$ strictly star-shaped domains with respect to the origin
if  inequality (\ref{5eq18a}) is strict, i.e.,
\begin{equation}
\label{5eq18b}
\langle x,\eta(x)\rangle>0, \ \ \hbox{ for every } x\in\partial\Omega,
\end{equation}
\end{itemize}
\end{definition}
Let us note that for star-shaped domains the sign of the additional term in (\ref{5eq8})
depends only on the sign of the constant $p-l-n$.

\begin{theorem}
\label{5th3} Suppose  $\Omega$ is a bounded, star-shaped domain with respect to the origin
in $R^n$, $n\geq2$ with $C^1$ smooth boundary $\partial\Omega$ and $0\in\Omega$. Then for every $p>1$ we have:

(i) If $p-l-n>0$, inequality (\ref{5eq5}) is satisfied for every $u(x)\in
\hat{W}^{1,p}_l(\Omega)$ and the constant
$\displaystyle\frac{p-l-n}{p}$ in (\ref{5eq5}) is optimal.

If additionally $\Omega$ is a strictly star-shaped domain with respect to the origin, then (\ref{5eq5}) is not sharp in $\hat{W}^{1,p}_l(\Omega)$.

(ii) If $p-l-n<0$, then (\ref{5eq5}) in general does not hold, for example,
for functions $u_{k}\in \hat{W}^{1,p}_l(\Omega)$ defined  in (\ref{5eq10}).
\end{theorem}
\begin{proof}
(i)$\quad$ From (\ref{5eq18a})  inequality
(\ref{5eq5}) holds from (\ref{5eq8}).

It is easy to prove that when
$p-l-n>0$ the constant
$\displaystyle\frac{p-l-n}{p}$ is optimal in (\ref{5eq5}) and (\ref{5eq8}). For this purpose, we will use function $u_{k}$, defined in (\ref{5eq10})
with $k>\displaystyle\frac{p-l-n}{p}$ and $\Phi\equiv1$. Since (\ref{5eq8}) is an equality
for every $u_{k}$ we get
$$
\begin{array}{lll}
&&\left(\frac{L(|x|^k)}{K(|x|^k)}\right)^{1/p}=\frac{p-l-n}{p}+
\frac{1}{p}\int_{\partial\Omega}|x|^{l-p+kp}<x,
\eta>dS
\\[2pt]
\\
&\times&\left(\displaystyle\int_{\Omega}|x|^{l-p+kp}dx\right)^{-1}\rightarrow\displaystyle\frac{p-l-n}{p}
\ \ \hbox{for} \ \ k\rightarrow\frac{p-l-n}{p}+0.
\end{array}
$$
The above limit follows from the inequalities
$$
\begin{array}{lll}
&&\displaystyle\int_{\Omega}|x|^{l-p+kp}dx\geq\int_{B_a}|x|^{l-p+kp}dx=\omega_n\int_0^ar^{l-p+kp+n-1}dr
\\[2pt]
\\&=&\omega_n\displaystyle\frac{a^{l-p+kp+n}}{l-p+kp+n}\rightarrow+\infty,
\ \ \hbox{for} \ \ k\rightarrow\frac{p-l-n}{p}+0,
\end{array}
$$
and
$$
\begin{array}{lll}
&&\frac{1}{p}\int_{\partial\Omega}|x|^{l-p+kp}\langle x,\eta\rangle dS\rightarrow\frac{1}{p}\int_{\partial\Omega}|x|^{-n}\langle x,\eta\rangle dS<\infty
\\[2pt]
\\
&&\hbox{for} \ \ k\rightarrow\frac{p-l-n}{p}+0,
\end{array}
$$

where $\omega_n$ is the measure of the unit sphere in $R^n$ and $B_a\subset\Omega, \partial B_a\cap\partial\Omega=\emptyset$.

If we suppose that (\ref{5eq5}) is sharp for some function $w(x)\in
\hat{W}^{1,p}_l(\Omega)$ in a strictly star--shaped domain $\Omega$ then from (\ref{5eq5}) and (\ref{5eq8}) we
have
$$
\int_{\partial\Omega}|x|^{l-p}\langle
x,\eta\rangle|Tw|^pdS=0.
$$
Hence due to (\ref{5eq18b}) it follows that
$$
 Tw=0 \ \ \hbox{ for a. e. } x\in\partial\Omega.
$$
This means that (\ref{5eq5}) is also sharp in $\hat{W}^{1,p}_{l,0}(\Omega)$  which proves Theorem
\ref{5th3} (i).

(ii)$\quad$ If $p-l-n<0$, then for $u_k(x)=|x|^k$ we get from
Theorem \ref{5th2} and (\ref{5eq18a})
$$\begin{array}{lll}
&&\displaystyle\int_{\Omega}|x|^l\left|\frac{\langle x,\nabla
u_{k}(x)\rangle}{|x|}\right|^pdx=-\displaystyle\frac{1}{p}
\int_{\partial\Omega}|x|^{l-p}\langle x,\eta\rangle|u_{k}(x)|^pdS
\\[2pt]
\\
&\times&\left(\displaystyle\int_{\Omega}|x|^{l-p}|u_{k}(x)|^pdx\right)^{-1/p'}
+\displaystyle\frac{|p-l-n|}{p}
\left(\displaystyle\int_{\Omega}|x|^{l-p}|u_{k}(x)|^pdx\right)^{1/p}
\\[2pt]
\\
&<&\displaystyle\frac{|p-l-n|}{p}
\left(\displaystyle\int_{\Omega}|x|^{l-p}|u_{k}(x)|^pdx\right)^{1/p}.
\end{array}
$$
Hence (\ref{5eq5}) is not satisfied for
$u=u_{k}(x)=|x|^k\in\hat{W}^{1,p}_l(\Omega)$ and $k>\displaystyle\frac{p-l-n}{p}$ which proves Theorem \ref{5th3} (ii).
\end{proof}

\subsection{General domains}
\label{5sec3-2} In order to prove (\ref{5eq5}) without geometry
conditions (\ref{5eq18a}) or (\ref{5eq18b}) as in Sect. \ref{5sec3-1}
we specify the class of functions. Let us introduce the spaces
$W^{1,p}_{l,+}(\Omega)$, resp. $W^{1,p}_{l,-}(\Omega)$, which are
the completion of $C^{\infty}(\Omega)\cap C(\bar{\Omega})$ functions
with respect to the norm (\ref{5eq6a}), satisfying in addition
(\ref{5eq7}) and (\ref{5eq19}), resp. (\ref{5eq7}) and (\ref{5eq20}):
\begin{equation}
\label{5eq19}\displaystyle\int_{\partial\Omega}|x|^{l-p}\langle
x,\eta\rangle|Tu|^pdS\geq0,
\end{equation}
\begin{equation}
\label{5eq20}\displaystyle\int_{\partial\Omega}|x|^{l-p}\langle
x,\eta\rangle|Tu|^pdS\leq0.
\end{equation}
Note that obviously the following inclusions hold:
$$
W^{1,p}_{l,0}(\Omega)\subset W^{1,p}_{l,\pm}(\Omega)\subset
\hat{W}^{1,p}_l(\Omega); \ \ W^{1,p}_{l,+}(\Omega)\cup
W^{1,p}_{l,-}(\Omega)=\hat{W}^{1,p}_l(\Omega).
$$
By means of conditions (\ref{5eq19}) or (\ref{5eq20}) one can control
the sign of the additional term in inequality (\ref{5eq8}), and, we have the
following result for general domains.

\begin{theorem}
\label{5th4}
Suppose $\Omega$ is a  bounded domain in $R^n$, $n\geq2$ with $C^1$ smooth
boundary, $0\in\Omega$ and $p>1$. Then:

(i) Inequality  (\ref{5eq5}) holds for every $u\in
W^{1,p}_{l,0}(\Omega)$;

(ii) If $p-l-n<0$, then
inequality (\ref{5eq5}) holds for all functions $u\in W^{1,p}_{l,-}(\Omega)$;

(iii) If $p-l-n>0$, then inequality (\ref{5eq5}) holds for all functions $u\in
W^{1,p}_{l,+}(\Omega)$. The constant
$\displaystyle\frac{p-l-n}{p}$ is optimal but inequality
(\ref{5eq5}) is not a sharp one in $W^{1,p}_{l,+}(\Omega)$. However,
an inequality with additional term (\ref{5eq8}) is sharp in $W^{1,p}_{l,+}(\Omega)$.
\end{theorem}
In order to prove Theorem \ref{5th4} we need the following
Lemma:
\begin{lemma}
\label{5lem1} Let $\Omega$ be a bounded domain in $R^n$, $n\geq2$ with $C^1$
smooth boundary $\partial\Omega$,  $0\in\Omega$ and $p>1$. Then identity
\begin{equation}
\label{5eq21}\begin{array}{lll}
&&\int_{\partial\Omega}|x|^{l-p+kp}\langle x,\eta\rangle \left|T\Phi\left(\displaystyle\frac{x}{|x|}\right)\right|^pdS
\\[2pt]
\\
&=&(l-p+kp+n)\int_{\Omega}|x|^{l-p+kp}\left|\Phi\left(\displaystyle\frac{x}{|x|}\right)\right|^pdx>0
\end{array}
\end{equation}
holds for every $k>\displaystyle\frac{p-l-n}{p}$ and every nontrivial function
$|x|^k\Phi\left(\displaystyle\frac{x}{|x|}\right)\in\hat{W}^{1,p}_l(\Omega)$.
\end{lemma}
\begin{proof}
If $\varepsilon>0$ is a sufficiently small constant such that $B_{\varepsilon}\subset\Omega$,
where $B_{\varepsilon}=\left\{x: |x|<\varepsilon\right\}$ then
$$
\begin{array}{lll}
&&\hbox{div}\left(|x|^{l-p+kp}x\left|\Phi\left(\displaystyle\frac{x}{|x|}\right)\right|^p\right)
\\[2pt]
\\
&=&(l-p+kp+n)|x|^{l-p+kp}\left|\Phi\left(\displaystyle\frac{x}{|x|}\right)\right|^p+
|x|^{l-p+kp}\langle x,\nabla|\Phi|^p\rangle
\\[2pt]
\\
&=&(l-p+kp+n)|x|^{l-p+kp}\left|\Phi\left(\displaystyle\frac{x}{|x|}\right)\right|^p, \ \ \hbox{for a.e. } x\in\Omega\backslash B_{\varepsilon}.
\end{array}
$$
With integration by parts of the
above equality in $\Omega\backslash B_{\varepsilon}$, equality (\ref{5eq21}) follows after the limit $\varepsilon\rightarrow0$ from (\ref{5eq7}) because
$$
\begin{array}{lll}
&&\int_{S_\varepsilon}|x|^{l-p+kp}\langle x,\eta\rangle\left|T\Phi\left(\frac{x}{|x|}\right)\right|^pdS
\\[2pt]
\\
&&=-\varepsilon^{l-p+kp+1}\int_{S_\varepsilon}\left|T\Phi\left(\frac{x}{|x|}\right)\right|^pdS\rightarrow0, \hbox{ as } \varepsilon\rightarrow0.
\end{array}
$$
\end{proof}
\begin{proof}[Proof of Theorem \ref{5th4}]
(i)$\quad$ The inequality (\ref{5eq5}) is a
direct consequence of (\ref{5eq8}).
The optimality of the constant $\displaystyle\frac{|p-l-n|}{p}$ for
(\ref{5eq5}) follows in the same way as in Theorem \ref{5th3} (i) for
(\ref{5eq8}) with $p-l-n>0$.

(ii)$\quad$ For $p-l-n<0$ inequality  (\ref{5eq5}) follows from (\ref{5eq8}) and (\ref{5eq20}). If (\ref{5eq5}) is sharp
for some function $z(x)\in
W^{1,p}_{l,-}(\Omega)$, then from (\ref{5eq5}) and (\ref{5eq8}) we get
\begin{equation}
\label{5eq21a}
G(z)=\displaystyle\int_{\partial\Omega}|x|^{l-p}\langle x,\eta\rangle|Tz|^pdS\geq0.
\end{equation}
Since $z(x)\in W^{1,p}_{l,-}(\Omega)$ from (\ref{5eq20}) it follows
that $G(z)=0$ and (\ref{5eq5}) is sharp in $\hat{W}^{1,p}_{l,0}(\Omega)$ which is impossible.

(iii) $\quad$ For $p-l-n>0$ inequality  (\ref{5eq5}) follows from (\ref{5eq8}) and (\ref{5eq19}). From Theorem \ref{5th2} for  functions
$u_{k}(x)=|x|^k$, $k>\displaystyle\frac{p-l-n}{p}$ inequality
(\ref{5eq8}) becomes an equality. Since $u_{k}(x)\in
\hat{W}^{1,p}_l(\Omega)$ it is enough to show that $u_{k}(x)$
satisfies (\ref{5eq19}), i.e., $u_{k}(x)\in W^{1,p}_{l,+}(\Omega)$.
This follows from  Lemma \ref{5lem1} for $\Phi\equiv1$.

If we suppose that (\ref{5eq5}) is sharp for some function $w(x)\in
W^{1,p}_{l,+}(\Omega)$, then from (\ref{5eq5}) and (\ref{5eq8}) we have $G(w)\leq0$ where the function $G$ is defined in (\ref{5eq21a}).
Since $w(x)\in W^{1,p}_{l,+}(\Omega)$, i.e., $G(w)\geq0$, then  from (\ref{5eq19}) it follows
that $G(w)=0$. This means that (\ref{5eq8}) is sharp for the function
$w(x)\in\hat{W}^{1,p}_{l,0}(\Omega)$ which is impossible.

The optimality of the constant $\displaystyle\frac{p-l-n}{p}$ follows in the same
way as in the proof of Theorem \ref{5th3} (i).
\end{proof}

\section{Sharp Hardy inequalities in star-shaped domains with double singular weights}
\label{sect6}

In the present section we prove  Hardy inequalities with double singular weights in bounded, star-shaped domains $\Omega\subset R^n$, $n\geq2$.  The weights are singular  at an interior point and  on the boundary of the domain. Hardy's constant is optimal and the inequality
is sharp due to the additional   term, i.e., there exists a non-trivial function for which the inequality becomes equality, see Definition \ref{2def1}. This section is based on \citet{FKR17}.

In section \ref{5sec3-1}, star-shaped domain and a strictly star-shaped domain with respect to $0\in\Omega$ are defined, where $\partial\Omega\in C^1$, see Definition \ref{5def0}.  Here we use more general Definitions \ref{6def1} and \ref{6def2}, when $\partial\Omega\in C^0$.
\begin{definition}
\label{6def1}
The bounded domain $\Omega\subset R^n$, $n\geq2$ with $C^0$ boundary $\partial\Omega$ is star--shaped domain with respect to a point $x_0\in\Omega$ if
every ray starting from $x_0$ intersects the boundary $\partial\Omega$ only at one point.
\end{definition}
\begin{definition}
\label{6def2}
The bounded domain $\Omega$, where $\partial\Omega\in C^0$ is a star-shaped with respect to an interior ball $B_{\varepsilon}=\{|x|<\varepsilon\} \subset\Omega$  if  $\Omega$ is star-shaped with respect to every point of the ball $B_{\varepsilon}$, see Definition \ref{6def1} and Ch. 1.1.6 in \citet{Ma85}.
\end{definition}

Let $\Omega=\{|x|<\varphi(x)\}\subset R^n$ be a star-shaped domain with respect to a small ball. Here $0\in \Omega$, $n\geq2$, $p>1$, and   $\varphi$ is a homogeneous function of the 0-th order. Note that, according to Ch. 1.1.8 in \citet{Ma85},  in this case $\varphi(x)$ is Lipschitz function on the unit sphere $S_1$ in $R^n$.

We denote by
$W_0^{1,p}(|x|^{l(1-p)},\Omega)$, $l\leq\displaystyle\frac{n-1}{p-1}, l\in R$,  the completion of
$C^\infty_0(\Omega)$ functions with respect to the norm
\begin{equation}
\label{61eq4-1}
\|u\|_{W_0^{1,p}(|x|^{l(1-p)} ,\Omega)}=\left(\int_\Omega|x|^{l(1-p)}|\nabla
u|^pdx\right)^{\frac{1}{p}}<\infty,
\end{equation}
(see \citet{Ma85}, Ch.1.1.6).

In Theorem \ref{4th1}  the  following Hardy inequality with double singular weights (\ref{4eq3}) is proved
\begin{equation}
\label{61eq50}
\begin{array}{lll}
&&\int_{\Omega}|x|^{l(1-p)}\left|\frac{\langle x,\nabla
u\rangle}{|x|}\right|^pdx
\\[2pt]
\\
&&\geq\left|\displaystyle\frac{m+l}{p'}\right|^p\displaystyle\int_{\Omega}\frac{|u|^p}{|x|^{n-m-l}
\left|1-\left(\frac{|x|}{\varphi}\right)^{m+l}\right|^p}dx,  \ \ u\in
W_0^{1,p}(|x|^{l(1-p)},\Omega),
\end{array}
\end{equation}
where we use the notations in Sect. \ref{sect4}:
$$
\begin{array}{lll}
&&\alpha=1+\frac{l}{p'}, \ \ \beta=1,  \ \ \gamma=\frac{p-1}{p},
\\
&&v=|x|^{-\frac{l}{p'}}, \ \ w=|m+l||x|^{-1-\frac{l}{p'}}\left|1-\left(\frac{|x|}{\varphi}\right)^{m+l}\right|^{-1},
\\
&&g(s(x))=\frac{1-s(x)^{m+l}}{m+l}, \ \ s(x)=\frac{|x|}{\varphi(x)}  \hbox{ and } m+l\neq0, \ \ l\leq1-m.
\end{array}
$$
 Here $m=\frac{p-n}{p-1}\neq0$, $\frac{1}{p}+\frac{1}{p'}=1$, $\langle.,.\rangle$ is the scalar product in $R^n$   and the constant  $\left|\displaystyle\frac{m+l}{p'}\right|^p$ is optimal.

In this section we generalize (\ref{61eq50}) and  prove a sharp Hardy inequality
with additional term and an optimal constant for star-shaped domains and $m+l>0$.
\subsection{Hardy inequalities with additional boundary term}
\label{6sect2}
We start with the following theorem:
\begin{theorem}\label{61th1}
Suppose $\Omega=\{|x|<\varphi(x)\}\subset R^n$, $n\geq2$ is a
star-shaped domain with respect to a small ball centered at the
origin,  $p>1$, $m=\displaystyle\frac{p-n}{p-1}$, $-m<l\leq 1-m$. Then for every $u\in
W_0^{1,p}(|x|^{l(1-p)},\Omega)$, the improved Hardy inequality
\begin{equation}
\label{61eq5}
\begin{array}{lll}
&&\left(\displaystyle\int_{\Omega}|x|^{l(1-p)}\left|\frac{\langle x,\nabla
u\rangle}{|x|}\right|^pdx\right)^{\frac{1}{p}}
\\[2pt]
\\
&\geq&\left(\displaystyle\frac{m+l}{p'}\right)\left(\displaystyle\int_{\Omega}\frac{|u|^p}{|x|^{n-m-l}
\left|\varphi^{m+l}(x)-|x|^{m+l}\right|^p}dx\right)^{\frac{1}{p}}
\\[2pt]
\\
&+&\displaystyle\frac{1}{p}\limsup_{\varepsilon\rightarrow0}\varepsilon^{1-n}
\displaystyle\int_{S_\varepsilon}\frac{|u(x)|^pdS}{\varphi^{(m+l)(p-1)}(x)}
\\[2pt]
\\
&\times&\left(\displaystyle\int_{\Omega}\frac{|u|^p}{|x|^{n-m-l}
\left|\varphi^{m+l}(x)-|x|^{m+l}\right|^p}dx\right)^{-\frac{1}{p'}},
\end{array}
\end{equation}
holds, where $S_\varepsilon=\{|x|=\varepsilon\}$.

\end{theorem}
In inequality (\ref{61eq5}) instead of the distance to zero in the denominator there is the distance to the boundary on the ray, and the constant is optimal.
\begin{remark}\rm
\label{61rem2} For $m>0$, i.e., $p>n$, the choice $l=0$ is possible in
(\ref{61eq5}) and in this case Hardy inequality (\ref{61eq5}) is true for every
$u\in W_0^{1,p}(\Omega)$.
\end{remark}
In order to prove Theorem \ref{61th1} we  need some
auxiliary results.

Fix $r<\inf_{|x|=1}\varphi(x)$ and for the annulus
$A[r,\varphi)=\{r\leq|x|<\varphi(x)\}$ we introduce the space
$W^{1,p}_0(A[r,\varphi))$ which is the completion in the norm
(\ref{61eq4-1}) for $A[r,\varphi)=\Omega\backslash \bar{B_r}$ of the
$C^{\infty}(A[r,\varphi))$ functions which are zero in a
neighborhood of the boundary $S_\varphi=\{|x|=\varphi(x)\}$ (see
\citet{Ma85}, Ch. 1.1.15 and Ch. 1.1.6), i.e., functions in $C_0^\infty(\Omega)$. The main element of the proof of Theorem \ref{61th1} is the following Proposition \ref{62pr1} and Corollary \ref{cor3_1}.
\begin{proposition}
\label{62pr1} Suppose $n\geq2$, $p>1$,
$\frac{1}{p}+\displaystyle\frac{1}{p'}=1$, $-m<l\leq 1-m$, $m=\frac{p-n}{p-1}$ and
\begin{equation}
\label{62eq8-1}
f(x)=(f_1,\ldots,f_n)=-x|x|^{-n}\left(\displaystyle\frac{\varphi^{m+l}(x)-|x|^{m+l}}{m+l}\right)^{1-p},
\end{equation}
where $f\not\equiv 0$, $f_j\in C^{1}(A[r,\varphi))$. Then $f$ satisfies the identity
\begin{equation}
\label{62eq2} -\hbox{div}f-(p-1)|f|^{p'}|x|^l=0, \ \  \hbox{ in }
 A[r,\varphi),
\end{equation}
and for every $u\in W^{1,p}_0(|x|^{l(1-p)}, A[r,\varphi))$ the inequality
\begin{equation}
\label{62eq3}
\begin{array}{l}
\left(\int_{A[r,\varphi)}|x|^{l(1-p)}\left|\frac{\langle
f,\nabla
u\rangle}{|f|}\right|^pdx\right)^{1/p}\geq\displaystyle\frac{1}{p'}
\left(\displaystyle\int_{A[r,\varphi)}|x|^l|f|^{p'}|u|^pdx\right)^{1/p}
\\[2pt]
\\
-\displaystyle\frac{1}{p}\int_{S_r}\frac{\langle
f,x\rangle}{|x|}|u|^pdS
\left(\displaystyle\int_{A[r,\varphi)}|x|^l|f|^{p'}|u|^pdx\right)^{-\frac{1}{p'}},
\end{array}
\end{equation}
holds.
\end{proposition}
\begin{proof}
Without loss of generality we suppose that $u\in C^\infty_0(\Omega)$. Simple computations give us
$$
|f|^{p'}=|x|^{m-n}\left(\frac{\varphi^{m+l}(x)-|x|^{m+l}}{m+l}\right)^{-p},
$$
$$
\begin{array}{lll}
&-&\hbox{div}f=|x|^{-n}\left\langle x,
\nabla\left(\displaystyle\frac{\varphi^{m+l}(x)-|x|^{m+l}}{m+l}\right)^{1-p}\right\rangle
\\[2pt]
\\
&=&-(p-1)|x|^{-n}\left(\displaystyle\frac{\varphi^{m+l}(x)-|x|^{m+l}}{m+l}\right)^{-p}
\left[\varphi^{m+l-1}(x)\langle
x, \nabla\varphi(x)\rangle-|x|^{m+l}\right]
\\[2pt]
\\
&=&(p-1)|x|^{m+l-n}\left(\displaystyle\frac{\varphi^{m+l}(x)-|x|^{m+l}}{m+l}\right)^{-p}=(p-1)|f|^{p'}|x|^l,
\end{array}
$$
because $\langle x, \nabla\varphi(x)\rangle=0$.

Thus we have that (\ref{62eq2}) is
satisfied. Inequality (\ref{62eq3}) follows from (\ref{3eq3}) in Corollary \ref{cor3_1} for $v=|x|^l$ and $w=0$.
\end{proof}
\begin{proposition}
\label{62pr2} Suppose $n\geq2$, $p>1$, and
$m=\displaystyle\frac{p-n}{p-1}$, $-m<l\leq 1-m$. Then for every $u\in W^{1,p}_0(|x|^{l(1-p)}, A[r,\varphi))$ the following inequalities
hold:
\begin{equation}
\label{62eq7}
\begin{array}{lll}
&&\left(\displaystyle\int_{A[r,\varphi)}|x|^{l(1-p)}\left|\frac{\langle
x,\nabla u\rangle}{|x|}\right|^pdx\right)^{\frac{1}{p}}
\\[2pt]
\\
&\geq&\displaystyle\frac{m+l}{p'}\left(\int_{A[r,\varphi)}\frac{|u|^p}{|x|^{n-m-l}
\left(\varphi^{m+l}(x)-|x|^{m+l}\right)^p}dx\right)^{\frac{1}{p}}
\\[2pt]
\\
&+&\displaystyle\frac{r^{1-n}}{p}
\displaystyle\int_{S_r}\displaystyle\frac{|u|^p}{\left(\varphi^{m+l}(x)-r^{m+l}\right)^{p-1}}dS
\\[2pt]
\\
&\times&\left(\displaystyle\int_{A[r,\varphi)}\frac{|u|^p}{|x|^{n-m-l}
\left(\varphi^{m+l}(x)-|x|^{m+l}\right)^p}dx\right)^{-\frac{1}{p'}},
\end{array}
\end{equation}
and
\begin{equation}
\label{62eq7-1}
\begin{array}{lll}
&&\displaystyle\int_{A[r,\varphi)}|x|^{l(1-p)}\left|\frac{\langle
x,\nabla u\rangle}{|x|}\right|^pdx
\\[2pt]
\\
&\geq&\left(\displaystyle\frac{m+l}{p'}\right)^p\displaystyle\int_{A[r,\varphi)}\frac{|u|^p}{|x|^{n-m-l}
\left(\varphi^{m+l}(x)-|x|^{m+l}\right)^p}dx
\\[2pt]
\\
&+&\frac{r^{1-n}}{p}\left(\frac{1}{p'}\right)^{p-1}
\displaystyle\int_{S_r}\displaystyle\frac{|u|^p}{\left(\varphi^{m+l}(x)-r^{m+l}\right)^{p-1}}dS.
\end{array}
\end{equation}
For the functions
\begin{equation}
\label{62eq8}
u_k(x)=\left(\varphi^{m+l}(x)-|x|^{m+l}\right)^{k}\Phi(x), \
\ k>\displaystyle\frac{1}{p'},
\end{equation}
where $\Phi(x)$ is a homogeneous function of the 0-th order,  inequality (\ref{62eq7}) becomes an equality.
\end{proposition}
\begin{proof}
Without loss of generality we suppose that $u(x)$ is $C^\infty$
function which is zero near the boundary
$S_\varphi=\{|x|=\varphi(x)\}$. Let us choose $f$ as
in Proposition \ref{62pr1}, see (\ref{62eq8-1}).
The proof of (\ref{62eq7}) follows from (\ref{62eq3}) with the special
choice (\ref{62eq8-1}) of $f$ since
$$
\begin{array}{lll}
K_0(u)&=&-\int_{S_r}\frac{\langle
f,x\rangle}{|x|}|u|^pdS
\\[2pt]
\\
&=&(m+l)^{p-1}r^{1-n}\int_{S_r}\frac{|u|^p}{
\left|\varphi^{m+l}(x)-r^{m+l}\right|^{p-1}}dx\geq0.
\end{array}
$$
For the proof of (\ref{62eq7-1}) we use (\ref{300eq33}).

For function $u_k(x)$ in (\ref{62eq8}) we get an equality in
(\ref{62eq7}). The proof is similar to the proof of Theorem \ref{61th2} and we omit it.
\end{proof}
\begin{proof}[Proof of Theorem \ref{61th1}]
Let $0<\varepsilon<\inf_{|x|=1}\varphi(x)$ be a small positive number and $u\in
W^{1,p}_0(|x|^{l(1-p)}, \Omega)$. Then $u\in
W^{1,p}_0(|x|^{l(1-p)},A[\varepsilon,\varphi))$, where
$A[\varepsilon,\varphi)=\{\varepsilon\leq|x|<\varphi(x)\}$.   From
Corollary \ref{cor3_1} we get the following Hardy inequality in the
annulus $A[\varepsilon,\varphi)$ for $-m<l\leq 1-m$, $v=|x|^l$, $w=0$ and $L(u), K_0(u), K(u), N(u)$ defined in (\ref{3eq9}) for $\Omega=A[\varepsilon,\varphi)$, i.e.,
\begin{equation}
\label{63eq1}
\begin{array}{lll}
L^{\frac{1}{p}}(u)&=&\left(\displaystyle\int_{A[\varepsilon,\varphi)}|x|^{l(1-p)}\left|\frac{\langle
x,\nabla u\rangle}{|x|}\right|^pdx\right)^{\frac{1}{p}}
\geq\frac{1}{p'}K^{\frac{1}{p}}(u)+\frac{1}{p}K_0(u)K^{\frac{1-p}{p}}(u)
\\[2pt]
\\
&=&\frac{m+l}{p'}\left(\int_{A[\varepsilon,\varphi)}\frac{|u|^p}{|x|^{n-m-l}
\left|\varphi^{m+l}(x)-|x|^{m+l}\right|^p}dx\right)^{\frac{1}{p}}
\\[2pt]
\\
&+&\frac{\varepsilon^{1-n}}{p}
\int_{S_{\varepsilon}}\frac{|u|^p}{\left|\varphi^{m+l}(x)-\varepsilon^{m+l}\right|^{p-1}}d S
\\[2pt]
\\
&\times&\left(\displaystyle\int_{A[\varepsilon,\varphi)}\frac{|u|^p}{|x|^{n-m-l}
\left|\varphi^{m+l}(x)-|x|^{m+l}\right|^p}dx\right)^{\frac{1-p}{p}}.
\end{array}
\end{equation}
After the limit $\varepsilon\rightarrow0$ in (\ref{63eq1}) and from the inequality
$$
\displaystyle\int_{S_{\varepsilon}}\frac{|u|^p}{\left|\varphi^{m+l}(x)-\varepsilon^{m+l}\right|^{p-1}}d S\geq\displaystyle\int_{S_{\varepsilon}}\frac{|u|^p}{\varphi^{(m+l)(p-1)}(x)}d S,
$$
we get (\ref{61eq5}).
\end{proof}

\subsection{Sharpness of  the Hardy inequalities}
\label{6sec2}
The inequality (\ref{61eq5}) becomes an equality for a class of functions defined in Theorem \ref{61th2}. Moreover, in Corollary \ref{61cor2} we show that in the special case of a ball, inequality  (\ref{61eq5}) transforms into (\ref{61eq10}) with the distance function in the denominator and the optimal constant.
\begin{theorem}
\label{61th2} Suppose $\Omega=\{|x|<\varphi(x)\}\subset R^n$, $n\geq2$ is a
star-shaped domain with respect to a small ball centered at the
origin  $p>1$, $m=\frac{p-n}{p-1}$, $-m<l\leq 1-m$. Then  Hardy
inequality (\ref{61eq5}) is an equality if $u(x)=u_k(x)$,
$$
u_k(x)=\left(\varphi^{m+l}(x)-|x|^{m+l}\right)^k \Phi(x), \
\ k>\frac{1}{p'},
$$
where $\Phi$ is a homogeneous function of the 0--th order. Moreover, the constant $\left(\displaystyle\frac{m+l}{p'}\right)^p$ is optimal for (\ref{61eq50}) and  (\ref{61eq5}).
\end{theorem}
\begin{proof}
From Theorem \ref{3th10} it follows that
(\ref{61eq5}) is an equality if and only if for a.e. $x\in\Omega$ and some constant $k_1\geq0$ identities  (\ref{3eq52})--(\ref{3eq54})  are satisfied, i.e., (\ref{61eq5}) becomes an equality if (\ref{63eq2}) holds for a.e. $x\in\Omega$.
\begin{equation}
\label{63eq2}\begin{array}{lll}
&&u\langle f, \nabla u\rangle=|u\langle f, \nabla u\rangle|,
\\[2pt]
\\
&&\langle f,\nabla u\rangle=k_1v|f|^{p'}u.
\end{array}
\end{equation}
From the choice of $f$ in (\ref{62eq8-1}) equalities (\ref{63eq2}) are equivalent to
\begin{equation}
\label{63eq3}\begin{array}{l}
-u\langle x, \nabla u\rangle=|u\langle x, \nabla u\rangle|,
\\
\langle x, \nabla u\rangle=-k_1|x|^{m+l}\left(\varphi^{m+l}(x)-|x|^{m+l}\right)^{-1}u,
\end{array}
\end{equation}
for a.e. $x\in\Omega$.

We are looking for solution of the second equation in (\ref{63eq3}) of the form
$$
u=\left(\varphi^{m+l}(x)-|x|^{m+l}\right)^{\frac{k_1}{m+l}}z.
$$
Together with (\ref{63eq3}) a simple computation gives us
$$\begin{array}{lll}
&-&k_1|x|^{m+l}\left(\varphi^{m+l}(x)-|x|^{m+l}\right)^{-1}u=\langle x, \nabla u\rangle
\\
&=&k_1\left(\varphi^{m+l}(x)-|x|^{m+l}\right)^{\frac{k_1}{m+l}-1}\left(\varphi^{m+l-1}(x)\langle x, \nabla \varphi(x)\rangle-|x|^{m+l-2}\langle x, x\rangle\right)z
\\
&+&\left(\varphi^{m+l}(x)-|x|^{m+l}\right)^{\frac{k_1}{m+l}}\langle x, \nabla z\rangle
\\
&=&-k_1|x|^{m+l}\left(\varphi^{m+l}(x)-|x|^{m+l}\right)^{\frac{k_1}{m+l}-1}z+
\left(\varphi^{m+l}(x)-|x|^{m+l}\right)^{\frac{k_1}{m+l}}
\langle x, \nabla z\rangle
\\
&=&-k_1|x|^{m+l}\left(\varphi^{m+l}(x)-|x|^{m+l}\right)^{-1}u+
\left(\varphi^{m+l}(x)-|x|^{m+l}\right)^{\frac{k_1}{m+l}}\langle x, \nabla z\rangle.
\end{array}
$$
Hence $z$ is a solution of the homogeneous first order partial differential equation
\begin{equation}
\label{63eq4}
\langle x, \nabla z\rangle=0, \ \ \hbox{ in } \Omega\backslash\{0\}.
\end{equation}
It is well known, see \citet{Ev98}, that all solutions of (\ref{63eq4}) are
in the form of $z(x)=\Phi(x)$, where $\Phi$ is a homogeneous function of the 0-th order, i.e., all solutions of (\ref{63eq3}) are
$u_k=\left(\varphi^{m+l}(x)-|x|^{m+l}\right)^k \Phi(x)$ for
$k\geq0$. When $k>\frac{1}{p'}$, then
$u_k\in W^{1,p}_0(|x|^{l(1-p)}, \Omega)$.

Let us check up that for $u_k$ inequality (\ref{61eq5}) becomes an equality. Since $\varphi(x)$ and $\Phi(x)$ are homogeneous functions of the 0-th order we have $\langle x,\nabla\varphi(x)\rangle=0$, $\langle x,\nabla\Phi(x)\rangle=0$. We make a polar change of the variables and for simplicity we use the same notations $\varphi(x), \Phi(x)$ for this functions depending only on the angular variables. We obtain the following chain of equalities for the left-hand side of (\ref{61eq5}):
 \begin{equation}
\label{63eq401}
\begin{array}{lll}
L(u_k)&=&\int_{\Omega}|x|^{l(1-p)}\left|\frac{\langle x,\nabla
u_k\rangle}{|x|}\right|^pdx
\\[2pt]
\\
&=&\left[k(m+l)\right]^p\int_{\Omega}|\Phi(x)|^p|x|^{l(1-p)+p(m+l-1)}
\left|\varphi^{m+l}(x)-|x|^{m+l}\right|^{p(k-1)}dx
\\[2pt]
\\
&=&\left[k(m+l)\right]^p\displaystyle\int_{S_1}\int^{\varphi(x)}_0
|\Phi(x)|^p\rho^{m+l-1}|\varphi^{m+l}(x)-\rho^{m+l}|^{p(k-1)}d\rho dS
\\[2pt]
\\
&=&k^p(m+l)^{p-1}\displaystyle\int_{S_1}\int^{\varphi(x)}_0
|\Phi(x)|^p|\varphi^{m+l}(x)-\rho^{m+l}|^{p(k-1)}d\rho^{m+l} dS
\\[2pt]
\\
&=&\displaystyle\frac{k^p(m+l)^{p-1}}{pk-p+1}\int_{S_1}
|\Phi(x)|^p\varphi^{(m+l)(pk-p+1)}(x)dS<\infty,
\end{array}
\end{equation}
because $pk-p+1>0$ for $k>\frac{1}{p'}$ and $m+l>0$.

Analogously, for the  terms $K_1(u_k)$, $K_2(u_k)$ in the right-hand side of (\ref{61eq5}) we obtain
\begin{equation}
\label{63eq402}
\begin{array}{lll}
K_1(u_k)&=&\int_{\Omega}\frac{|u_k|^p}{|x|^{n-m-l}
\left|\varphi^{m+l}(x)-|x|^{m+l}\right|^p}dx
\\[2pt]
\\
&=&\int_{S_1}\int^{\varphi}_0\frac{|\Phi(x)|^p\rho^{m+l-1}}{
\left|\varphi^{m+l}(x)-\rho^{m+l}\right|^{p-pk}}d\rho dS
\\[2pt]
\\
&=&\frac{1}{(m+l)(pk-p+1)}\int_{S_1}|\Phi(x)|^p
\varphi^{(m+l)(pk-p+1)}(x) dS,
\end{array}
\end{equation}
and
$$
\begin{array}{lll}
K_2(u_k)&=&\frac{1}{p}\lim_{\varepsilon\rightarrow0}\varepsilon^{1-n}
\displaystyle\int_{S_\varepsilon}\frac{|u_k|^p}{\left|\varphi^{m+l}(x)-\varepsilon^{m+l}\right|^{p-1}}dS
\\[2pt]
\\
&=&\frac{1}{p}
\int_{S_1}|\Phi(x)|^p\varphi^{(m+l)(pk-p+1)}dS.
\end{array}
$$

Thus for the right-hand side of (\ref{61eq5}) we get finally
$$
\begin{array}{lll}
K_{12}(u_k)&=&\left(\frac{m+l}{p'}K_1(u_k)+K_2(u_k)\right)K_1^{-\frac{1}{p'}}(u_k)
\\
&=&\left[\frac{(m+l)(p-1)}{(m+l)(pk-p+1)p}
+\frac{1}{p}\right]\left[(m+l)(pk-p+1)\right]^{\frac{1}{p'}}
\\
&&\times\left[\int_{S_1}|\Phi(x)|^p\varphi^{(m+l)(pk-p+1)}dS\right]^{\frac{1}{p}}
\\
&=&\frac{(m+l)^{\frac{1}{p'}}k}{(pk-p+1)^\frac{1}{p}}
\left[\int_{S_1}|\Phi(x)|^p\varphi^{(m+l)(pk-p+1)}dS\right]^{\frac{1}{p}}.
\end{array}
$$
So the left-hand side  $(L(u_k))^{1/p}$ of (\ref{61eq5}) coincides with the right-hand side $K_{12}(u_k)$ of (\ref{61eq5}). Thus (\ref{61eq5}) is an equality for $u_k(x)$.

Let us check now that the constant $\left(\frac{m+l}{p'}\right)^p$ is optimal for (\ref{61eq5}). From (\ref{61eq5}), (\ref{63eq401}) and (\ref{63eq402}) we have
$$
\begin{array}{lll}
&&\left(\frac{m+l}{p'}\right)^p\leq\int_{\Omega}|x|^{l(1-p)}\left|\frac{\langle x,\nabla
u_k\rangle}{|x|}\right|^pdx
\\[2pt]
\\
&&\times\left(\int_{\Omega}\frac{|u_k|^p}{|x|^{n-m-l}
\left|\varphi^{m+l}(x)-|x|^{m+l}\right|^p}dx\right)^{-1}
\\[2pt]
\\
&&=\frac{(m+l)(pk-p+1)k^p(m+l)^{p-1}}{(pk-p+1)}
\\[2pt]
\\
&&=\left(\frac{m+l}{p'}\right)^p(p'k)^p\rightarrow\left(\frac{m+l}{p'}\right)^p+0, \ \ \hbox{ when }k\rightarrow\frac{1}{p'}+0.
\end{array}
$$
\end{proof}
Let us illustrate Theorem \ref{61th1} and Theorem \ref{61th2} for the ball $B_R=\{|x|<R\}$ and $l=0$.

\begin{corollary}
\label{61cor1} Suppose $B_R=\{|x|<R\}\subset R^n$,
$p>n\geq 2$, $m=\frac{p-n}{p-1}$. Then for every $u\in
W^{1,p}_0(B_R)$ Hardy inequality
\begin{equation}
\label{61eq7}
\begin{array}{lll}
&&\left(\displaystyle\int_{B_R}\left|\frac{\langle x,\nabla
u\rangle}{|x|}\right|^pdx\right)^{\frac{1}{p}}
\geq\displaystyle\frac{p-n}{p}\left(\displaystyle\int_{B_R}\frac{|u|^p}{|x|^{n-m}
\left|R^m-|x|^m\right|^p}dx\right)^{\frac{1}{p}}
\\
\\
&+&\displaystyle\frac{1}{p}R^{n-p}\omega_n|u(0)|^p
\left(\displaystyle\int_{B_R}\frac{|u|^p}{|x|^{n-m}
\left|R^m-|x|^m\right|^p}dx\right)^{-\frac{1}{p'}},
\end{array}
\end{equation}
holds, where $\omega_n$ is the (n-1)-dimensional measure of the unite sphere $S_1$.

For
$$
u_k(x)=\left(R^m-|x|^m\right)^k\Phi(x), \ \
k>\frac{1}{p'},
$$
where $\Phi$ is a homogeneous function of the 0--th order,  (\ref{61eq7}) becomes an
equality, i.e., (\ref{61eq7}) is sharp and the constant $\frac{p-n}{p}$ is optimal.
\end{corollary}
\begin{proof}
Corollary \ref{61cor1} follows from Theorem \ref{61th1} for
$\varphi(x)=R$, the constant $\frac{p-n}{p}$ in (\ref{61eq7}) is optimal and inequality (\ref{61eq7}) is sharp  due to Theorem \ref{61th2}.
\end{proof}
As it is shown in Sect. \ref{2sec1-2}, different forms of Hardy inequalities with additional term and optimal constant are considered in \citet{BM97, HHL02, Ti04, FMT06, AW07}, etc. In  Hardy inequality (\ref{61eq7}), the weights in the leading term of the right-hand side have singularities at $0$ and on the boundary of the ball $B_R$. This inequality is not in the `linear' form and it is sharp. Moreover, in Corollary \ref{61cor2}  we obtain the inequality (\ref{61eq10}) where the leading term in the right-hand side is written with the distance function $d(x)$ and an additional term which depends on the value of the function at $0$.

Inequalities (\ref{61eq5}) and  (\ref{61eq7}) are sharp, but they are not in a `linear' form. Using Young inequality, we can get a `linear' form of these inequalities.
\begin{theorem}
\label{61th3}
Suppose   $\Omega=\{|x|<\varphi(x)\}\subset R^n$, $n\geq2$ is a
star-shaped domain with respect to a small ball centered at the
origin,  $p>1$, $m=\displaystyle\frac{p-n}{p-1}$, $-m<l\leq 1-m$. Then the following  Hardy inequality holds in $\Omega$:
$$
\begin{array}{lll}
&&\displaystyle\int_{\Omega}|x|^{l(1-p)}\left|\frac{\langle x,\nabla
u\rangle}{|x|}\right|^pdx
\geq\left(\displaystyle\frac{m+l}{p'}\right)^p\displaystyle\int_{\Omega}\frac{|u|^p}{|x|^{n-m-l}
\left|\varphi^{m+l}(x)-|x|^{m+l}\right|^p}dx
\\[2pt]
\\
&+&\left(\frac{m+l}{p'}\right)^{p-1}\frac{1}{p}
\limsup_{\varepsilon\rightarrow0}\varepsilon^{1-n}
\displaystyle\int_{S_\varepsilon}\frac{|u(x)|^pdS}{\varphi^{(m+l)(p-1)}(x)},
 \ \ u\in W_0^{1,p}(|x|^{l(1-p)},\Omega),
\end{array}
$$
When $p>n, l=0$ then in $B_R$ the inequality
\begin{equation}
\label{61eq71}
\begin{array}{lll}
&&\displaystyle\int_{B_R}\left|\frac{\langle x,\nabla
u\rangle}{|x|}\right|^pdx
\geq\left(\displaystyle\frac{p-n}{p}\right)^p\displaystyle\int_{B_R}\frac{|u|^p}{|x|^{n-m}
\left|R^m-|x|^m\right|^p}dx
\\[2pt]
\\
&+&\left(\frac{p-n}{p}\right)^{p-1}\frac{R^{n-p}}{p}\omega_n|u(0)|^p,
 \ \ u\in W_0^{1,p}(B_R),
\end{array}
\end{equation}
holds.
\end{theorem}
\begin{proof}
Since $K_0(u)\geq 0$, we have from  (\ref{300eq33})  that
\begin{equation}
\label{62eq5-3}
L(u)\geq\left(\frac{1}{p'}\right)^pK(u)+\left(\frac{1}{p'}\right)^{p-1}K_0(u).
\end{equation}
The rest of the proof follows from (\ref{62eq5-3}), Theorem \ref{61th1} and Corollary \ref{61cor1}.
\end{proof}
As a corollary of Theorems \ref{61th1} and \ref{61th3} we get
\begin{corollary}
\label{61cor2} If $m=\frac{p-n}{p-1}>0$, then
\begin{equation}
\label{61eq9} \displaystyle\left(\frac{p-n}{p}\right)^p\int_{B_R}\frac{|u|^p}{|x|^{n-m}
\left|R^m-|x|^m\right|^p}dx\geq\displaystyle\left(\frac{1}{p'}\right)^p\int_{B_R}\frac{|u|^p}{d^p(x)}dx
\end{equation}
and correspondingly
\begin{equation}
\label{61eq10-0}
\begin{array}{lll}
&&\left(\int_{B_R}\left|\frac{\langle x,\nabla
u\rangle}{|x|}\right|^pdx\right)^{\frac{1}{p}}
\geq\frac{1}{p'}\left(\int_{B_R}\frac{|u|^p}{d^p(x)}dx\right)^{\frac{1}{p}}
\\[2pt]
\\
&+&\frac{R^{n-p}}{p}\omega_n|u(0)|^p
\left(\int_{B_R}\frac{|u|^p}{|x|^{n-m}
\left|R^m-|x|^m\right|^p}dx\right)^{-\frac{1}{p'}},
\end{array}
\end{equation}
\begin{equation}
\label{61eq10}
\displaystyle\int_{B_R}\left|\frac{\langle x,\nabla
u\rangle}{|x|}\right|^pdx
\geq\left(\displaystyle\frac{1}{p'}\right)^p\int_{B_R}\frac{|u|^p}{d^p(x)}dx
+\left(\frac{p-n}{p}\right)^{p-1}\frac{R^{n-p}}{p}\omega_n|u(0)|^p,
\end{equation}
hold for every  $u\in W^{1,p}_0(B_R)$. Moreover, the constant $\left(\frac{1}{p'}\right)^p$ in (\ref{61eq10-0}) is optimal.
\end{corollary}
\begin{proof}
It is enough to prove the inequality
$$
\left(\displaystyle\frac{p-n}{p}\right)^p|x|^{m-n}\left(R^m-|x|^m\right)^{-p}\geq \left(\displaystyle\frac{p-1}{p}\right)^p\left(R-|x|\right)^{-p},
$$
or equivalently
\begin{equation}
\label{63eq101}
(p-n)(R-\rho)\geq (p-1)\rho^{\frac{n-1}{p-1}}(R^m-\rho^m),
\end{equation}
for $|x|=\rho$ and $n-m=\displaystyle\frac{(n-1)p}{p-1}$. The function
$$
\begin{array}{lll}
h(\rho)&=&(p-n)(R-\rho)-(p-1)\rho^{\frac{n-1}{p-1}}(R^m-\rho^m)
\\[2pt]
\\
&=&(p-n)(R-\rho)-(p-1)\left(R^m \rho^{\frac{n-1}{p-1}}-\rho\right),
\end{array}
$$
is a decreasing one for $\rho\in[0,R]$ because
$$
\begin{array}{lll}
h'(\rho)&=&(n-p)-(n-1)R^m\rho^{\frac{n-1}{p-1}-1}+p-1
\\[2pt]
\\
&=&(n-1)\left[1-\left(\displaystyle\frac{R}{\rho}\right)^m\right]\leq 0,
\end{array}
$$
and $m>0$. Since $h(R)=0$ inequality (\ref{63eq101}) is satisfied.

To obtain inequalities (\ref{61eq10-0}) and (\ref{61eq10}) we replace (\ref{61eq9}) in  (\ref{61eq7}) and (\ref{61eq71}). The optimality of the constants $\frac{1}{p'}$  in (\ref{61eq10-0}) and  $\left(\frac{1}{p'}\right)^p$ in  (\ref{61eq10}) follows from Corollary \ref{61cor1}.
\end{proof}
\section{Estimates from below for the first eigenvalue of the p-Laplacian}
\label{sect7}

In this section we  give an application of  Hardy inequalities   for  the estimate from below of the
first eigenvalue $\lambda_{p,n}$ of the p--Laplacian
$\Delta_pu=\hbox{div}(|\nabla u|^{p-2}\nabla u)$, $p>1$ in a bounded
simply connected domain $\Omega\subset R^n$, $n\geq2$ with smooth
boundary $\partial\Omega$
\begin{equation}
\label{2eq18} \left\{\begin{array}{l} -\Delta_pu=\lambda_{p,n}|u|^{p-2}u \
\  \hbox{ in } \Omega,
\\[2pt]
\\
u=0 \ \ \hbox{ on } \partial \Omega .
\end{array}\right.
\end{equation}
\subsection{Existing analytical estimates from below}
\label{sect7-1}
Here we listed some estimates of $\lambda_{p,n}$.
 The first eigenvalue
$\lambda_{p,n}(\Omega)$ can be characterized through
Reyleigh quotient, see \citet{Ch70, Li94}
\begin{equation}
\label{2eq19}
\lambda_{p,n}(\Omega)=\inf_{u\in
W^{1,p}_0(\Omega)}\displaystyle\frac{\displaystyle\int_{\Omega}|\nabla
u|^pdx}{\displaystyle\int_{\Omega}|u|^pdx},
\end{equation}
and $\lambda_{p,n}(\Omega)$ is simple, i.e., the first eigenfunction $u_{p,n}(x)$
is unique up to multiplication with non-zero constant $C$. Moreover,
$u_{p,n}(x)$ is positive in $\Omega$, $u_{p,n}(x)\in W^{1,p}_0(\Omega)\cap C^{1,\delta}(\bar{\Omega})$ for some $\delta\in(0,1)$ (see for example
\citet{BK02} and the references therein).

Analytical values of $\lambda_{p,n}$ are known only for $p>1$ and $n=1$ or $p=2$ and $n\geq2$.
For $p>1$ and $n=1$, $\Omega=(a,b)$, see  \citet{Ot84}, the analytical value of $\lambda_{p,1}(\Omega)$ is
$$
\lambda_{p,1}(\Omega)=(p-1)\left(\frac{\pi}{(b-a)p\sin\frac{\pi}{p}}\right)^p.
$$
For $n\geq2$  in the case $p=2$, i.e., for the Laplace operator,
the value of $\lambda_{2,n}(\Omega)$ is known by analytical
formulae for domains $\Omega$ with simple geometry like a ball, a spherical shell, a parallelepiped etc.
Numerical approximations have been done for more general domains, see \citet{Vl71} and the review by \citet{GN13}. For example, if $\Omega$ is a ball centered at zero, $B_R\subset R^n$ then
$$
\lambda_{2,n}(B_R)=\left(\frac{\mu_1^{(\alpha)}}{R}\right)^2, \
\ \ \alpha=\frac{n}{2}-1,
$$
where $\mu_1^{(\alpha)}$ is the first positive zero of the Bessel
function $J_{\alpha}$.

If $p\neq2$, the
explicit value of $\lambda_{p,n}(\Omega)$ is not known even for
domains $\Omega$ like a ball or a cube. That is why an explicit lower
bound for $\lambda_{p,n}(\Omega)$ is an important task.

For this purpose  the Faber--Krahn theorem simplifies the estimate of the first eigenvalue for arbitrary domain to the estimate in a ball. Let us recall the  Faber--Krahn  inequality which
gives an estimate from below of $\lambda_{p,n}(\Omega)$ for
arbitrary bounded domain $\Omega\subset R^n$ with
$\lambda_{p,n}(\Omega^\ast)$, where $\Omega^\ast$ is the
n-dimensional ball of the same volume as $\Omega$, see \citet{Li94,
Bh99, Hu97, KF03}.
In \citet{KF03}
is proved that among all domains $\Omega$ of a given n-dimensional volume the
ball $\Omega^\ast$ with the same volume as $\Omega$ minimizes
$\lambda_{p,n}(\Omega)$, in other words
$$
 \lambda_{p,n}(\Omega)\geq\lambda_{p,n}(\Omega^\ast).
$$

\subsubsection*{Estimate with a Cheeger's constant}

One of the first  lower bounds for $\lambda_{p,n}(\Omega)$ is
based on  Cheeger's constant
$$
h(\Omega)=\inf_{D\subset\Omega}\displaystyle\frac{|\partial D|}{|D|}.
$$
Here $D$ varies over all smooth sub-domains of $\Omega$ whose
boundary $\partial D$ does not touch $\partial\Omega$, where
$|\partial D|$ and $|D|$ are the (n-1)- and n-dimensional
Lebesgue measure of $\partial D$ and $D$ respectively.

In \citet{Ch70} for $p=2$ and in \citet{LW97} for
$p>1$ it was proved that
 the first eigenvalue of
(\ref{2eq18}) can be estimated from below via
\begin{equation}
\label{2eq20}
\lambda_{p,n}(\Omega)\geq\left(\displaystyle\frac{h(\Omega)}{p}\right)^p.
\end{equation}

Inequality (\ref{2eq20}) is sharp for $p\rightarrow1$, because
$\lambda_{p,n}(\Omega)$ converges to the Cheeger's constant
$h(\Omega)$, see \citet{KF03}, Corollary 6.

The Cheeger's constant $h(\Omega)$ is known only for special
domains. For example, if $\Omega$ is a ball $B_R\subset R^n$, then
$h(\Omega)=\displaystyle\frac{n}{R}$ and (\ref{2eq20}) gives
the following lower bound for $\lambda_{p,n}(B_R)$, see \citet{KF03},
\begin{equation}
\label{2eq200}
\lambda_{p,n}(B_R)\geq\Lambda^{(1)}_{p,n}(B_R)=\left(\displaystyle\frac{n}{pR}\right)^p,
\hbox{ for } p>1, \ \ n\geq 2.
\end{equation}
Thus combining the above results the following inequality holds for
$p\rightarrow1$, see \citet{KF03}, Remark 5,
\begin{equation}
\label{2eq21} \lambda_{1,n}(\Omega)\geq
n\left(\displaystyle\frac{\omega_n}{|\Omega|}\right)^{1/n}=\Lambda^{(1)}_{1,n}(\Omega),
\ \ n\geq2,
\end{equation}
where $\omega_n$ is the volume of the unit ball in $R^n$. If
$\Omega$ is a ball, then (\ref{2eq21}) becomes an equality.

In the other limit case $p\rightarrow\infty$ the result in
\citet{JLM99} says that
$$
\lambda_{\infty,n}(\Omega)=\lim_{p\rightarrow\infty}\left(\lambda_{p,n}(\Omega)\right)^{1/p}
=\max\{\hbox{dist}(x,\partial\Omega),
x\in \Omega\}^{-1}.
$$
In particular for $\Omega=B_R$
$$
\lambda_{\infty,n}(B_R)=\lim_{p\rightarrow\infty}\left(\Lambda^{(1)}_{p,n}(B_R)\right)^{1/p}=\displaystyle\frac{1}{R}.
$$

\subsubsection*{Estimate with  Picone's identity}

In \citet{BD12}, Theorem 2, and in \citet{BD13}, Theorem 2,  by  Picone's identity the following estimate for $p>1$ was proved
$$
\lambda_{p,n}(B_R)\geq\left\{\begin{array}{l}
\Lambda^{(2,1)}_{p,n}(B_R)=\frac{n}{R^p}\left(\frac{p}{p-1}\right)^{p-1},
\\
\Lambda^{(2,2)}_{p,n}(B_R)=\frac{np}{R^p}.
\end{array}\right.
$$
Since
$$
\max\left\{\Lambda^{(2,1)}_{p,n}(B_R), \Lambda^{(2,2)}_{p,n}(B_R)\right\}=\left\{\begin{array}{l}
\Lambda^{(2,1)}_{p,n}(B_R), \ \ \hbox{ for } 1<p<2,
\\
\Lambda^{(2,2)}_{p,n}(B_R), \ \ \hbox{ for } p\geq2,
\end{array}\right.
$$
see Proposition \ref{10prop4-1},  the estimate
\begin{equation}
\label{2eq201}
\lambda_{p,n}(B_R)\geq\Lambda^{(2)}_{p,n}(B_R)=\left\{\begin{array}{l}
\Lambda^{(2,1)}_{p,n}(B_R)=\frac{n}{R^p}\left(\frac{p}{p-1}\right)^{p-1}, \ \ \hbox{ for } 1<p<2,
\\
\Lambda^{(2,2)}_{p,n}(B_R)=\frac{np}{R^p}, \ \ \hbox{ for } p\geq2,
\end{array}\right.
\end{equation}
holds.

\subsubsection*{Estimate with Sobolev constant}

It is not difficult to estimate
$\lambda_{p,n}(\Omega)$ from below in a bounded domain for $1<p<n$
by the well-known Sobolev and H\"older inequalities
\begin{equation}
\label{2eq22}
 \|\nabla u\|_p\geq C_{n,p}\|u\|_{\frac{np}{n-p}}\geq
C_{n,p}\|u\|_p|\Omega|^{-1/n}.
\end{equation}

The best  Sobolev's constant $C_{n,p}$ is obtained in \citet{Au76}
and \citet{Ta76}. For  more details see \citet{Ma85} and \citet{LXZ11}
$$
C_{n,p}=n^{1/p}\omega_n^{1/n}\left(\frac{n-p}{p-1}\right)^{\frac{p-1}{p}}
\left[\frac{\Gamma\left(\frac{n}{2}\right)\Gamma\left(n+1-\frac{n}{p}\right)}{\Gamma(n)}\right]^{1/n}.
$$
 From
(\ref{2eq22}) the estimate from below of the first eigenvalue for $\Omega=B_R$ becomes
$$
\begin{array}{lll}
\lambda_{p,n}(B_R)&\geq&\displaystyle\frac{n}{R^p}
\left(\displaystyle\frac{n-p}{p-1}\right)^{p-1}
\left[\displaystyle\frac{\Gamma\left(\frac{n}{2}\right)\Gamma\left(n+1-\frac{n}{p}\right)}
{\Gamma(n)}\right]^{p/n}
\\
&=&\Lambda^{(S)}_{p,n}(B_R), \hbox{ for } 1<p<n.
\end{array}
$$

\subsubsection*{Lindqvist's estimate in parallelepiped}

In the
parallelepiped
$$
P=\{x\in R^n, 0<x_j<a_j, j=1,2,\ldots,n\} \ \ \hbox{ with }
a_{min}=\min_{i\leq i\leq n}a_i
$$
for $p>n$ we have
the estimate
$$
 \lambda_{p,n}(P)\geq\frac{p}{a_{min}^p},
$$
by \citet{Li95}.
If $\Omega$ is an arbitrary bounded domain in $R^n$ and $R$ is the
radius of the largest ball inscribed in the smallest parallelepiped
(with minimal $a_{min}$) containing $\Omega$, then
\begin{equation}
\label{2eq23}
\lambda_{p,n}(\Omega)\geq\frac{p}{R^p}=\Lambda^{(L)}_{p,n}(B_R),
\hbox{ for } p>n.
\end{equation}

\subsection*{Numerical estimates}

In \citet{BEM09, BBEM12}, different numerical methods for computing $\lambda_{p,n}(\Omega)$ inspired by the inverse
power method in finite dimensional algebra are developed. By means of iterative technique the authors
define two sequences of functions. One of the sequences is  monotone decreasing,
the other one is monotone increasing. The first eigenvalue
$\lambda_{p,n}(\Omega)$ is between the limits of these sequences. In
the case of a ball  the two limits are equal and
$\lambda_{p,n}(\Omega)$ coincides with them.

In \citet{LW97}  a finite element technique for numerical
approximation of the first eigenfunction and the first eigenvalue of
 (\ref{2eq18}) is used.

\subsection{Estimates from below of $\lambda_{p,n}$ using Hardy inequalities}

We prove  several analytical  bounds from below of the first eigenvalue of the p--Laplacian in bounded domains  using different Hardy inequalities with weights derived in section  \ref{sect3}. The section is based on the results in \citet{FKR14c, FKR15b, FKR15c, KR19a,KR19b}.

For this purpose we use the following   Faber--Krahn theorem.

\begin{theorem}[\citet{KF03}]
\label{21th1} Among all domains of a given n-dimensional volume the
ball $\Omega^\ast$ with the same volume as $\Omega$ minimizes every
$\lambda_{p,n}(\Omega)$, in other words
\begin{equation}
\label{21eq3} \lambda_{p,n}(\Omega)\geq\lambda_{p,n}(\Omega^\ast).
\end{equation}
\end{theorem}
Thus, from (\ref{21eq3}) it is enough for us  to find an estimate for $\lambda_{p,n}$ only in the ball $\Omega^\ast=B_R$.

Hardy inequalities are with weights singular
either at some interior point of $\Omega$, usually at the origin, or
with weights singular on the boundary or  combined double
singularities at $0$ and at $\partial\Omega$. We will apply these
three types of Hardy inequalities in order to estimate from below the
first eigenvalue $\lambda_{p,n}(\Omega)$. We are  concentrating only
on those Hardy's inequalities (among the large number of results in the
literature) which are with explicitly given constants.

Let us note that from the classical Hardy inequality, see (\ref{2eq3}) using the Reyleigh quotient (\ref{2eq19}),
we get immediately  the estimate
\begin{equation}
\label{38eq2} \lambda_{p,n}(B_R)\geq
\left|\frac{n-p}{pR}\right|^p=\Lambda^{(H)}_{p,n}(B_R), \hbox{ for }
n\geq 2, p>1, n\neq p.
\end{equation}

\subsubsection*{Estimates by means of  Hardy inequalities with double singular weights}
\label{11sect1}

From (\ref{7eq16}), (\ref{7eq161}) and (\ref{7eq17}) for $u\in W^{1,p}_0(B_R)$ ignoring the boundary terms  we have  Hardy inequalities
\begin{equation}
\label{74eq2}
\int_{B_R}\left|\nabla
u\right|^pdx
\geq\left|\frac{p-n}{p}\right|^p\int_{B_R}\frac{|u|^p}{|x|^{n-m}
\left|R^m-|x|^m\right|^p}dx, \ \  \hbox{ for } p\neq n,
\end{equation}

\begin{equation}
\label{74eq3}
\int_{B_R}\left|\nabla
u\right|^ndx
\geq\left(\frac{n-1}{n}\right)^n\int_{B_R}\frac{|u|^n}{|x|^n
\left|\ln\frac{R}{|x|}\right|^n}dx, \ \  \hbox{ for } p=n,
\end{equation}
where $p>1$, $n\geq2$, $m=\frac{p-n}{p-1}$.

With the estimates (\ref{74eq2}) and (\ref{74eq3}) we obtain the following estimate for $\lambda_{p,n}(B_R)$.
\begin{theorem}
\label{7th5}
For every $n\geq2$, $p>1$  the following estimates hold:
\begin{itemize}
\item[(i)] $\quad$ If $p\neq n$, then
\begin{equation}
\label{74eq7}
\lambda_{p,n}(B_R)\geq\left(\displaystyle\frac{1}{pR}\right)^{p}
\left[\frac{(n-1)^{n-1}}{(p-1)^{p-1}}\right]^{\frac{p}{n-p}}.
\end{equation}
\item[(ii)] $\quad$ If p=n, then
\begin{equation}
\label{74eq8}
\lambda_{n,n}(B_R)\geq\left(\frac{n-1}{nR}\right)^ne^n.
\end{equation}
\end{itemize}
\end{theorem}
\begin{proof}
(i)$\quad$
If $|x|=\rho\in[0,R)$, then for every $x\in B_R$ and $p\neq n$ we get for the right-hand side of (\ref{74eq2}) the estimate
$$
\int_{B_R}\frac{|u|^p}{|x|^{n-m}
\left|R^m-|x|^m\right|^p}dx
\geq\inf_{\rho\in(0,R)}\left(\rho^{n-m}|R^m-\rho^m|^p\right)^{-1}
\int_{B_R}|u|^pdx.
$$
Further on we will use the identities
\begin{equation}
\label{74eq00}\begin{array}{lll}
&&m-n=(m-1)p=(1-n)\frac{p}{p-1}<0,\ \ 1-m=\frac{n-1}{p-1}>0,
\\[2pt]
\\
&&\frac{n-m}{m}=\frac{(1-m)p}{m}=\frac{(n-1)p}{p-n}.
\end{array}
\end{equation}

Now applying the definition of $\lambda_{p,n}(B_R)$ with Reyleigh quotient  from (\ref{2eq19}) and (\ref{74eq00})
we have
\begin{equation}
\label{74eq4}\begin{array}{lll}
\lambda_{p,n}(B_R)&\geq&\left|\frac{p-n}{p}\right|^p\inf_{\rho\in(0,R)}\left[\rho^{1-m}|R^m-\rho^m|\right]^{-p}
\\[2pt]
\\
&=&\left|\frac{p-n}{p}\right|^p\left[\sup_{\rho\in(0,R)}\left(\rho^{1-m}|R^m-\rho^m|\right)\right]^{-p}.
\end{array}
\end{equation}
For the function $z(\rho)=\rho^{1-m}|R^m-\rho^m|$  in  the interval
$(0,R)$ we have
$$
z'(\rho)=\left[(1-m)R^m\rho^{-m}-1\right]\hbox{sgn}(m),
$$
and $z'(\rho)=0$ only at the point $\rho_0=R(1-m)^{1/m}=R\left(\frac{n-1}{p-1}\right)^{1/m}$. For $m>0$, i.e., $p>n$ we have $\frac{n-1}{p-1}<1$ and hence $\left(\frac{n-1}{p-1}\right)^{\frac{1}{m}}<1$ while for $m<0$, i.e., $p<n$  the inequality  $\frac{n-1}{p-1}>1$ holds and hence from $m<0$ we get $\left(\frac{n-1}{p-1}\right)^{\frac{1}{m}}<1$.

Since $0<\left(\frac{n-1}{p-1}\right)^{1/m}<1$ for every $m\neq0$, then $0<\rho_0<R$ and from

$z''(\rho_0)=-|m|(1-m)R^m\rho_0^{-m-1}<0$ it follows that the function $z(\rho)$ has a maximum at the point $\rho_0$ and
\begin{equation}
\label{74eq5}
z(\rho_0)=R\left(\displaystyle\frac{n-1}{p-1}\right)^{\frac{n-1}{p-n}}\left|\frac{p-n}{p-1}\right|.
\end{equation}
Hence from (\ref{74eq4}) and (\ref{74eq5}) we get
$$
\begin{array}{lll}
\lambda_p(B_R)&\geq&\left(\displaystyle\frac{p-1}{p}\right)^p
\left(\displaystyle\frac{n-1}{p-1}\right)^{-\frac{(n-1)p}{p-n}}R^{-p}
\\
&=& \left(\displaystyle\frac{1}{Rp}\right)^p
\left[\displaystyle\frac{(n-1)^{n-1}}{(p-1)^{p-1}}\right]^{\frac{p}{n-p}}.
\end{array}
$$

(ii)$\quad$
As in the proof of (\ref{74eq7}) for the right-hand side of (\ref{74eq3}) we get the estimate
$$
\int_{B_R}\frac{|u|^n}{|x|^n
\left|\ln\frac{R}{|x|}\right|^n}dx
\geq\inf_{\rho\in(0,R)}\left(\rho\ln\frac{R}{\rho}\right)^{-n}\int_{B_R}|u|^ndx.
$$
From (\ref{74eq3}) and (\ref{2eq19}) we obtain
\begin{equation}
\label{74eq400}\begin{array}{lll}
\lambda_{n,n}(B_R)&\geq&\left(\frac{n-1}{n}\right)^n\inf_{\rho\in(0,R)}\left(\rho\ln\frac{R}{\rho}\right)^{-n}
\\[2pt]
\\
&=&\left(\frac{n-1}{n}\right)^n\left[\sup_{\rho\in(0,R)}\rho\ln\frac{R}{\rho}\right]^{-n}.
\end{array}
\end{equation}
For the function $y(\rho)=\rho\ln\frac{R}{\rho}$  in  the interval
$(0,R)$ we have

$y'(\rho)=\ln\frac{R}{\rho}-1$ and $y'(\rho)=0$ only at the point $\rho_1=Re^{-1}\in(0,R)$. Since $y''(\rho_1)=-\frac{1}{\rho_1}<0$, the function $y(\rho)$ has a maximum at $\rho_1$ and $y(\rho_1)=Re^{-1}$.

Hence from (\ref{74eq400}) we obtain
$$
\lambda_{n,n}(B_R)\geq\left(\frac{n-1}{n}\right)^n\left(Rr_1^{-1}\right)^{-n}=\left(\frac{n-1}{Rn}\right)^ne^n.
$$
\end{proof}

\subsubsection*{Estimates by means of Hardy inequalities with additional logarithmic term}
\label{11sec3}

In this section we  estimate from below the first eigenvalue of
the p--Laplacian in $B_R\subset R^n$,   $n\geq2$,
$p>n$, $m=\frac{p-n}{p-1}>0$ using   Hardy inequality (\ref{9eq14}), i.e.,
\begin{equation}
\label{9eq140}
\begin{array}{lll}
&&\int_{B_R}\left|\nabla u\right|^pdx\geq\int_{B_R}\left|\frac{\langle x, \nabla u\rangle}{|x|}\right|^pdx
\\[2pt]
\\
&\geq&\left(\frac{p-n}{p}\right)^p\int_{B_R}
\left[1+\frac{p}{2(p-1)}\frac{1}{\ln^2\frac{R^m-|x|^m}{e\tau_0 R^m}}\right]
\frac{|u|^p}{|x|^{(n-1)p'}|R^m-|x|^m|^p}dx
\end{array}
\end{equation}
where $\tau=e^{\frac{1}{y_0}-1}$ and
\begin{equation}
\label{90eq140}
y_0=\frac{1-\sqrt{1+4|a|}}{-2|a|}=2\left(1+\sqrt{\frac{5p-7}{3(p-1)}}\right)^{-1} \ \  \hbox{for } a=-\frac{p-2}{6(p-1)}.
\end{equation}
\begin{theorem}
\label{9prop2}
For every ball $B_R\in R^n$, $n\geq2$, $p>n$, $m=\frac{p-n}{p-1}$ the estimate
\begin{equation}
\label{9eq30}\begin{array}{lll}
\lambda_{p,n}(B_R)&\geq& \left(\frac{1}{pR}\right)^p\left[\frac{(p-1)^{p-1}}{(n-1)^{n-1}}\right]^{\frac{p}{p-n}}
\\[2pt]
\\
&\times&
\left\{1+\frac{p}{4(p-1)}\left[1+\sqrt{\frac{5p-7}{3(p-1)}}-2\ln m -2\ln \tau\right]^{-2}\right\},
\end{array}
\end{equation}
holds, where
\begin{equation}
\label{9eq28}
\tau=1-4(1-m)\left[p\left(1+\sqrt{\frac{5p-7}{3(p-1)}}-2\ln m\right)-4m\right]^{-1}\in(0,1).
\end{equation}
\end{theorem}
\begin{proof}
Suppose that $\int_{B_R}|u|^pdx=1$ and with the notation $\varepsilon=\frac{1}{e\tau_0}$, i.e.,
\begin{equation}
\label{90eq28}
\ln\varepsilon=-\frac{1}{y_0}=-\frac{1}{2}\left(1+\sqrt{\frac{5p-7}{3(p-1)}}\right),
\end{equation}
from (\ref{90eq140}) and from (\ref{9eq140}) we obtain the estimate
\begin{equation}
\label{90eq24}
\begin{array}{lll}
\lambda_{p,n}(B_R)
&\geq&\left(\frac{p-n}{pR}\right)^p\inf_{\rho\in[0,R]}\left\{\frac{1}{\left(\frac{\rho}{R}\right)^{n-m}\left(1-\left(\frac{\rho}{R}\right)^m\right)^p}\right.
\\[2pt]
\\
&&\left.\times\left[1+\frac{p}{2(p-1)}
\frac{1}{\ln^2\varepsilon\left(1-\left(\frac{\rho}{R}\right)^m\right)}\right]\right\}
\\[2pt]
\\
&=&\left(\frac{p-n}{pR}\right)^p\inf_{z\in[0,1]}\left\{\left[1+\frac{p}{2(p-1)}
\frac{1}{\ln^2\varepsilon z}\right]
\frac{1}{z^p(1-z)^{\frac{n-m}{m}}}\right\}
\\[2pt]
\\
&\geq&\left(\frac{p-n}{pR}\right)^p\left\{\inf_{z\in[0,1]}\frac{1}{z^p(1-z)^{\frac{n-m}{m}}}\right.
\\[2pt]
\\
&+&\left.\inf_{z\in[0,1]}
\frac{p}{2(p-1)}
\frac{1}{\ln^2\varepsilon z}
\frac{1}{z^p(1-z)^{\frac{n-m}{m}}}\right\}
\\[2pt]
\\
&=& \left(\frac{p-n}{pR}\right)^p\left(I_1+\frac{p}{2(p-1)}I_2\right),
\end{array}
\end{equation}
where
$$
I_1=\inf_{z\in[0,1]}\frac{1}{z^p(1-z)^{\frac{n-m}{m}}}
=\left[\sup_{z\in[0,1]}z^{p}(1-z)^{\frac{n-m}{m}}\right]^{-1},
$$
\begin{equation}
\label{9eq25}
I_2=\inf_{z\in[0,1]}\frac{1}{z^p(1-z)^{\frac{n-m}{m}}\ln^2\varepsilon z}
=\left[\sup_{z\in[0,1]}z^{p/2}(1-z)^{\frac{n-m}{2m}}(-\ln\varepsilon z)\right]^{-2}.
\end{equation}
Since $\frac{n-m}{m}>0$ from (\ref{74eq00}), the function $h(z)=z^p(1-z)^{\frac{n-m}{m}}$ satisfies the conditions $h(0)=h(1)=0$ and
$$
\begin{array}{lll}
h'(z)&=&pz^{p-1}(1-z)^{\frac{n-m}{m}}-\frac{n-m}{m}z^p(1-z)^{\frac{n-m}{m}-1}
\\[2pt]
\\
&=&pz^{p-1}(1-z)^{\frac{n-m}{m}-1}\left(1-\frac{z}{m}\right),
\end{array}
$$
so it follows that $h(z)$ has a maximum at the point $m\in(0,1)$. Hence
$$
\sup_{z\in[0,1]}h(z)=h(m)=m^p(1-m)^{\frac{n-m}{m}}=\left(\frac{p-n}{p-1}\right)^p
\left(\frac{n-1}{p-1}\right)^{\frac{p(n-1)}{p-n}},
$$
and we get
$$
\begin{array}{lll}
\left(\frac{p-n}{pR}\right)^pI_1&=&\left(\frac{p-n}{pR}\right)^p\left(\frac{p-n}{p-1}\right)^{-p}
\left(\frac{n-1}{p-1}\right)^{-\frac{p(n-1)}{p-n}}
\\[2pt]
\\
&=&\left(\frac{1}{pR}\right)^p\left(\frac{(n-1)^{n-1}}{(p-1)^{p-1}}\right)^{\frac{p}{n-p}}
\end{array}
$$
which gives the same estimate from below for $\lambda_{p,n}(B_R)$ as in (\ref{74eq7}).

Let us estimate $I_2$. For the function

$G(z)=z^{\frac{p}{2}}(1-z)^{\frac{n-m}{2m}}(-\ln\varepsilon z)$ we have $G'(z)=z^{\frac{p}{2}-1}(1-z)^{\frac{n-m}{2m}-1}g(z)$,

where
$$
\begin{array}{lll}
g(z)&=&\frac{p}{2}(1-z)(-\ln\varepsilon z)-1+z-\frac{n-m}{2m}z(-\ln\varepsilon z)
\\[2pt]
\\
&=&\frac{1}{2}\left[p-\left(p+\frac{p(n-1)}{p-n}\right)\right](-\ln\varepsilon z)-1+z
\\[2pt]
\\
&=&\frac{p}{2}\left(1-\frac{z}{m}\right)(-\ln\varepsilon z)-1+z.
\end{array}
$$
Simple computations give us

$g'(z)=\frac{p}{2}\left(\frac{1}{m}-\frac{1}{z}+\frac{1}{m}\ln\varepsilon z\right)+1$, $g''(z)=\frac{p}{2}\left(\frac{1}{mz}+\frac{1}{z^2}\right)>0$, for $z\in[0,1]$, and hence from the monotonicity of $g'(z)$,  (\ref{74eq00}) and (\ref{90eq28}) we get the chain of equalities
$$
\begin{array}{lll}
&&\sup_{z\in[0,1]}g'(z)=g'(1)=\frac{p}{2m}\left(1-m+\ln\varepsilon\right)+1
\\[2pt]
\\
&=&\frac{p}{2m}\left(\frac{n-1}{p-1}-\frac{1}{2}-\frac{1}{2}\sqrt{\frac{5p-7}{3(p-1)}}\right)+1
\\[2pt]
\\
&=&\frac{p}{4(p-n)}\left[-2(p-n)-\sqrt{p-1}\left(\sqrt{\frac{5p-7}{3}}-\sqrt{p-1}\right)\right]+1
\\[2pt]
\\
&<&-\frac{2p(p-n)}{4(p-n)}+1<-\frac{p}{2}+1<0, \ \ \hbox{ for } p>n\geq2.
\end{array}
$$
Since $g'(z)<0$ for $z\in[0,1]$ and $\lim_{z\rightarrow0}g(z)=\infty$, $g(m)=m-1<0$, it follows that there exists a unique point $z_\ast\in(0,m)$ such that $g(z_\ast)=0$, i.e., $G'(z_\ast)>0$ for $z\in[0,z_\ast)$, $G'(z)<0$ for $z\in(z_\ast,1])$ and
\begin{equation}
\label{9eq27}
\sup_{z\in[0,1]}G(z)=\sup_{z\in(0,m]}G(z)=G(z_\ast).
\end{equation}
In order to localize better the maximum point $z_\ast$ we look for $z=\tau m$, $\tau\in(0,1]$ such that $G'(\tau m)>0$. From (\ref{90eq28}) we get the chain of equalities
\begin{equation}
\label{90eq27}
\begin{array}{lll}
&&g(\tau m)=\frac{p}{2}\left(1-\tau \right)(-\ln(\varepsilon\tau m))-1+\tau m
\\[2pt]
\\
&&=\frac{p}{2}\left(1-\tau \right)(-\ln(\tau))+\frac{p}{4}\left(1+\sqrt{\frac{5p-7}{3(p-1)}}-2\ln m\right)(1-\tau)-1+\tau m
\\[2pt]
\\
&=&\frac{p}{2}\left(1-\tau \right)(-\ln(\tau))+\frac{p}{4}\left(1+\sqrt{\frac{5p-7}{3(p-1)}}-2\ln m \right)-1
\\[2pt]
\\
&&-\tau \left[-m+\frac{p}{4}\left(1+\sqrt{\frac{5p-7}{3(p-1)}}-2\ln m \right)\right]
\\[2pt]
\\
&&=\frac{p}{2}(1-\tau)(-\ln\tau).
\end{array}
\end{equation}
Since $m=\frac{p-n}{p-1}<1$ we get from (\ref{74eq00}) that $\ln m<0$ and
$$
\begin{array}{lll}
&&p\left(1+\sqrt{\frac{5p-7}{3(p-1)}}-2\ln m\right)
\\[2pt]
\\
&&=p\left(1+\sqrt{1+\frac{2(p-2)}{3(p-1)}}-2\ln m\right)>2p>4>4m,
\\[2pt]
\\
&&1-4(1-m)\left[p\left(1+\sqrt{\frac{5p-7}{3(p-1)}}-2\ln m\right)-4m\right]^{-1}
\\[2pt]
\\
&&>1-4(1-m)(4-4m)^{-1}=0.
\end{array}
$$
Hence
\begin{equation}
\label{9eq280}
0<\tau=1-4(1-m)\left[p\left(1+\sqrt{\frac{5p-7}{3(p-1)}}-2\ln m\right)-4m\right]^{-1}<1,
\end{equation}
and from (\ref{90eq27}) the inequality
\begin{equation}
\label{9eq29}
G'(\tau m)>0,
\end{equation}
is satisfied. Thus $z_\ast\in (\tau m,m)$, where $\tau$ is given in (\ref{9eq28}).

From (\ref{9eq25}), (\ref{9eq27}), (\ref{9eq280}) and  (\ref{9eq29}) it follows that $-\ln (\varepsilon z)\leq-\ln(\varepsilon\tau m)$ for every $z\in [\tau m,m]$ and
$$
\begin{array}{lll}
&&\sup_{z\in[\tau m,m]}\left[(-\ln(\varepsilon z))z^{\frac{p}{2}}(1-z)^{\frac{n-m}{2m}}\right]
\\
&&\leq(-\ln(\varepsilon\tau m))\sup_{z\in[\tau m,m]}z^{\frac{p}{2}}(1-z)^{\frac{n-m}{2m}}
\\[2pt]
\\
&&=(-\ln(\varepsilon\tau m))\sup_{z\in[\tau m,m]}h^{\frac{1}{2}}(z)=(-\ln(\varepsilon\tau m))h^{\frac{1}{2}}(m)
\\[2pt]
\\
&&=(-\ln\varepsilon \tau m)m^{\frac{p}{2}}(1-m)^{\frac{n-m}{2m}}
\\[2pt]
\\
&&=\left[\frac{1}{2}\left(1+\sqrt{\frac{5p-7}{3(p-1)}}\right)-\ln\tau-\ln m\right] I^{-\frac{1}{2}}_1
\end{array}
$$
from the considerations for $I_1$. Thus we have the following estimate for $I_2$
$$
I_2\geq\frac{1}{2} I_1\left(-\ln\varepsilon \tau m\right)^{-2}=\frac{1}{2}\left[1+\sqrt{\frac{5p-7}{3(p-1)}}-2\ln m -2\ln \tau\right]^{-2}
$$
and hence from (\ref{90eq24}) we obtain (\ref{9eq30}).
\end{proof}

\subsubsection*{Estimates by means of one-parametric family of Hardy inequalities}
\label{11sec4}
We will obtain a new analytical estimate for $\lambda_{p,n}(B_R)$ from below using one-parametric family of Hardy inequalities developed in section \ref{310sec4}.

For this purpose we introduce the notations:
\begin{equation}
\label{10eqs7}
\begin{array}{l}
A(p,n,\delta)=(p-1)[p-\delta-p(n-\delta)];
\\[2pt]
\\
B(p,n,\delta)=(p-1)(n-\delta)(p+\delta)-(\delta-1)(p-\delta);
\\[2pt]
\\
C(p,n,\delta)=-\delta(n-\delta)(p-1), D=B^2-4AC;
\\[2pt]
\\
A_0(p,n)=-p^2(n-p),  B_0(p,n)=p(p-1)(n-p-1);
\\[2pt]
\\
C_0(p,n)=p(p-1), D_0=B^2_0-4A_0C_0.
\end{array}
\end{equation}
Consider the quadratic equations
\begin{equation}
\label{10eqs90}
Az^2+Bz+C=0  \hbox{ and } Cy^2+By+A=0,
\end{equation}
\begin{equation}
\label{10eqs21}
A_0z^2+B_0z+C_0=0,
\end{equation}
and note  that their discriminants $D$ and  $D_0$ are correspondingly
\begin{equation}
\label{10eqs70}
\begin{array}{lll}
D&=&B^2-4AC=[(p-1)(n-\delta)(p+\delta)-(\delta-1)(p-\delta)]^2
\\[2pt]
\\
&+&4(p-1)[p-\delta-p(n-\delta)]\delta(n-\delta)(p-1)
\\[2pt]
\\
&=&(p-\delta)^2\left\{[(p-1)(n-\delta)+1-\delta]^2+4\delta(p-1)(n-\delta)\right\}
\\[2pt]
\\
&=&(p-\delta)^2D_1>0,
\end{array}
\end{equation}
for $p>1, \delta\in(0,n)$, $n\geq2$ and $\delta\neq p$,
$$
D_0=B^2_0-4A_0C_0=p^2(p-1)[(p-1)(n-p-1)^2+4p(n-p)]>0.
$$
Let us define some of the roots of (\ref{10eqs90}), (\ref{10eqs21})
\begin{equation}
\label{10eqs71}\begin{array}{lll}
&&z_+(p,n,\delta)=\frac{-B(p,n,\delta)+\sqrt{D(p,n,\delta)}}{2A(p,n,\delta)},
\\[2pt]
\\
&&y_+(p,n,\delta)=\frac{-B(p,n,\delta)+\sqrt{D(p,n,\delta)}}{2C(p,n,\delta)},
\\[2pt]
\\
&&z_-^0(p,n)=\frac{-B_0(p,n)-\sqrt{D_0(p,n)}}{2A_0(p,n)}.
\end{array}
\end{equation}

In this section our main result is:
\begin{theorem}
\label{10th1}
For $n\geq2, p>1$, if $\Sigma=\left\{\delta\in(0,n),  \delta\neq p\frac{n-1}{p-1}\right\}$ then the estimate
\begin{equation}
\label{10eqs161}
\lambda_{p,n}(B_R)\geq\Lambda^{(3)}_{p,n}(B_R)=\frac{1}{R^p}\sup_{\delta\in\Sigma}H(p,n,\delta),
\end{equation}
holds where
$$
\begin{array}{l}
H(p,n,\delta)
\\
=\left\{\begin{array}{l} \left[\frac{p-\delta}{p(1-z_+(p,n,\delta))z_+(p,n,\delta)^{\frac{\delta}{p-\delta}}}\right]^{p-1}
\left[\frac{(p-\delta)z_+(p,n,\delta)}{p(1-z_+(p,n,\delta))}+n-\delta\right],\ \ \ 0<\delta<p,
\\[2pt]
\\
\left[\frac{p-1}{pe^{-z_-^0(p,n)}(z_-^0(p,n))}\right]^{p-1}
\left[\frac{(p-1)}{pe^{-z_-^0(p,n)}(z_-^0(p,n))}+\frac{n-p}{ e^{-z_-^0(p,n)}}\right],\ \ \delta=p,
\\[2pt]
\\
\left[\frac{\delta-p}{p(1-y_+(p,n,\delta))y_+(p,n,\delta)^{\frac{p}{\delta-p}}}\right]^{p-1}
\left[\frac{(\delta-p)}{p(1-y_+(p,n,\delta))}+n-\delta\right],\ \ p<\delta.
\end{array}\right.
\end{array}
$$
\end{theorem}
Let us consider some special cases for $\delta$:
\begin{itemize}
\item $\quad$ Suppose $\delta\rightarrow0$, so that from  (\ref{10eqs161}) and $A(p,n,0)=-p(p-1)(n-1)$, $B(p,n,0)=p[(p-1)n+1]$, $C(p,n,0)=0$, $D(p,n,0)=p^2[(p-1)n+1]^2$, we obtain $z_+(p,n,0)=0$. Applying the L'Hospital rule we get

    $$
    \lim_{\delta\rightarrow0}H(p,n,\delta)=n\lim_{\delta\rightarrow0}z_+(p,n,\delta)^{-\frac{\delta(p-1)}{p-\delta}}=n\lim_{\delta\rightarrow0}e^{-\frac{\delta(p-1)}{p-\delta}\ln z_+(p,n,\delta)}=n,
    $$
    so that
   $$
    \Lambda^{(3)}_{p,n}(B_R)\geq\Lambda^{(3,1)}_{p,n}(B_R)=\frac{n}{R^p}.
   $$
\item $\quad$ We let $\delta\rightarrow n$ in (\ref{10eqs161}) and from

$A(p,n,n)=(p-1)(p-n)$, $B(p,n,n)=(n-p)(n-1)$, $C(p,n,n)=0$, $D(p,n,n)=(p-n)^2(n-1)^2$, $A_0(n,n)=0$, $B_0(n,n)=-n(n-1)$, $C_0(n,n)=n(n-1)$, $D_0(n,n)=n^2(n-1)^2$ we get
$$
\lim_{\delta\rightarrow n}z_+(p,n,\delta)=\frac{n-1}{p-1}, \ \ 1-\lim_{\delta\rightarrow n}z_+(p,n,\delta)=\frac{p-n}{p-1}, \ \ \hbox{ for } n<p,
$$
$$
\lim_{\delta\rightarrow n}y_+(p,n,\delta)=\frac{p-1}{n-1}, \ \ 1-\lim_{\delta\rightarrow n}y_+(p,n,\delta)=\frac{n-p}{n-1}, \ \ \hbox{ for } p<n.
$$
$$
\lim_{p\rightarrow n}z_-^0(p,n)=1, \ \ \hbox{ for } p=n.
$$

So
\begin{equation}
\label{10eqs271}\begin{array}{lll}
&&\Lambda^{(3)}_{p,n}(B_R)\geq\Lambda^{(3,0)}_{p,n}(B_R)
\\[2pt]
\\
&&=\frac{1}{R^p}\lim_{\delta\rightarrow n}H(p,n,\delta)=\left(\displaystyle\frac{1}{pR}\right)^{p}
\left[\frac{(n-1)^{n-1}}{(p-1)^{p-1}}\right]^{\frac{p}{n-p}}, \ \ \hbox{ for } p\neq n.
\end{array}
\end{equation}
and
\begin{equation}
\label{10eqs2710}
\Lambda^{(3)}_{n,n}(B_R)\geq\Lambda^{(3,0)}_{n,n}(B_R)=\frac{1}{R^n}\lim_{\delta\rightarrow n}H(\delta,n,\delta)=\left(\frac{n-1}{nR}\right)^ne^n, \ \ \hbox{ for } p= n.
\end{equation}

Estimates (\ref{10eqs271}), (\ref{10eqs2710}) coincide with the result in Theorem \ref{7th5}, estimates (\ref{74eq7}) and (\ref{74eq8}).
\end{itemize}

\begin{proof}[Proof of Theorem \ref{10th1}]

The proof   follows by means of the  estimate from below of the kernels of the integrals in the right-hand side of (\ref{10eq15}) and (\ref{10eq15-1}).

We will consider cases $\delta\neq p$ and $\delta= p$ separately.

(i) $\quad$ For the case $\delta\neq p$, $\delta\in(0,n)$ we have the following estimate from Hardy inequality (\ref{10eq15})
\begin{equation}
\label{10eqs1}
\int_{B_R}|\nabla u|^pdx\geq\int_{B_R}g(|x|)|u|^pdx,
\end{equation}
where
$$
\begin{array}{lll}
g(|x|)&=&\left|\frac{p-\delta}{p}\right|^p\frac{1}{|x|^{\frac{(\delta-1)p}{p-1}}\left|R^{\frac{p-\delta}{p-1}}-|x|^{\frac{p-\delta}{p-1}}\right|^p}
\\[2pt]
\\
&+&(n-\delta)\left|\frac{p-\delta}{p}\right|^{p-1}\frac{1}{|x|^{\delta}
\left|R^{\frac{p-\delta}{p-1}}-|x|^{\frac{p-\delta}{p-1}}\right|^{p-1}}.
\end{array}
$$
After the change of the variable $|x|=Rr$, $r\in(0,1)$ we get
$$
g(|x|)=g(Rr)=R^{-p}G(r),
$$
where
$$
G(r)=\left|\frac{p-\delta}{p}\right|^{p-1}\left[\left|\frac{p-\delta}{p}\right|r^{-\frac{(\delta-1)p}{p-1}}\left|1-r^{\frac{p-\delta}{p-1}}\right|^{-p}
+(n-\delta)r^{-\delta}\left|1-r^{\frac{p-\delta}{p-1}}\right|^{1-p}\right].
$$
Since
\begin{itemize}
\item $\quad$  for $0<\delta<p$
\begin{equation}
\label{10eqs4}
\begin{array}{lll}
G(r)&=&g_1(r)=\left(\frac{p-\delta}{p}\right)^{p-1}\left[\frac{p-\delta}{p}r^{\frac{(-\delta+1)p}{p-1}}\left(1-r^{\frac{p-\delta}{p-1}}\right)^{-p}\right.
\\[2pt]
\\
&+&\left.(n-\delta)r^{-\delta}\left(1-r^{\frac{p-\delta}{p-1}}\right)^{1-p}\right].
\end{array}
\end{equation}
\item  $\quad$ for $1<p<\delta<n$
$$
\begin{array}{lll}
G(r)&=&g_2(r)=\left(\frac{\delta-p}{p}\right)^{p-1}\left[\frac{\delta-p}{p}r^{-p}\left(1-r^{\frac{\delta-p}{p-1}}\right)^{-p}\right.
\\[2pt]
\\
&+&\left.(n-\delta)r^{-p}\left(1-r^{\frac{\delta-p}{p-1}}\right)^{1-p}\right].
\end{array}
$$
\end{itemize}
It is clear that $\lim_{r\rightarrow0}G(r)=\lim_{r\rightarrow1}G(r)=\infty$  and the positive function $G(r)$ has a positive minimum in $(0,1)$. Our aim is to find the critical point of $G(r)$ for $r\in(0,1)$.

Let us simplify  the derivative of $G(r)$ and
\begin{itemize}
\item $\quad$ for $0<\delta<p$ we denote $r^{\frac{p-\delta}{p-1}}=z, z\in (0,1)$  we get
$$
\begin{array}{lll}
&&\frac{\partial G(r)}{\partial r}= \frac{\partial g_1(r)}{\partial r}
\\[2pt]
\\
&=&\left(\frac{p-\delta}{p}\right)^{p-1}\left[-\frac{p-\delta}{p-1}(\delta-1)r^{\frac{1-\delta p}{p-1}}\left(1-r^{\frac{p-\delta}{p-1}}\right)^{-p}\right.
\\[2pt]
\\
&+&\frac{(p-\delta)^2}{p-1}r^{\frac{(1-\delta)( p+1)}{p-1}}\left(1-r^{\frac{p-\delta}{p-1}}\right)^{-p-1}-\delta(n-\delta)r^{-\delta-1}\left(1-r^{\frac{p-\delta}{p-1}}\right)^{1-p}
\\[2pt]
\\
&+&\left.(p-\delta)(n-\delta)r^{-\delta-1}r^{\frac{p-\delta}{p-1}}\left(1-r^{\frac{p-\delta}{p-1}}\right)^{-p}\right]
\\[2pt]
\\
&=&\left(\frac{p-\delta}{p}\right)^{p-1}r^{-\delta-1}\left(1-r^{\frac{p-\delta}{p-1}}\right)^{-p-1}\left[-\frac{p-\delta}{p-1}(\delta-1)z(1-z)\right.
\\[2pt]
\\
&+&\left.\frac{(p-\delta)^2}{p-1}z^2-\delta(n-\delta)(1-z)^2+(p-\delta)(n-\delta)z(1-z)\right]
\\[2pt]
\\
&=&\left(\frac{p-\delta}{p}\right)^{p-1}\frac{1}{p-1}z^{-\frac{(\delta+1)(p-1)}{p-\delta}}(1-z)^{-1-p}(Az^2+Bz+C).
\end{array}
$$
\item $\quad$ for $1<p<\delta<n $ we denote $r^{\frac{\delta-p}{p-1}}=y, y\in (0,1)$ and we get
$$
\begin{array}{lll}
&&\frac{\partial G(r)}{\partial r}= \frac{\partial g_2(r)}{\partial r}
\\[2pt]
\\
&=&\left(\frac{\delta-p}{p}\right)^{p-1}\left[-(\delta-p)r^{-p-1}\left(1-r^{\frac{\delta-p}{p-1}}\right)^{-p}
+\frac{(\delta-p)^2}{p-1}r^{-p-1+\frac{\delta-p}{p-1}}\left(1-r^\frac{\delta-p}{p-1}\right)^{-p-1}\right.
\\[2pt]
\\
&-&\left.p(n-\delta)r^{-p-1}\left(1-r^\frac{\delta-p}{p-1}\right)^{1-p}+(\delta-p)(n-\delta)r^{-p-1+\frac{\delta-p}{p-1}}\left(1-r^{\frac{\delta-p}{p-1}}\right)^{-p}\right]
\\[2pt]
\\
&=&\frac{1}{p-1}\left(\frac{\delta-p}{p}\right)^{p-1}r^{-p-1}\left(1-r^{\frac{\delta-p}{p-1}}\right)^{-p-1}\left[-(p-1)(\delta-p)(1-z)+(\delta-p)^2z\right.
\\[2pt]
\\
&-&\left.p(p-1)(n-\delta)(1-z)^2+(p-1)(\delta-p)(n-\delta)z(1-z)\right]
\\[2pt]
\\
&=&\left(\frac{\delta-p}{p}\right)^{p-1}\frac{1}{p-1}y^{-\frac{(p+1)(p-1)}{p-\delta}}(1-y)^{-p-1}(Cy^2+By+A),
\end{array}
$$
\end{itemize}
where $A$, $B$, $C$ are defined in (\ref{10eqs7}).

Suppose that $A\neq0$, i.e., $\delta\neq p\frac{n-1}{p-1}$, then in order to find the critical point of $G(r)$, we have to solve the quadratic equations in (\ref{10eqs90}).

The discriminant of the equations in (\ref{10eqs90}) for $\frac{\partial G(r)}{\partial r}=0$  is $D$ given in (\ref{10eqs70}).

Since $D>0$ and $A\neq0$ the equation
\begin{equation}
\label{10eqs11}
P_1(z)=Az^2+Bz+C=0
\end{equation}
has two real roots
$$
z_{\pm}=\frac{-B\pm|p-\delta|\sqrt{D_1}}{2A}=\frac{2C}{-B\mp|p-\delta|\sqrt{D_1}}, \ \ \hbox{for } \ \ 0<\delta<p.
$$
Analogously, from $C>0$ and $D>0$ the equation
$$
P_2(y)=Cy^2+By+A=0
$$
has two real roots
$$
y_{\pm}=\frac{-B\pm|p-\delta|\sqrt{D_1}}{2C}=\frac{2A}{-B\mp|p-\delta|\sqrt{D_1}}, \ \ \hbox{for } \ \ 1<p<\delta<n.
$$

Later on we will use only the roots of $P_1(z)=0$ and $P_2(y)=0$ which are in the interval $(0,1)$. In order to find which roots satisfy this condition we prove the following proposition.
\begin{proposition}
\label{newprop1}
Let $n\geq2$, $p>1$, $\delta\in(0,n)$, $\delta\neq p$, then the following statements hold
\begin{itemize}
\item[i)] If $p>n$, then
\begin{itemize}
\item[i1)]$\quad$ $A(p,n,\delta)>0$ if and only if  $\frac{p(n-1)}{p-1}<\delta<n$  so that $z_-<0<z_+$ and
$$
\inf_{r\in(0,1)}G(r)=g_1\left(z_+^{\frac{p-1}{p-\delta}}\right);
$$
\item[i2)]$\quad$ $A(p,n,\delta)<0$ if and only if $0<\delta<\frac{p(n-1)}{p-1}$  so that $0<z_+<z_-$ and
$$
\inf_{r\in(0,1)}G(r)=g_1\left(z_+^{\frac{p-1}{p-\delta}}\right);
$$
\item[i3)]$\quad$ $A(p,n,\delta)=0$ if and only if $0<\delta=\frac{p(n-1)}{p-1}$  so that (\ref{10eqs11}) has an unique positive root
$$
z_+\left(p,n,\frac{p(n-1)}{p-1}\right)=-\frac{C(p,n,\frac{p(n-1)}{p-1})}{B(p,n,\frac{p(n-1)}{p-1})}
$$
and
$$
\inf_{r\in(0,1)}G(r)=g_1\left(z_+^{\frac{(p-1)^2}{p(p-n)}}\left(p,n,\frac{p(n-1)}{p-1}\right)\right);
$$
\end{itemize}
\item[ii)]$\quad$ If $1<p<n$ then $A(p,n,\delta)<0$ for $\delta\in(0,n)$ and
\begin{itemize}
\item[ii1)]$\quad$ for $0<\delta<p$ we have $0<z_+<z_-$ and
$$
\inf_{r\in(0,1)}G(r)=g_1\left(z_+^{\frac{p-1}{p-\delta}}\right);
$$
\item[ii1)]$\quad$ for $1<p<\delta<n$ we have $0<y_+<y_-$ and
$$
\inf_{r\in(0,1)}G(r)=g_2\left(y_+^{\frac{p-1}{\delta-p}}\right);
$$
\end{itemize}
\end{itemize}
\end{proposition}
\begin{proof}
\begin{itemize}
\item[i)] Since for $n>p$ we have $n>\frac{p(n-1)}{p-1}$ and for $1<p<n$ the inequality $\frac{p(n-1)}{p-1}>n$ holds, the statements for the sign of $A(p,n,\delta)$ follow immediately after (\ref{10eqs7})
\begin{itemize}
\item[i1)]$\quad$ From $A>0,C<0$ it follows that $B^2-4AC>B^2$ and $z_-=\frac{-B-\sqrt{D}}{2A}<\frac{-B-|B|}{2A}\leq0$, $z_+\geq\frac{-B+|B|}{2A}=0$. Thus the minimum of $G(r)$ in the interval $(0,1)$ is attained at the point $z_+^{\frac{p-1}{p-\delta}}$;
\item[i2)]$\quad$ From $P_1(0)=C<0$, $\lim_{z\rightarrow\infty} P_1(z)=-\infty$ we get that $z_+>0, z_->0$ and $z_-=-\frac{B+\sqrt{D}}{2A}>-\frac{B-\sqrt{D}}{2A}=z_+$. Since $P_1(z)<0$ for $z\in(0,z_+)$ and $P_1(z)>0$ for $z\in(z_+, z_-)$, it follows that $G(r)$ has a minimum at the point $z_+^{\frac{p-1}{p-\delta}}$;
\item[i3)]$\quad$ The proof is trivial.
\end{itemize}
\item[ii)]
\begin{itemize}
\item[ii1) ]$\quad$ The proof is identical with the proof of i2);
\item[ii2) ]$\quad$ Since $P_2(0)=A<0$, $\lim_{z\rightarrow\infty} P_1(z)=-\infty$ it follows that $y_+>0, y_->0$ and $y_-=-\frac{B+\sqrt{D}}{2C}>-\frac{B-\sqrt{D}}{2C}=y_+$. From the sign of $P_2(z)$ we get that $g_2(r)$ attains its minimum at the point $y_+^{\frac{p-1}{\delta-p}}$.
\end{itemize}
\end{itemize}
\end{proof}

(ii)$\quad$ For the case $\delta=p<n$,  from (\ref{10eq15-1}) we get
$$
\int_{B_R}|\nabla u|^Pdx\geq\int_{B_R}g(|x|)|u|^pdx,
$$
where
$$
g(|x|)=\left(\frac{p-1}{p}\right)^p|x|^{-p}\left|\ln\frac{R}{|x|}\right|^{-p}+\left(\frac{p-1}{p}\right)^{p-1}(n-p)|x|^{-p}\left|\ln\frac{R}{|x|}\right|^{1-p}.
$$
For $|x|=Rr, r\in(0,1)$ we obtain
$$
g(|x|)=g(Rr)=R^{-p}G(r),
$$
where
$$
G(r)=\left(\frac{p-1}{p}\right)^{p-1}\left\{\frac{p-1}{p}r^{-p}(-\ln r)^{-p}+(n-p)r^{-p}(-\ln r)^{1-p}\right\}.
$$
Tedious calculations give us
$$
\frac{\partial G}{\partial r}=\left(\frac{p-1}{p}\right)^{p-1}\frac{1}{p}z^{-p-1}e^{-(p+1)z}[A_0z^2+B_0z+C_0], \ \ \hbox{ for } z=-\ln r>0.
$$

Critical points of $\frac{\partial G}{\partial r}$ are the solutions of the quadratic equation (\ref{10eqs21}).

Since $D_0(p,n)>0$ and $A_0(p,n)\neq 0$ equation  (\ref{10eqs21}) has two real roots
$$
z^0_{\pm}=\frac{-B_0\pm\sqrt{D_0}}{2A_0}, \ \ z^0_+<z^0_-.
$$
From $P_0(0)=p(p-1)>0$ and $\lim_{z\rightarrow\infty}P_0(z)=-\infty$ it follows that $z^0_+<0<z^0_-$.

Thus $G(r)$ attains its minimum in $(0,1)$ at the point $r_0=e^{-z^0_-}$, where from (\ref{10eqs71})
$$
z^0_-=2C_0\left[\sqrt{D_0}-B_0\right]^{-1}.
$$

Theorem \ref{10th1} is proved.
\end{proof}
\subsection{Comparison between different analytical estimates of $\lambda_{p,n}(B_R)$ }
\label{10sec4}
In this section we will compare analytical estimates from below of $\lambda_{p,n}(B_R)$ defined in Theorem \ref{10th1} $\Lambda^{(3)}_{p,n}(B_R)$ with  $\Lambda^{(1)}_{p,n}(B_R)$, defined in (\ref{2eq200})   with $\Lambda^{(2)}_{p,n}(B_R)$, defined in (\ref{2eq201}) and with $\Lambda^{(H)}_{p,n}(B_R)$,  $\Lambda^{(L)}_{p,n}(B_R)$, defined in (\ref{38eq2}) and (\ref{2eq23}). We compare only those estimates of $\lambda_{p,n}(B_R)$ that are given with analytical formulas for every $p>1$, $n\geq2$. The estimate in (\ref{9eq30}), Sect. \ref{11sec3} is valid only for $p>n\geq2$ and that is why we will not use it for the comparison, no matter, it is clear that the right-hand side of (\ref{9eq30}) is greater than $\Lambda^{(3,0)}_{p,n}(B_R)$.
\subsubsection*{Comparison of $\Lambda^{(H)}_{p,n}(B_R)$ and  $\Lambda^{(L)}_{p,n}(B_R)$ with $\Lambda^{(1)}_{p,n}(B_R)$, $\Lambda^{(2)}_{p,n}(B_R)$ and $\Lambda^{(3)}_{p,n}(B_R)$}
Let us compare $\Lambda^{(H)}_{p,n}(B_R)=\left|\frac{n-p}{pR}\right|^p$  for $n\geq2$, $p>1$, $n\neq p$ with other lower bounds for $\lambda_{p,n}(B_R)$.

\begin{itemize}
\item For $p>n\geq2$ we get the estimate
$$
\Lambda^{(H)}_{p,n}(B_R)=\left(\frac{p-n}{pR}\right)^p=\frac{1}{R^p}\left(1-\frac{n}{p}\right)^p<\frac{1}{R^p}<\frac{n}{R^p}=\Lambda^{(3,1)}_{p,n}(B_R);
$$
\item For $1<p<n$ the inequality
$$
\Lambda^{(H)}_{p,n}(B_R)=\left(\frac{n-p}{pR}\right)^p<\left(\frac{n}{pR}\right)^p=\Lambda^{(1)}_{p,n}(B_R)
$$
holds.
\item As for the Lindqvist's  constant $\Lambda^{(L)}_{p,n}(B_R)=\frac{p}{R^p}$ for $p>n\geq2$  given in (\ref{2eq23}) we get the estimate
$$
\Lambda^{(L)}_{p,n}(B_R)=\frac{p}{R^p}<\frac{np}{R^p}=\Lambda^{(2,2)}_{p,n}(B_R).
$$
\end{itemize}
\subsubsection*{Comparison of $\Lambda^{(3)}_{p,n}(B_R)$ with  $\Lambda^{(1)}_{p,n}(B_R)$}
\label{10sec4-3}
We will use   the estimate (\ref{10eqs271}), which coincides with the estimate (\ref{74eq7}).
\begin{proposition}
\label{10prop4-3} For every $n\geq2$ there exists $p_{0n}$, $1<p_{0n}<2$
such that
$$
\Lambda^{(1)}_{p,n}<\Lambda^{(3,0)}_{p,n}\leq\Lambda^{(3)}_{p,n}, \hbox{ for } p_{0n}<p.
$$
\end{proposition}
\begin{proof}
We define the function
$$
f_n(p)=\displaystyle\frac{1}{n-p}\left[(n-1)\ln(n-1)-(p-1)\ln(p-1)\right]-\ln
n.
$$
The inequality $\Lambda^{(3,0)}_{p,n}(B_R)>\Lambda^{(1)}_{p,n}(B_R)$
holds if and only if $f_n(p)>0$. We will show that for every fixed
$n\geq2$ the function $f_n(p)$ is astrictly increasing one for $p>1$
and $\lim_{p\rightarrow1}f_n(p)<0$, $f_n(2)>0$ for $n\geq2$, $\lim_{n\rightarrow2}f_n(2)=1-\ln 2>0$. Thus, there exists $p_{0n}\in(1,2)$ such
that
$f_n(p)<0$ for $1<p<p_{0n}$,
\begin{equation}
\label{10eq25-11}
f_n(p_{0n})=(n-1)\ln(n-1)-(p_{0n}-1)\ln(p_{0n}-1)-(n-p_{0n})\ln n=0,
\end{equation}
and $f_n(p)>0$ for $p_{0n}<p$.

For the first derivative of $f_n(p)$ we have
$$
\begin{array}{lll}
f'_n(p)&=&\displaystyle\frac{1}{(n-p)^2}\left[(n-1)\ln(n-1)-(n-1)+(p-1)-(n-1)
\ln(p-1)\right] \\
&=&\displaystyle\frac{g_n(p)}{(n-p)^2}.
\end{array}
$$
Since $g'_n(p)=\displaystyle\frac{p-n}{p-1}$,
$g''_n(p)=\displaystyle\frac{n-1}{(p-1)^2}>0$ then $g_n(p)$ has a
minimum at the point $p=n$ and $g_n(n)=0$. Using L'Hospital rule we
obtain  $\lim_{p\rightarrow
n}f'_n(p)=\displaystyle\frac{1}{2(n-1)}>0$ and hence $f'_n(p)>0$ for
every $p>1$. Moreover, $\lim_{p\rightarrow1}f_n(p)=\ln(n-1)-\ln n<0$, and
\begin{equation}
\label{100eq25}
f_n(2)=\displaystyle\frac{1}{n-2}\left[(n-1)\ln(n-1)-(n-2)\ln
n\right]>0.
\end{equation}
The  inequality (\ref{100eq25}) holds because for the function
$z(n)=(n-1)\ln(n-1)-(n-2)\ln n$ we have
$z'=\displaystyle\frac{2}{n}+\ln(n-1)-\ln n$ ,
$z''=\displaystyle\frac{2-n}{n^2(n-1)}\leq0$, i.e., $z'$ is
a decreasing function, $z'(n)> \lim_{n\rightarrow\infty}z'(n)>0$. Hence $z(n)$ is
a strictly increasing function and $z(n)>z(2)=0$.
\end{proof}

\subsubsection*{Comparison of $\Lambda^{(3)}_{p,n}(B_R)$ with $\Lambda^{(2)}_{p,n}(B_R)$}
\label{10sec4-4}
\begin{proposition}
\label{10prop4-4}
For integer  $n\geq2$ and $p\geq p_n=\frac{27}{8}\left(\frac{2n}{2n-3}\right)^2$ the estimate
\begin{equation}
\label{10eq26-1}
\Lambda^{(2)}_{p,n}(B_R)<\Lambda^{(3)}_{p,n}(B_R),
\end{equation}
holds.
\end{proposition}
\begin{proof}
From (\ref{10eqs1}) and (\ref{10eqs4}) for $\delta<p$, $\delta\in(0,n)$ it follows that
$$
\Lambda^{(3)}_{p,n}(B_R)=R^{-p}\sup_{\delta\in(0,n)}\inf_{r\in(0,1)}g_1(r)\geq R^{-p}\sup_{\delta\in(0,n)}\inf_{r\in(0,1)}H_2(r),
$$
where
$$
H_2(R)=(n-\delta)\left(\frac{p-\delta}{p}\right)^{p-1}r^{-\delta}\left(1-r^{\frac{p-\delta}{p-1}}\right)^{1-p}.
$$
The positive function
$$
h_2(r)=\frac{1}{p-\delta}r^{\frac{\delta}{p-1}}\left(1-r^{\frac{p-\delta}{p-1}}\right)
$$
attains its maximum for $0\leq r\leq 1$ at the point $r_2=\left(\frac{\delta}{p}\right)^{\frac{p-1}{p-\delta}}<1$, because
$$
\begin{array}{l}
h'_2(r)=\frac{r^{\frac{\delta-p+1}{p-1}}}{(p-1)(p-\delta)}\left[\delta -pr^{\frac{p-\delta}{p-1}}\right],
\\[2pt]
\\
 h'_2(r_2)=0, \ \ h_2(0)=h_2(1)=0.
\end{array}
$$
Thus from the equalities
$$
\begin{array}{lll}
\inf_{r\in(0,1)}H_2(r)&=&(n-\delta)\inf_{r\in(0,1)}\left(ph_2(r)\right)^{1-p}=(n-\delta)\sup_{r\in(0,1)}\left(ph_2(r)\right)^{p-1}
\\[2pt]
\\
&=&(n-\delta)\left(ph_2(r_2)\right)^{p-1}=(n-\delta)\left(\frac{p}{\delta}\right)^{\frac{\delta(p-1)}{p-\delta}}
\end{array}
$$
we get the estimate
\begin{equation}
\label{10eq611}
\Lambda^{(3)}_{p,n}(B_R)\geq\Lambda^{(3,2)}_{p,n}(B_R)=R^{-p}\sup_{\delta\in(0,n)}(n-\delta)\left(\frac{p}{\delta}\right)^{\frac{\delta(p-1)}{p-\delta}}.
\end{equation}
Thus from (\ref{10eq611}), the estimate (\ref{10eq26-1}) holds if
\begin{equation}
\label{10eq67}
\Lambda^{(2)}_{p,n}(B_R)<\Lambda^{(3,2)}_{p,n}(B_R), \ \ n\geq2, p\geq p_n.
\end{equation}

Note that $p_n$ is a decreasing function for $n\in[2,\infty)$, so $3.375<p_n<54$.

Since $p_n>3$ then  for $n\geq2$, $\delta<p$  the estimate (\ref{10eq67}) is equivalent to the inequality
$$
\sup_{\delta\in(0,n), \delta<p}(n-\delta)\left(\frac{p}{\delta}\right)^{\frac{\delta(p-1)}{p-\delta}}>np
$$
for $p\geq p_n$, $p>\delta$.

For $\delta=\frac{3}{2}$ and $p\geq p_{n}$ a simple computation gives us
$$
\begin{array}{lll}
&&\sup_{\delta\in(0,n), \delta<p}(n-\delta)\left(\frac{p}{\delta}\right)^{\frac{\delta(p-1)}{p-\delta}}>\left(n-\frac{3}{2}\right)\left(\frac{2p}{3}\right)^{\frac{3(p-1)}{2p-3}}
\\[2pt]
\\
&&=\frac{1}{2}(2n-3)\left(\frac{2p}{3}\right)^{\frac{3}{2}}\left(\frac{2p}{3}\right)^{\frac{3}{2(2p-3)}}\geq\frac{1}{3}(2n-3)\left(\frac{2p}{3}\right)^{\frac{1}{2}}p
\\[2pt]
\\
&&\geq\frac{1}{3}(2n-3)\left(\frac{2p_{0,n}}{3}\right)^{\frac{1}{2}}p=\frac{1}{3}(2n-3)\frac{3}{2}\frac{2n}{2n-3}p=np.
\end{array}
$$
\end{proof}

\subsubsection*{Comparison of $\Lambda^{(2,1)}_{p,n}(B_R)$ with $\Lambda^{(2,2)}_{p,n}(B_R)$}
\label{10app1}
\begin{proposition}
\label{10prop4-1}
For every $n\geq2$ the estimates
\begin{equation}
\label{10eq24-1}
\Lambda^{(2,1)}_{p,n}(B_R)>\Lambda^{(2,2)}_{p,n}(B_R) \ \ \hbox{ for } p\in(1,2),
\end{equation}
\begin{equation}
\label{10eq24-2}
\Lambda^{(2,1)}_{p,n}(B_R)<\Lambda^{(2,2)}_{p,n}(B_R) \ \ \hbox{ for } p>2,
\end{equation}
hold.
\end{proposition}
\begin{proof}
The inequality (\ref{10eq24-1}) is equivalent to
$h(p)=(p-1)\ln\frac{p}{p-1}-\ln p>0$ for $p\in(1,2)$, while (\ref{10eq24-2}) holds when $h(p)<0$ for $p>2$.

A simple computation gives us
\begin{equation}
\label{10eq24-3}
h'(p)=\ln\frac{p}{p-1}-\frac{2}{p}, \ \ h''(p)=\frac{p-2}{p^2(p-1)}, \ \ h'(2)=\ln\frac{2}{e}<0,
\end{equation}
and from the L'Hospital rule
$$
\lim_{p\rightarrow\infty}h'(p)
=
\lim_{p\rightarrow\infty}\frac{p\ln\frac{p}{p-1}-2}{p}=\lim_{p\rightarrow\infty}\left(\ln\frac{p}{p-1}-\frac{1}{p-1}\right)=0.
$$
Thus we have from (\ref{10eq24-3}) that $h'(p)<0$ for $p>2$. From $h(2)=0$ it follows that $h(p)<0$ for $p>2$.

From $\lim_{p\rightarrow1}h(p)=0$ and the concavity of $h(p)$ for $p\in(1,2)$ we get $h(p)>0$ for $p\in(1,2)$.
\end{proof}

\subsubsection*{Comparison of $\Lambda^{(1)}_{p,n}(B_R)$ with  $\Lambda^{(2)}_{p,n}(B_R)$}
\label{10app2}
According to Proposition \ref{10prop4-1} we will compare $\Lambda^{(1)}_{p,n}(B_R)$ with $\Lambda^{(2,1)}_{p,n}(B_R)$  for $p\in (1,2)$, and $\Lambda^{(1)}_{p,n}(B_R)$ with $\Lambda^{(2,2)}_{p,n}(B_R)$ for $p\geq2$.
\begin{proposition}
\label{10prop22}

\begin{itemize}
\item If $n\in[2,8]$, then
$$
\Lambda^{(2)}_{p,n}(B_R)>\Lambda^{(1)}_{p,n}(B_R)\ \ \hbox{ for } p>1, p\neq 2 and \Lambda^{(2)}_{2,8}(B_R)>\Lambda^{(1)}_{2,8}(B_R).
$$
\item If $n\geq9$, then there exist  constants  $p_{1,n}\in(1,2)$ and $p_{3,n}\geq2$   such that
$$
\Lambda^{(2)}_{p,n}(B_R)>\Lambda^{(1)}_{p,n}(B_R) \ \ \hbox{ for } p\in(1,p_{1,n})\cup(p_{3,n},\infty),
$$
\begin{equation}
\label{10eq35-3}
\Lambda^{(2)}_{p,n}(B_R)<\Lambda^{(1)}_{p,n}(B_R) \ \ \hbox{ for } p\in (p_{1,n},p_{3,n}).
\end{equation}
\end{itemize}
\end{proposition}
\begin{proof}
Case 1: $p \in(1,2)$.

For every fixed $n$, inequality $\Lambda^{(1)}_{p,n}(B_R)>\Lambda^{(2,1)}_{p,n}(B_R)$ is equivalent to
$$
h(p)=(p-1)\ln n-(2p-1)\ln p+(p-1)\ln(p-1)>0, \ \  \hbox{ for } p \in(1,2).
$$
Simple computations give us
$h'(p)=\ln n-2\ln p-1+\frac{1}{p}+\ln(p-1)$, $h''(p)=-\frac{p^2-p-1}{p^2(p-1)}$  and $h''(p)>0$ for $p\in\left(1,\frac{\sqrt{5}+1}{2}\right)$,  $h''(p)<0$ for $p\in\left((\frac{\sqrt{5}+1}{2},2\right)$. Hence $h'(p)$ has a maximum at the point $\frac{\sqrt{5}+1}{2}$ and
$$
\begin{array}{l}
h'\left(\frac{\sqrt{5}+1}{2}\right)=\ln n-2\ln\frac{\sqrt{5}+1}{2}-1+\frac{2}{\sqrt{5}+1}+\ln\frac{\sqrt{5}-1}{2}
\\
=\ln n-\left(\frac{\sqrt{5}-1}{2}\right)^2-\ln\left(\frac{\sqrt{5}+1}{2}\right)^3,
\end{array}
$$
i.e.,
$$
h'\left(\frac{\sqrt{5}+1}{2}\right)<0 \ \ \hbox{ for  } n<\left(\frac{\sqrt{5}+1}{2}\right)^3e^{\frac{\sqrt{5}-1}{2}}\approx6.2065
$$
and
$$
h'\left(\frac{\sqrt{5}+1}{2}\right)>0 \ \ \hbox{ for  } n>6.2065.
$$
Since $\lim_{p\rightarrow1}h'(p)=-\infty$, $h'(2)=\frac{1}{2}\ln\frac{n^2}{16e}<0$ for $n\in[2,6]$ and $h'(2)>0$ for $n\geq7$, it follows that
$$
h'(p)<0, \ \ \hbox{ for } n\in[2,6] \ \ \hbox{ and } p\in(1,2).
$$
For $n\geq7$ there exists $q_n\in\left(1,\frac{\sqrt{5}+1}{2}\right)$ such that
$$
\begin{array}{l}
h'(p)<0, \ \ \hbox{ for } p\in(1,q_n),
\\
h'(p)>0, \ \ \hbox{ for } p\in(q_n, 2).
\end{array}
$$
Since $\lim_{p\rightarrow1}h(p)=0$, $h(2)=\ln\frac{n}{8}$ and $h(2)<0$ for $n\in[2,7]$, $h(2)=0$ for $n=8$, $h(2)>0$ for $n>8$ it follows that $h(p)<0$ for $p\in(1,2)$, $n\in[2,8]$.

Since $h(2)=\ln\frac{n}{8}>0$ for $n\geq9$, there exists $p_{1,n}\in(1,2)$ such that
\begin{equation}
\label{10eq35-88}
h(p_{1,n})=(p_{1,n}-1)\ln n-(2p_{1,n}-1)\ln p_{1,n}+(p_{1,n}-1)\ln(p_{1,n}-1)=0,
\end{equation}
and $h(p)<0$ for $p\in(1,p_{1,n})$ and $h(p)>0$ for $p\in(p_{1,n},2)$. Thus Proposition  \ref{10prop22} for $p\in(1,2)$ is proved.

Case 2: $p\geq 2$.

The inequality  (\ref{10eq35-3}) for $p\geq2$ is equivalent for every fixed $n\geq2$ to
$$
h_1(p)=(p-1)\ln n-(p+1)\ln p>0, \hbox{ for } p\geq2.
$$
A simple computation gives us $h_1'(p)=\ln n-\ln p-\frac{p+1}{p}$, $h_1''(p)=\frac{1-p}{p^2}<0$   for $p\geq2$, $h_1(2)=\ln\frac{n}{8}<0$ for $n\in[2,7]$, $h_1(2)=0$, for $n=8$.

Since $h_1'(2)=\ln\frac{n}{2e^{3/2}}\leq\ln\frac{8}{2e^{3/2}}\approx -0.1137<0$ for $n\in[2,8]$ it follows that $h_1(p)<0$ for $p>2$ and $n\in[2,8]$.

For $n\geq 9$ we get $h_1(2)=\ln\frac{n}{8}>0$, $\lim_{p\rightarrow\infty}h_1(p)=-\infty$ and consequently there exists a constant $p_{3,n}>2$,
$$
h_1(p_{3,n})=(p_{3n}-1)\ln n-(p_{3n}+1)\ln p_{3,n}=0,
$$
such that $h_1(p)>0$ for $p\in[2,p_{3,n})$, $h_1(p)<0$ for $p>p_{3,n}$. For example, $p_{3,n}=3$ for $n=9$.
\end{proof}

Finally, we summarize the analytical results in Propositions \ref{10prop4-3}--\ref{10prop22}.

Suppose $n\geq9$, then from Proposition \ref{10prop4-3} we get $\Lambda^{(3)}_{p,n}(B_R)>\Lambda^{(1)}_{p,n}(B_R)$ for $p>p_{0n}$, where $p_{0n}\in(1,2)$  is a solution of equation
(\ref{10eq25-11}).

Analogously, from Proposition \ref{10prop22} we obtain the estimates $\Lambda^{(1)}_{p,n}(B_R)>\Lambda^{(2)}_{p,n}(B_R)$ for $p>p_{1n}$ and $\Lambda^{(1)}_{p,n}(B_R)<\Lambda^{(2)}_{p,n}(B_R)$ for $p\in(1,p_{1n})$, where $p_{1n}\in(1,2)$ is a solution of the equation (\ref{10eq35-88}). Thus for $n\geq 9$  we identify the following cases:
\begin{itemize}
\item[(i)] $\quad$ If $p_{1,n}<p_{0,n}$ then
$$
\begin{array}{lll}
\Lambda^{(2)}_{p,n}(B_R)&>&\max\left\{\Lambda^{(1)}_{p,n}(B_R), \Lambda^{(3)}_{p,n}(B_R)\right\}, \ \  \hbox{ for } p\in(1,p_{1,n}),
\\[2pt]
\\
\Lambda^{(1)}_{p,n}(B_R)&>&\Lambda^{(2)}_{p,n}(B_R), \ \  \hbox{ for } p\in(p_{1,n},p_{0,n}),
\\[2pt]
\\
\Lambda^{(3)}_{p,n}(B_R)&>&\max\left\{\Lambda^{(1)}_{p,n}(B_R), \Lambda^{(2)}_{p,n}(B_R)\right\}, \ \  \hbox{ for } p>p_{0,n}.
\end{array}
$$
\item[(ii)] $\quad$  If $p_{0,n}<p_{1,n}$ then there exists $p_{2,n}\in[p_{0,n},p_{1,n})\subset(1,2)$ such that
$$
\begin{array}{lll}
\Lambda^{(2)}_{p,n}(B_R)&>&\max\left\{\Lambda^{(1)}_{p,n}(B_R), \Lambda^{(3)}_{p,n}(B_R)\right\}, \ \  \hbox{ for } p\in(1,p_{2,n}),
\\[2pt]
\\
\Lambda^{(3)}_{p,n}(B_R)&>&\max\left\{\Lambda^{(1)}_{p,n}(B_R), \Lambda^{(2)}_{p,n}(B_R)\right\}, \ \  \hbox{ for } p>p_{2,n}.
\end{array}
$$
\end{itemize}
For $1<p<p_n$, $n\in[2,8]$ the comparison is by means of numerical calculations.
\begin{remark}\rm
\label{10rem11}
For sufficiently large values of $n$ we get $p_{1,n}>p_{0,n}$. Indeed, after the limit $n\rightarrow\infty$ in (\ref{10eq35-88}) it follows that $\lim_{n\rightarrow\infty}p_{1n}=1$. From the definitions of $p_{0,n}$ and $p_{1,n}$ we have that $p=p_{0,n}$ satisfying the equation
$$
(n-p)f_n(p)=(n-1)\ln(n-1)-(p-1)\ln(p-1)-(n-p)\ln n=0,
$$
while $y=p_{1,n}$ satisfies the equation
$$
h(p)=(y-1)\ln n-(2y-1)\ln y+(y-1)\ln(y-1)=0.
$$
Hence for $y$ we obtain
$$
(n-p)f_n(y)-h(y)=(n-1)\ln\frac{n-1}{n}+(2y+1)\ln y-2(y-1)\ln(y-1)\rightarrow_{n\rightarrow\infty}-1,
$$
because
$$
(n-1)\ln\frac{n-1}{n}=\left(\frac{n}{n-1}\right)\ln\left(\frac{n-1}{n}\right)^n\rightarrow_{n\rightarrow\infty}\ln e^{-1}=-1, \ \  \hbox{ and } \lim_{y\rightarrow1}(y-1)\ln(y-1)=0.
$$
From the inequality $f_n(y)=f_n(p_{1,n})<0$ for $n$ sufficiently large it follows that $p_{1,n}<p_{0,n}$ for $n\gg 1$.
\end{remark}
\subsubsection*{Numerical comparison of  $\Lambda^{(3)}_{p,n}(B_R)$ and $\Lambda^{(2)}_{p,n}(B_R)$}
\label{sec6}
Using the formulas (\ref{10eqs161}) for $\Lambda^{(3)}_{p,n}(B_R)$ and (\ref{2eq201}) for $\Lambda^{(2)}_{p,n}(B_R)$ we listed below in Table \ref{7tab:1} for $R=1$ and fixed $n\in[2,9]$ the intervals of $p$ where $\Lambda^{(3)}_{p,n}(B_1)geq\Lambda^{(2)}_{p,n}(B_1)$ and where $\Lambda^{(2)}_{p,n}(B_1)\geq\Lambda^{(3)}_{p,n}(B_1)$. Numerical calculations are made by Mathematica 6. For example, for $n=3$ and $p>4.25$ we have $\Lambda^{(3)}_{p,n}>\Lambda^{(2)}_{p,n}$, while for $p<4.25$ we have $\Lambda^{(3)}_{p,n}<\Lambda^{(2)}_{p,n}$.

\footnotesize
{\begin{table}[h]
\caption{Numerical comparison of  $\Lambda^{(3)}_{p,n}$ and $\Lambda^{(2)}_{p,n}$ for $R=1$, $n=2,\ldots, 9$ and $p>1.3$.} \label{7tab:1}
\resizebox{1\textwidth}{!} {
\begin{tabular}{@{}lllllllll@{}}
 $n$ &  $p\approx 1.32$ & $p\approx 1.33$ & $p\approx 1.35$  & $p\approx 1.38$ & $p\approx 1.43$ & $p\approx 1.64$ &$p\approx 4.25$&$p\approx 38.68$
\\
\hline
2&&&&&&&&$\Lambda^{(2)}_{p,n}$$\leftarrow | \rightarrow$$\Lambda^{(3)}_{p,n}$ \\
\hline
3&&&&&&&
$\Lambda^{(2)}_{p,n}$$\leftarrow | \rightarrow$$\Lambda^{(3)}_{p,n}$&\\
\hline
4&&&&&&
$\Lambda^{(2)}_{p,n}$$\leftarrow | \rightarrow$$\Lambda^{(3)}_{p,n}$&& \\
\hline
5&&&&&
$\Lambda^{(2)}_{p,n}$$\leftarrow | \rightarrow$$\Lambda^{(3)}_{p,n}$&&& \\
\hline
6&&&&
$\Lambda^{(2)}_{p,n}$$\leftarrow | \rightarrow$$\Lambda^{(3)}_{p,n}$&&&& \\
\hline
7&&&
$\Lambda^{(2)}_{p,n}$$\leftarrow | \rightarrow$$\Lambda^{(3)}_{p,n}$&&&&& \\
\hline
8&&
$\Lambda^{(2)}_{p,n}$$\leftarrow | \rightarrow$$\Lambda^{(3)}_{p,n}$&&&&&& \\
\hline
9&
$\Lambda^{(2)}_{p,n}$$\leftarrow | \rightarrow$$\Lambda^{(3)}_{p,n}$&&&&&&& \\
\hline
\end{tabular}}
\end{table}}

\normalsize
\subsubsection*{Comparison of $\Lambda^{(3,0)}_{p,n}(B_R)$  with numerical values}
As is mention in Sect. \ref{sect7-1} iterative numerical method for evaluating
the first eigenvalue  was developed in
\citet{BEM09, BBEM12} where the approximate
values, denoted here as  $\Lambda^{(num)}_{p,n}(B_1)$ of the first eigenvalue $\Lambda_{p,n}(B_1)$ are
given  for $p\in(1,4]$  and $n=2,3,4$.

Table \ref{7tab:2} shows that the difference between the calculated numerical values of $\lambda_{p,n}(B_1)$ and the estimates from below $\Lambda^{(3,0)}_{p,n}(B_1)$ is about 2.5 times more. Nevertheless the presented method for estimates of  $\lambda_{p,n}(B_1)$  from below using Hardy inequality with double singular weights gives analytical estimates for every $p>1$ and $n\geq2$.

{\begin{table}[h]
\caption{Numerical comparison of  $\Lambda^{(3,0)}_{p,n}$ and numerical values $\Lambda^{(num)}_{p,n}$ in Table 1, \citet{BBEM12}} \label{7tab:2}
\begin{tabular}{@{}llll@{}}
\hline
 $p$
 &\hfil $n=2$ \hfil & \hfil $n=3$ \hfil & \hfil $n=4$ \hfil
\\
\hline
\\
&\hfill $\Lambda^{(3,0)}_{p,n}$ \hfill $\Lambda^{(num)}_{p,n}$ \hfill & \hfill $\Lambda^{(3,0)}_{p,n}$ \hfill $\Lambda^{(num)}_{p,n}$ \hfill & \hfill $\Lambda^{(3,0)}_{p,n}$ \hfill $\Lambda^{(num)}_{p,n}$ \hfill
\\
\hline
$1.2$ $\quad$ &  $1.3021$ \hfil $2.9601$ &  $2.5093$ \hfil $4.5026$ &  $3.7873$ \hfil $6.0797$ \\
\hline
$1.4$ & $1.4683$ \hfil $3.6637$ & $2.8940$ \hfil $5.7188$ & $4.4860$ \hfil $7.8947$ \\
\hline
$1.6$ & $1.6063$ \hfil $4.3477$ & $3.2628$ \hfil $6.9849$ & $5.2046$ \hfil $9.8786$ \\
\hline
$1.8$ & $1.7308$ \hfil $5.0434$ & $3.6298$ \hfil $8.3443$ & $5.9574$ \hfil $12.0940$ \\
\hline
$2.0$ & $1.8472$ \hfil $5.7616$ & $4.000$ \hfil $9.8144$ & $6.7500$ \hfil $14.5735$ \\
\hline
$2.2$ & $1.9582$ \hfil $6.5071$ & $4.3755$ \hfil $11.405$ & $7.5854$ \hfil $17.3421$ \\
\hline
$2.4$ & $2.0652$ \hfil $7.2823$ & $4.7579$ \hfil $13.1232$ & $8.4658$ \hfil $20.4220$ \\
\hline
$2.6$ & $2.1621$ \hfil $8.0885$ & $5.1476$ \hfil $14.9747$ & $9.3926$ \hfil $23.8345$ \\
\hline
$2.8$ & $2.2707$ \hfil $8.9265$ & $5.5453$ \hfil $16.9646$  & $10.3672$ \hfil $27.6004$ \\
\hline
$3.0$ & $2.3703$ \hfil $9.7967$ & $5.9512$ \hfil $19.0977$ & $11.3906$ \hfil $31.7409$ \\
\hline
$3.2$ & $2.4683$ \hfil $10.6994$ & $6.3655$ \hfil $21.3785$ & $12.4639$ \hfil $36.2769$ \\
\hline
$3.4$ & $2.5648$ \hfil $11.6347$ & $6.7884$ \hfil $23.8111$ & $13.5881$ \hfil $41.2298$ \\
\hline
$3.6$ & $2.6601$ \hfil $12.6027$ & $7.2199$ \hfil $26.3977$ & $14.7642$ \hfil $46.6213$ \\
\hline
$3.8$ & $2.7543$ \hfil $13.6034$ & $7.6601$ \hfil $29.1486$ & $15.9929$ \hfil $52.4734$ \\
\hline
$4.0$ & $2.8476$ \hfil $14.6369$ & $8.1091$ \hfil $32.0618$ & $17.2752$ \hfil $58.8085$ \\
\hline
\end{tabular}
\end{table}}

\normalsize
\textbf{Acknowledgement} This paper has been accomplished with the financial support by Grant No BG05M2OP001-1.001-0003, financed by the Science and Education for Smart Growth Operational Program (2014-2020) in Bulgaria and co-financed by the European Union through the European Structural and Investment Funds and  also by  the National Scientific Program "Information and Communication Technologies for a Single Digital Market in Science, Education and Security (ICTinSES)", under contract No DO1-205/23.11.2018.

\end{document}